\documentclass[11pt,reqno]{amsart}

\usepackage{tikz}
\usetikzlibrary{arrows.meta} 
\usepackage{amsmath,amsthm,amsfonts,amssymb,latexsym,mathrsfs,color,extarrows}
\usepackage{subcaption}

\usepackage[backref=page]{hyperref}

\usepackage{hyperref}  
\hypersetup{
    colorlinks=true,
    urlcolor=black
}
\makeatletter
\def\thm@space@setup{%
  \thm@preskip=8pt \thm@postskip=8pt 
}
\makeatother

\newtheorem{Theorem}{Theorem}[section]
\newtheorem{Corollary}[Theorem]{Corollary}
\newtheorem{Proposition}[Theorem]{Proposition}

\newtheorem{Conjecture}[Theorem]{Conjecture}
\newtheorem{Lemma}[Theorem]{Lemma}

\theoremstyle{definition} 

\newtheorem{Definition}[Theorem]{Definition}
\newtheorem{Example}[Theorem]{Example}

\usepackage{subcaption}
\usetikzlibrary{arrows.meta}

\allowdisplaybreaks

\DeclareMathOperator{\cyc}{cyc}
\DeclareMathOperator{\des}{des}
\DeclareMathOperator{\asc}{asc}
\DeclareMathOperator{\exc}{exc}
\DeclareMathOperator{\drop}{drop}

\DeclareMathOperator{\inv}{inv}

\DeclareMathOperator{\RLmin}{RLmin}
\DeclareMathOperator{\fix}{fix}

\DeclareMathOperator{\maj}{maj}
\DeclareMathOperator{\ides}{ides}
\DeclareMathOperator{\imaj}{imaj}

\DeclareMathOperator{\Gen}{Gen}
\DeclareMathOperator{\Sec}{Sec}

\DeclareMathOperator{\Sin}{Sin}
\DeclareMathOperator{\Cos}{Cos}

\DeclareMathOperator{\KSO}{\bf KSO}
\DeclareMathOperator{\LPO}{\bf LPO}
\DeclareMathOperator{\AIO}{\bf AIO}
\DeclareMathOperator{\DIO}{\bf DIO}

\numberwithin{equation}{section}
\renewcommand{\theequation}{\arabic{section}.\arabic{equation}}

\def\AndI{\mathop{\rm And}\nolimits^{I}}  
\def\AndII{\mathop{\rm And}\nolimits^{I\!I}} 

\title{ $q$-derivative Grammar}

\author{Guo-Niu Han}
\address{I.R.M.A., UMR 7501, Universit\'e de Strasbourg et CNRS, 7 rue
Ren\'e Descartes, F-67084 Strasbourg, France}
\email{guoniu.han@unistra.fr}

\author{Kathy Q. Ji}
\address{ Center for Applied Mathematics and KL-AAGDM,
Tianjin University,
Tianjin 300072, P.R. China
}
\email{kathyji@tju.edu.cn}

\author{Huan Xiong}
\address{
 Institute for Advanced Study in Mathematics, 
   Harbin Institute of Technology,
   Heilongjiang 150001, P.R. China
}
\email{huan.xiong.math@gmail.com}

\makeatletter
\@namedef{subjclassname@2020}{%
	\textup{2020} Mathematics Subject Classification}
\makeatother

\date{2026/06/14}
\begin{document}
\begin{abstract}
Context-free grammars, originating in computer science, are related to enumerative combinatorics through two distinct lines of development pioneered by Sch\"utzenberger and Chen, respectively.
In the framework established by Sch\"utzenberger and Delest–Sch\"utzenberger–Viennot, unambiguous grammars are translated into functional equations for ordinary generating functions. Inspired by Rota’s umbral calculus, Chen later developed a grammatical calculus by associating each context-free grammar with a formal derivative operator.  Dumont further developed this method through numerous combinatorial interpretations of grammars with finite and infinite alphabets.   Substantial progress in this direction has been achieved over the last decade. In this paper, we introduce a $q$-analogue of grammatical calculus, which we call the {\it  $q$-derivative grammar}. We establish the basic framework of $q$-grammars and develop the $q$-grammatical calculus for computing $q$-exponential generating functions associated with $q$-grammars. Concrete $q$-grammars are constructed to study $q$-Eulerian, $q$-Roselle and $q$-Andr\'e polynomials, including their generating functions and recurrences. This work extends the grammatical method to the $q$-setting and opens up new research directions.
\end{abstract}
\maketitle

\section{Introduction}\label{sec:intro}

Context-free grammars, originating from computer science and proposed by Noam Chomsky, serve as the theoretical foundation of programming languages,  see \cite{Chomsky-1956, Chomsky-1959, Hopcroft-Ullman-1969}. 
Sch\"utzenberger~\cite{Schutzenberger-1962,Schutzenberger-1963} observed that grammars provide recursive specifications for combinatorial structures and can therefore be used to study their ordinary generating functions. This   was later developed into the Delest--Sch\"utzenberger--Viennot (DSV) methodology~\cite{Delest1996,Viennot1985}, which translates suitable unambiguous grammars into functional equations for ordinary generating functions.  Moreover, \(q\)-analogues of the DSV methodology, formulated in terms of \(q\)-grammars, were developed by Dubernard~\cite{Dubernard1993} and Duchon~\cite{Duchon1998,Duchon1999}. In 1993, Chen~\cite{Chen-1993} associates a formal derivative with each context-free grammar and extends it by linearity and the Leibniz rule. 
Using the resulting properties of the formal derivative, rigorous grammatical calculus can be performed.   Within this framework, Chen \cite{Chen-1993} derived elegant proofs of Fa\`a di Bruno's formula, along with several identities involving Bell polynomials, Stirling numbers, and symmetric functions. In particular, the Lagrange inversion formula receives a concise grammatical interpretation, from which Cayley's formula for labeled trees follows in a natural way. 
Subsequently, Dumont~\cite{Dumont-1996}  substantially extended Chen's grammatical method by providing  numerous examples of grammars with finite and infinite alphabets together with their combinatorial interpretations. 
Since then, this approach has been further developed and applied to a broad range of enumerative problems.   An overview of these developments is presented in Appendix \ref{sec:survey}.

The objective of this paper is to develop a $q$-analogue of grammatical calculus developed by Chen ~\cite{Chen-1993}, which we refer to as the {\it  $q$-derivative grammar} (or $q$-grammar for short). To the best of our knowledge, no such $q$-analogue has been constructed in the past three decades, apart from a study about the multivariable tangent and secant $q$-derivative polynomials due to Foata and Han \cite{Foata-Han-2012}. 

Constructing a $q$-analogue of combinatorial context-free grammar poses substantial challenges. Beyond introducing the additional parameter $q$ consistently   throughout all computations, this generalization fundamentally converts elementary commutative calculations into highly nontrivial non-commutative operations.  We address this non-commutative obstacle in Section \ref{sec:noncom} using the $q$-product formula for the $q$-derivative~\cite{Foata-Han-2012}.

Having reviewed the non-commutative computation in Section \ref{sec:noncom},  we proceed in Section~\ref{sec:def} to rigorously define the formal $q$-derivative grammar.  Nevertheless, we first confront a foundational notational choice. As claimed by Chen \cite{Chen-1993}, combinatorial context-free grammars borrow terminology from formal language theory, where standard definitions are universally tuple-based, see \cite{Chomsky-1956,hopcroftintroduction, sipser2012introduction}.  In contrast, combinatorial applications traditionally adopt simplified, streamlined definitions. For the rigorous formulation of our  $q$-derivative grammar, we revert to the tuple-based framework, which is better suited to our algebraic and combinatorial requirements. We next address the natural combinatorial questions motivated by  $q$-derivative grammars. For a fixed $q$-grammar, key problems include seeking concrete combinatorial interpretations, enumerating the number of terms, and deriving associated generating functions. We systematically investigate these problems and supply illustrative examples in Section~\ref{sec:def}.

In Section \ref{secqgram}, we develop a unified computational tool, called $q$-grammatical calculus, to compute the $q$-exponential generating functions associated with the corresponding $q$-grammars. As illustrations of this method, we provide grammatical derivations of the $q$-binomial inversion formula and the $q$-Hoffman formula. We then construct specialized  $q$-grammars to produce $q$-analogs of the Eulerian polynomials, the Roselle polynomials, and two $q$-analogs of the Andr\'e polynomials in Sections \ref{sec:qgrammars} and \ref{sec:qAndre}, respectively. Based on these $q$-grammars, we obtain grammatical derivations of the $q$-exponential generating functions for the $q$-Eulerian polynomials (due to Stanley) and the cycle $q$-Roselle polynomials. Recurrence relations for the two $q$-Andr\'e polynomials are also derived using the $q$-grammatical calculus developed in 
Section~\ref{secqgram}.

Lastly,  we believe that most of the topics discussed in Appendix \ref{sec:survey} within the framework of context-free grammars can be extended to the $q$-analogue by employing  $q$-derivative grammars. Consequently, this line of research offers many promising directions for further investigation.

\section{\texorpdfstring{$q$-derivative}{q-derivative}}\label{sec:noncom}

Let $\mathbb{K}$ be a commutative ring with unity and characteristic zero.
For $f(u) \in \mathbb{K}[[u]]$, the {\it $q$-derivative operator} \cite[p.~22]{GR90} used in this paper is defined as
\begin{equation}\label{defi:qderiv}
D_qf(u):=\frac{f(u)-f(qu)}{u},
\end{equation}
which differs from the conventional form:
\[
\frac{f(qu) - f(u)}{(q-1)u}.
\]
For $n\geq 1$, we recursively define $D^n_q(u)=D_q(D^{n-1}_q(u))$.

For $f(u), g(u) \in \mathbb{K}[[u]]$, we have the product rule
\[
D_q(f(u)g(u)) = D_q(f(u))g(uq) + f(u)D_q(g(u)),
\]
In general, for $f_i(u)\in \mathbb{K}[[u]]$ where $1\leq i\leq n$, the following result holds.
\begin{Proposition} \cite[(4.2)]{Foata-Han-2012}\label{th:qLeibniz}
We have
\begin{align}\label{eq:qprod}
&D_q \left(\prod_{1 \le i \le n} f_i(u)\right) \nonumber\\
&\quad = \sum_{1 \le i \le n}
 f_1(u) \cdots f_{i-1}(u) \cdot D_q(f_i(u)) \cdot
 f_{i+1}(qu) \cdots f_n(qu).
\end{align}
\end{Proposition}

A straightforward computation yields the following useful identity:
\begin{equation}\label{eq:Dq:qm}
D_q(f(q^m u)) = q^m \cdot D_q(f(x))\big|_{x=q^m u}.
\end{equation}

The $q$-analogues of the exponential function $e^x$, first introduced by Euler~\cite{Eu48}, are given by (see \cite[p.~9]{GR90}):
\begin{align}
 e_q(u)
 &= \sum_{n\ge 0} \frac{u^n}{(q;q)_n}
  = \frac{1}{(u;q)_\infty}, \label{defi:expfor}\\
 E_q(u)
 &= \sum_{n\ge 0} q^{n \choose 2}
    \frac{u^n}{(q;q)_n}
  = (-u;q)_\infty, \label{defi:Expfor}
\end{align}
where the $q$-shifted factorial is defined by
\[
\begin{aligned}
(u;q)_n &:=
  \begin{cases}
    1, & \text{if } n = 0, \\
    (1-u)(1-uq) \cdots (1-uq^{n-1}), & \text{if } n \ge 1,
  \end{cases} \\
(u;q)_\infty &:= \lim_{n \to \infty} (u;q)_n
  = \prod_{n \ge 0} (1 - uq^n).
\end{aligned}
\]

They both serve to define the $q$-trigonometric functions
\cite{Ja04}:
\begin{align*}
\sin_q(u)
  &:= \frac{e_q(iu) - e_q(-iu)}{2i},   \quad
\cos_q(u)
  := \frac{e_q(iu) + e_q(-iu)}{2}, \\
\Sin_q(u)
  &:= \frac{E_q(iu) - E_q(-iu)}{2i}, \quad 
\Cos_q(u)
  := \frac{E_q(iu) + E_q(-iu)}{2}, \\
\sec_q(u) &:= \frac{1}{\cos_q(u)}, \quad
\Sec_q(u) := \frac{1}{\Cos_q(u)},\quad
\tan_q(u) := \frac{\sin_q(u)}{\cos_q(u)}.
\end{align*}

The following $q$-derivative formulas for the $q$-trigonometric functions are given by Foata and Han \cite{Foata-Han-2012}: 
\begin{Proposition}[\cite{Foata-Han-2012}]\label{th:Dtanq} We have 
\[
\begin{aligned}
D_q(\tan_q(u)) &= 1 + \tan_q(u)\tan_q(qu), \\[4pt]
D_q(\sec_q(u)) &= \sec_q(qu)\,\tan_q(u), \\[4pt]
D_q(\Sec_q(u)) &= \Sec_q(u)\tan_q(qu).
\end{aligned}
\]
\end{Proposition}

In the $q$-product formula \eqref{eq:qprod},  although the left-hand side of the $q$-product formula is symmetric in the functions $f_i$, this symmetry is not explicit in the right-hand side, which depends on the ordering of the functions. We illustrate this subtlety with an example. From the identity
$$D_q(\tan_q(u)) = 1 + \tan_q(u)\tan_q(qu) = 1 + \tan_q(qu)\tan_q(u),$$
there are several ways to compute $D_q^2\tan_q(u)$ via the $q$-product formula. One such method is the following:
\smallskip
\begin{align*}
D_q^2\tan_q(u)
  &= D_q \bigl(1 + \tan_q(u)\tan_q(qu)\bigr) \\
  &= \tan_q(u)\,D_q(\tan_q(qu)) + D_q(\tan_q(u))\,\tan_q(q^2u) \\
  &= \tan_q(u)\,q\bigl(1 + \tan_q(qu)\tan_q(q^2u)\bigr) \\
  & \qquad + \bigl(1 + \tan_q(u)\tan_q(qu)\bigr)\tan_q(q^2u) \\
  &= q\tan_q(u) + \tan_q(q^2u) + (1+q)\tan_q(u)\tan_q(qu)\tan_q(q^2u).
\end{align*}
To simplify notation, we introduce the shorthand
$$
x_j := \tan_q(q^j u).
$$
Using this substitution, the first identity in Proposition~\ref{th:Dtanq} combined with \eqref{eq:Dq:qm} becomes
\begin{equation}\label{eq:Dtan:subs}
D_q(x_j) = q^j \bigl(1+ x_j x_{j+1}\bigr).
\end{equation}
The first computation can then be rewritten compactly as
\begin{align*}
D_q^2 (x_0)
  &= D_q \bigl(1 + x_0 x_1\bigr) \\
  &= x_0 D_q(x_1) + D_q(x_0) x_2 \\
  &= x_0 q\bigl(1 + x_1 x_2 \bigr)  + \bigl(1 + x_0 x_1 \bigr) x_2  \\
  &= q x_0  + x_2  + (1+q)  x_0 x_1 x_2.
\end{align*}

The second natural derivation gives
\begin{align*}
D_q^2 (x_0)
  &= D_q\bigl(1 + x_1 x_0 \bigr) \\
  &= x_1 D_q(x_0) + D_q(x_1) x_1 \\
  &= x_1\bigl(1 + x_0 x_1 \bigr)
     + q\bigl(1 + x_1 x_2 \bigr) x_1 \\
  &= (1+q) x_1  +  x_0 x_1^2  + q x_1^2 x_2.
\end{align*}

Notably, both derivations depend only on the identity \eqref{eq:Dtan:subs} and the $q$-product formula, without using any special properties of the function $\tan_q$ itself. This observation motivates the definition of our  $q$-derivative grammar, which will be presented in Section~\ref{sec:def}.

\section{\texorpdfstring{Foundations of $q$-derivative grammar}{Foundations of formal $q$-derivative grammar}}\label{sec:def}

As noted by Chen \cite{Chen-1993}, the terminology of context-free grammars used in Combinatorics is inspired by that of context-free grammars in formal language theory. In formal language theory, one consistently uses {\it tuple-based definitions}, see \cite{Chomsky-1956,hopcroftintroduction, sipser2012introduction}.
In contrast, in Combinatorics, one usually prefers a much simpler way of stating definitions. 

{For our definition of a  $q$-derivative grammar, we return to the tuple-based framework, which is more convenient for our purposes.
We therefore begin by introducing some fundamental concepts that will serve as preparatory material for our tuple-based definition of the  $q$-derivative grammar.}

\subsection{Free group and group algebra} 

Free groups and group algebras are two fundamental concepts in algebra, playing important roles in group theory and representation theory.
The following definitions are well-known and can be found in \cite{ johnson1980topics, serre1977linear}.

\begin{Definition}[Free Group]\label{def: free_group}
    Let $S$ be a set. Consider the alphabet $S \cup S^{-1}$, where $S^{-1} = \{ s^{-1} \mid s \in S \}$ is a set of formal inverses. A \textit{word} is a finite sequence of elements from $S \cup S^{-1}$. A word is called \textit{reduced} if it contains no adjacent pair of the form $s s^{-1}$ or $s^{-1} s$ for any $s \in S$. The \textit{free group} on $S$, denoted $F(S)$, consists of all reduced words. The group operation is concatenation followed by reduction (removing adjacent inverse pairs). The empty word is the identity element, and the inverse of a word $x_1 x_2 \cdots x_n$ is $x_n^{-1} \cdots x_2^{-1} x_1^{-1}$.
\end{Definition}

\begin{Definition}\label{def: group_algebra}
    Let $\mathbb{K}$ be a commutative ring with unity and characteristic zero and $G$ be a group.
    The \textit{group algebra} $\mathbb{K}[G]$ is the set of all formal finite sums
    \[
    \sum_{g \in G} \alpha_g \, g, \quad \alpha_g \in \mathbb{K},
    \]
    where only finitely many coefficients $\alpha_g$ are non-zero. Addition is defined component-wise:
    \[
    \sum_{g} \alpha_g g + \sum_{g} \beta_g g = \sum_{g} (\alpha_g + \beta_g) g.
    \]
    Multiplication is defined by extending the group multiplication linearly:
    \[
    \left( \sum_{g \in G} \alpha_g g \right) \left( \sum_{h \in G} \beta_h h \right)
    = \sum_{g,h \in G} (\alpha_g \beta_h) (gh).
    \]
    This makes $\mathbb{K}[G]$ an associative $\mathbb{K}$-algebra with identity $1_{\mathbb{K}} \cdot e_G$, where $e_G$ is the identity element of $G$.
\end{Definition}

\subsection{Basic algebraic structures}
Let $q$ be an indeterminate, let $\mathbb{K}$ be a commutative ring with unity and characteristic zero, and let $S$ be a (finite or infinite) set of symbols, called \emph{master variables}.  
To each master variable $s \in S$ we associate an infinite family of indexed non-commutative variables
\[
s_0, s_1, s_2, \ldots
\]
The total set of variables is denoted by $\mathbb{S} := S^{\mathbb{N}}:=\{ s_i: s\in S, ~~i\in \mathbb{N}\}$, where $\mathbb{N} = \{0,1,2,\ldots\}$.

The fundamental algebraic object in the formal $q$-grammar framework is the \emph{group algebra $\mathbb{K}[q][F(\mathbb{S})]$ on $F(\mathbb{S})$}, denoted by $\mathbb{E}$. Therefore, an element of $\mathbb{E}$, called an \emph{expression}, is a (finite) $\mathbb{K}[q]$-linear combination of words in the variables from the free group $F(\mathbb{S})$, that is,
\[
E = \sum_{w \in F(\mathbb{S})} a_w\, w,
\]
where each coefficient $a_w$ lies in $\mathbb{K}[q]$. The sum is finite in the sense that $a_w = 0$ for all but finitely many words $w \in F(\mathbb{S})$. The scalar $a_w$ is referred to as the \emph{coefficient} of the word $w$ in the expression $E$.

\subsection{Rules}
A {\it rule} $R\colon \mathbb{S}\cup \mathbb{S}^{-1}  \rightarrow \mathbb{E}$ is a map that assigns to each variable in $\mathbb{S}\cup \mathbb{S}^{-1}$   an expression in $\mathbb{E}$. More precisely, the rule is first defined on the set of variables in $\mathbb{S}$,  where it is typically written as 
$$
\{ s_0 \mapsto R(s_0), \
s_1 \mapsto R(s_1), \ \ldots
\}.
$$
The definition is then extended to the set of formal inverses $s_i^{-1}$  by setting, for each  $s_i\in \mathbb{S}$,
$$R(s_i^{-1}):=-s_i^{-1}R(s_i)s_{i+1}^{-1}.$$

For instance, as will be discussed later, the assignment
$$
R = \{x_j \mapsto q^j(1 + x_j x_{j+1})\}
$$
constitutes a rule.
The purpose of a rule is to transform a given expression into a new expression. 

\subsection{Orders}
An {\it order} is a map that rewrites a word in $F(\mathbb{S})$ by permuting the non-commutative variables in $\mathbb{S}$. This order is extended to arbitrary expressions in $\mathbb{E}:=\mathbb{K}[q][F(\mathbb{S})]$ by linearity. In this paper, the following four orders are frequently employed.

\begin{Definition}
For a set of master variables $S=\{x,y,\ldots\}$, we define the following special orders based on priority rules.
\begin{itemize}
    \item \textbf{$\KSO$ (Keep Sequence Order)}:
    This is the identity map, which leaves the word unchanged.

    \item \textbf{$\LPO$ {\rm(}Letter-Priority Order{\rm)}}:
    This rewrites the word by prioritizing all $x$ variables before all $y$ variables, preserving their internal index order, i.e., reordering variables associated with  the order
    $$
    x_0, x_1, x_2, \ldots,\quad y_0, y_1, y_2, \ldots
    $$

    \item \textbf{$\AIO$ {\rm(}Ascending Interleaving Order{\rm)}}:
    This rewrites the word by interleaving $x$ and $y$ variables in ascending index order, i.e., reordering variables associated with  the order
    $$
    x_0, y_0, x_1, y_1, x_2, y_2, \ldots
    $$

    \item \textbf{$\DIO$ {\rm(}Descending Interleaving Order{\rm)}}:
    This rewrites the word by interleaving $x$ and $y$ variables in descending index order, i.e., reordering variables associated with  the order
    $$
    \ldots, x_2, y_2, x_1, y_1, x_0, y_0.
    $$
\end{itemize}
\end{Definition}

For example, take $S=\{x,y\}$ and let $E=x_2 y_2^{-1} x_1 x_3^2 y_3 + (1+q) y_1^2 x_1^{-1} x_2 \in \mathbb{E}$. Then
\begin{align*}
{\KSO}  (E) &= x_2 y_2^{-1} x_1 x_3^2 y_3 + (1+q) y_1^2 x_1^{-1} x_2 = E, \\
{\LPO}  (E) &= x_1 x_2 x_3^2 y_2^{-1} y_3 + (1+q)  x_1^{-1} x_2 y_1^2, \\
{\AIO}  (E) &= x_1  x_2 y_2^{-1} x_ 3^{2}y_3 + (1+q) x_1^{-1} y_1^2 x_2, \\
{\DIO}  (E) &= x_3^{2}y_3 x_2 y_2^{-1} x_1   + (1+q)  x_2  x_1^{-1} y_1^2. 
\end{align*}
\subsection{\texorpdfstring{Definition of  $q$-derivative grammar}{Definition of formal q-derivative grammar}}
We now introduce the main definition of this work.
\begin{Definition}
A formal {\it $q$-derivative grammar} {\rm(}or {\it $q$-grammar} for short{\rm )} $G$ over a commutative ring $\mathbb{K}$   with unity and characteristic zero is a triple
    $$G = (S, R, \rho),$$
    where
    \begin{itemize}
        \item $S$ denotes a finite or infinite set of {\it master variables},
        \item $R$ is a rule,
        \item $\rho$ is an order.
    \end{itemize}
\end{Definition}

For instance, the following triples define two formal \(q\)-derivative grammars:
\begin{equation}\label{def:Gtan}
G_{\tan} = \bigl(S = \{x\},\; R = \{x_j \rightarrow q^j (1 + x_j x_{j+1})\},\; \rho = \DIO \bigr),
\end{equation}
and
\begin{equation}\label{def:Gtan2}
G_{\tan'} = \bigl(S = \{x\},\; R = \{x_j \rightarrow q^j (1 + x_j x_{j+1})\},\; \rho = \LPO \bigr).
\end{equation}

\subsection{\texorpdfstring{$q$-derivative operator $D$}{q-derivative operator D}}
The goal of our framework for  $q$-derivative grammars is to perform certain computations. We begin by defining the up-arrow operator.
\begin{Definition}
The {\it up-arrow operator} $\uparrow: \mathbb{E} \rightarrow \mathbb{E}$ is a linear operator that acts on words over $F(\mathbb{S})$ by increasing the index of each variable in $\mathbb{S}\cup \mathbb{S}^{-1}$ by~$1$. In general, for a positive integer $k$, $\uparrow^k$ denotes the operator that increases the index of each variable in $\mathbb{S}\cup \mathbb{S}^{-1}$ by $k$.
\end{Definition}

For example, if $S=\{x,y\}$, then
\[
\uparrow\!\bigl(y_2^{-1}x_0x_0y_1\bigr) = y_3^{-1}x_1x_1y_2
\quad\text{and}\quad
\uparrow^3\!\bigl(y_2^{-1}x_0x_0y_1\bigr) = y_5^{-1}x_3x_3y_4.
\]

\begin{Definition}\label{defi:q-derivative}
	Let $G=(S, R, \rho)$ be a $q$-derivative grammar. The \emph{$q$-derivative operator} $D\colon \mathbb{E} \rightarrow \mathbb{E}$ associated with $G$ is the linear operator defined by
	$$
	D(w_1w_2\cdots w_n) :=  \sum_{j=1}^n \rho \Bigl( w_1w_2\cdots w_{j-1}\, R(w_j)\, \uparrow\! \bigl(w_{j+1}\cdots w_n\bigr) \Bigr),
	$$
	for each word $w_1w_2\cdots w_n$, where each $w_j \in \mathbb{S}\cup \mathbb{S}^{-1}$.
\end{Definition}

Note that the expression $\uparrow\! w_{j+1}$ is not to be interpreted as $w_{j+2}$. For instance, if $w_{j+1} = s_3$, then
$\uparrow\! w_{j+1} = \uparrow\! s_3 = s_4$.

Note that {$q$-derivative operator} is well-defined, since 
$$
D(x_ix_i^{-1})
=\rho\left(R(x_i)\cdot x_{i+1}^{-1}  -x_i\cdot x_i^{-1}R(x_i)x_{i+1}^{-1})\right)=0
$$
and
$$
D(x_i^{-1}x_i)
=\rho\left(-x_i^{-1}R(x_i)x_{i+1}^{-1}\cdot x_{i+1}  +x_i^{-1}\cdot R(x_i)\right)=0.
$$

\begin{Definition} 
	Let $G=( S, R, \rho)$
	be a  $q$-grammar.
    The high order $q$-derivative operator $D^k$ associated with $G$ is defined by 
	$ D^0(f) = f$,	and for $k\geq 1$, $ D^k(f) = D(D^{k-1}(f))$ for $f \in  \mathbb{E}$.
\end{Definition}

For example, take the grammar $G_{\tan}$ defined in  \eqref{def:Gtan}, we have
\begin{align}\label{eq:D2tan}
D(x_0) &= \DIO(1+x_0 x_1) = 1+x_1 x_0, \nonumber\\
D^2(x_0) &= \DIO(R(x_1) x_1+x_1 R(x_0)) \nonumber\\
 &= \DIO(q(1+x_1x_2) x_1+x_1  (1+x_0x_1)) \nonumber\\
 &=  (1+q)x_1 +   x_1 x_ 1x_0 + qx_2 x_1x_1. 
\end{align}
For a second example, we take the grammar $G_{\tan'}$ defined in  \eqref{def:Gtan2}. We have
\begin{align}\label{eq:D2tan2}
D(x_0) &= \LPO(1+x_0 x_1) = 1+x_0 x_1, \nonumber \\
D^2(x_0) &= \LPO(R(x_0) x_2+x_0 R(x_1)) \nonumber \\
 &= \LPO((1+x_0x_1) x_2+x_0 q (1+x_1x_2)) \nonumber \\
 &=  qx_0 + x_2+ (1+q) x_0x_1 x_2. 
\end{align}

Let $G$ be a $q$-grammar and let $f$ be an expression in $\mathbb{E}$. The high order $q$-derivative associated with $G$ can be written in the following form:
\begin{equation*}
D^k(f) =  \sum_{w \in F(\mathbb{S})} a_w\, w.
\end{equation*}
In general, the computation of higher-order \(q\)-derivatives associated with the \(q\)-grammars is lengthy 
and technically involved.
To facilitate these computations, we provide a {\tt SageMath} package, {\tt qgrammar.sage}, which is accessible on the first author's personal webpage at
$$\hbox{\tt https://irma.math.unistra.fr/\char126guoniu/qgrammar.html}.$$
On this webpage, we also compile the initial higher-order $q$-derivatives associated with the $q$-grammars investigated in this paper. 

\subsection{\texorpdfstring{Evaluation on $q$-grammars}{Evaluation on q-grammars}}
Let $G=(S, R, \rho)$ be a $q$-grammar,
and
let $\mathbb{V}$ be a commutative ring.  
Recall that $\mathbb{S}$ is the set of all non-commutative variables and $\mathbb{E}$ is the free $\mathbb{K}[q]$-algebra generated by $\mathbb{S}$.
An {\it evaluation} $\phi$ on a $q$-grammar
	is a map $\phi \colon \mathbb{S} \rightarrow \mathbb{V}$ that sends each non-commutative variable to a commutative element of $\mathbb{V}$.
Such an evaluation is typically expressed in the form:
$$
\{ x_0 \mapsto \phi(x_0), \
x_1 \mapsto \phi(x_1), \ \ldots
\}.
$$
The evaluation  map $\phi$ can be uniquely extended to a morphism  $\phi \colon \mathbb{E}  \rightarrow \mathbb{V}$ by  preserving linearity, multiplicativity and  invertibility (i.e., $\phi(x_0^{-1}):=\phi(x_0)^{-1}$).

For instance, take $\mathbb{V}=\mathbb{Z}[[u,q]]$, the assignment
$$
\phi = \{x_j \mapsto \tan_q(q^j u)\}
$$
constitutes an evaluation. Applying this evaluation to the expression \eqref{eq:D2tan2} for   the grammar $G_{\tan'}$,  we get 
$$
\phi(D^2(x_0))  
= q \tan_q(u) + \tan_q(q^2u) + (1+q) \tan_q(u)\tan_q(qu) \tan_q(q^2u) \in 
\mathbb{Z}[[u,q]].
$$

\subsection{\texorpdfstring{Research problems on $q$-grammars} {Research problems on q-grammars}}\label{sec:3.8}
Let $G=( S, R, \rho)$ be a $q$-grammar and let $f$ be an expression in $\mathbb{E}$. The high order $q$-derivative associated with $G$ on $f$ can be written in the following form:
\begin{equation}\label{Dk:gen}
D^k(f) =  \sum_{w \in F(\mathbb{S})} a_w\, w.
\end{equation}
The $q$-grammar framework gives rise to several interesting combinatorial problems, which we formulate below:
\begin{itemize}
\item \textit{Terms}: Characterize the words $w \in F(\mathbb{S})$ appearing in \eqref{Dk:gen} for which $a_w \neq 0$. These words are called \textit{terms} of the $q$-derivative.
    \item \textit{Number of terms}: Determine the cardinality of the set of terms.

\item \textit{Coefficients}: Obtain explicit formulas for the coefficients $a_w$.

\item \textit{Combinatorial model}: Provide a combinatorial interpretation for the coefficients $a_w$.

\item \textit{Generating functions}: 
\begin{enumerate}
    \item Determine an appropriate evaluation map $\phi$ such that $\phi(D^k(f))$ admits a simple closed form;
    \item Develop a general method, which we refer to as \emph{$q$-grammatical calculus}, to compute the   $q$-exponential generating functions for $\phi(D^k(f))$. 
\end{enumerate} 
\end{itemize}

\subsection{A canonical example}
Take the grammar $G_{\tan\cup \sec}$ defined in  \eqref{def:Gtan}:   
\begin{equation} \label{defi:gramartansec}
G_{\tan\cup \sec} = \bigl(\{x, y\},\; \{x_j \rightarrow q^j (1 + x_j x_{j+1}),   \; y_j\rightarrow q^j x_j y_{j+1}\},\; \DIO \bigr).
\end{equation}

We have

\begin{align*}
D^0(x_0) &= x_0, \\
D^1(x_0) &= 1+x_1 x_0, \\
D^2(x_0) &= (1+q)x_1 +   x_1^2 x_0 + qx_2 x_1^2, \\
D^3(x_0) &=
(q+ q^2)
+(1+ q )  x_1^2 
+ x_1^3 x_0 
+( q^2 + q^3 )  x_2^2 \\
&\quad +(q+ q^2  )  x_2^2  x_1^2 
+ q   x_2 x_1^3 
+( 2q^2 + 2q )  x_2 x_1
+ q^2   x_2^3 x_1
+ q^3   x_3 x_2^3, \\[5pt]
D^0(y_0) &= y_0 \\ 
D^1(y_0) &= y_1 x_0 \\ 
D^2(y_0) &= q y_2 x_1^2+ y_1+ x_1 y_1 x_0 \\ 
D^3(y_0) &= q^3 y_3 x_2^3 + q^2\, x_2 y_2+ q^2\, x_2^2 y_2 x_1 + (q^2 + 2q)\, y_2 x_1 \\
&\quad + (q^2 + q)\, x_2 y_2 x_1^2  + q\, y_2 x_1^3 +  x_1 y_1 + x_1^2y_1 x_0.  
\end{align*}

Let $D$ be the formal $q$-derivative associated with $G_{\tan\cup \sec}$ given by \eqref{defi:gramartansec} and define 
\begin{equation}\label{eq:canonicalexp1}
D^{\,n} (x_0)=\sum_{w \in F(\mathbb{S})} a_w\, w.
\end{equation}
We have the following two consequences concerning the terms and the number of terms in \eqref{eq:canonicalexp1}. However, no explicit closed-form expression for the coefficients $a_w$ has yet been established.

\begin{Proposition}\label{lem: number terms02}
  For $n \geq 3$, every term in \eqref{eq:canonicalexp1} takes exactly one of the following three forms:
\begin{enumerate}
    \item[(i)] $x_n x_{n-1}^{\,n}$;
    \item[(ii)] $x_1^{\,n} x_{0}$;
    \item[(iii)] $x_{j+1}^{\,a} x_{j}^{\,b}$, 
    where  $a,b \leq n$, $1 \leq j \leq n-2$, 
    $a+b \leq n+1$, and $a+b+n+1$ is even. 
\end{enumerate}
\end{Proposition}

\begin{proof}
We prove by induction on $n$.
For $n=3$, a direct computation verifies that the statement holds.
Assume the proposition holds for $n-1$ (with $n-1 \geq 3$). That is, every term in $D^{\,n-1}(x_0)$ is of one of the forms (i)–(iii) with $n$ replaced by $n-1$. Applying the operator $D$ to such a term and using the grammar rule $D(x_j) = q^j (1 + x_j x_{j+1})$, we obtain new terms. By a straightforward case analysis, one can verify that each resulting term again fits one of the forms (i)–(iii) for the index $n$.

Conversely, to show that every term of the forms (i)–(iii) for index $n$ indeed appears in $D^{\,n} (x_0)$, we construct a preimage in $D^{\,n-1} (x_0)$ for each case.

\begin{itemize}
    \item \textbf{Case (i):} $x_n x_{n-1}^{\,n}$.  
    Consider the term $x_{n-1} x_{n-2}^{\,n-1}$ in $D^{\,n-1} (x_0)$. Applying $D$ yields that $D\bigl(x_{n-1} x_{n-2}^{\,n-1}\bigr)$ contains
    \[
     D(x_{n-1})\, x_{n-1}^{\,n-1}
    = q^{\,n-1} (1 + x_n x_{n-1})\, x_{n-1}^{\,n-1}
    = q^{\,n-1} x_{n-1}^{\,n-1} + q^{\,n-1} x_n x_{n-1}^{\,n}.
    \]
    The second term (up to the scalar factor $q^{\,n-1}$) is exactly $x_n x_{n-1}^{\,n}$.

    \item \textbf{Case (ii):} $x_1^{\,n} x_{0}$.  
    Take the term $x_1^{\,n-1} x_{0}$ in $D^{\,n-1} (x_0)$. Then we obtain that $D\bigl(x_1^{\,n-1} x_{0}\bigr)$ contains
    \[
     x_1^{\,n-1} D(x_{0})
    = x_1^{\,n-1} (1 + x_1 x_{0})
    = x_1^{\,n-1} + x_1^{\,n} x_{0},
    \]
    where the second summand gives the desired term.

    \item \textbf{Case (iii):} $x_{j+1}^{\,a} x_{j}^{\,b}$ with $1 \leq j \leq n-2$, $a+b \leq n+1$, $2 \mid (a+b+n+1)$, and $a,b \leq n$.  
    We split into subcases according to the parity and size conditions. In each subcase, we exhibit a suitable term in $D^{\,n-1} (x_0)$ whose $D$-image contains $x_{j+1}^{\,a} x_{j}^{\,b}$.
    \begin{enumerate}
        \item If $a+b = n+1$ and $a \ge 2$, then $x_{j+1}^{\,a} x_{j}^{\,b}$ appears as a term in $D\bigl(x_{j+1}^{\,a-1} x_{j}^{\,b}\bigr)$.
        \item If $a+b = n+1$ and $a = 1$, then $x_{j+1}^{\,1} x_{j}^{\,b}$ can be obtained from $D\bigl(x_{j}^{\,n-1} x_{j-1}\bigr)$ (note that $b = n$ in this situation).
        \item If $a+b < n+1$ and $b \neq n-1$, then $x_{j+1}^{\,a} x_{j}^{\,b}$ is a term in $D\bigl(x_{j+1}^{\,a} x_{j}^{\,b+1}\bigr)$.
        \item If $a+b < n+1$ and $b = n-1$, then $x_{j+1}^{\,a} x_{j}^{\,n-1}$ appears in $D\bigl(x_{j} x_{j-1}^{\,b}\bigr)$ (with $j$ replaced appropriately).
    \end{enumerate}
    In each subcase, the chosen preimage lies in $D^{\,n-1} (x_0)$ by the induction hypothesis, and a direct application of $D$ verifies the claim.
\end{itemize}

The combination of the forward and backward arguments completes the induction, establishing the proposition.
\end{proof}

The following proposition determines the number of terms in \eqref{eq:canonicalexp1}. Here and in the sequel, we define $\Omega(D^k(f))$ to denote the total number of terms of the expression in $D^k(f)$, where the term-counting operator $\Omega\colon \mathbb{E} \to \mathbb{N}$ returns the cardinality of the term set of a given polynomial. 
For the examples presented at the beginning of this subsection, we have 
$\Omega(D^{2}(x_{0})) = 3$
and
$ \Omega(D^{3}(x_{0})) = 9$.
The general expressions are stated in the following proposition.

\begin{Proposition}\label{lem: number terms03} Let $D$ be the formal $q$-derivative associated with $G_{\tan\cup \sec}$ given by \eqref{defi:gramartansec}. For $n \ge 3$,
\[
\Omega(D^n(x_0))= 
\begin{cases}
2k^3 + 5k^2 +2, & \text{if } n = 2k+1,\\[4pt]  
(k+2)(2k^2-2k+1),        & \text{if } n = 2k.
\end{cases}
\]
\end{Proposition}

\begin{proof}
We treat the two parity cases separately.

\noindent\textbf{Case 1: $n = 2k+1$.}  
From Proposition~\ref{lem: number terms02}, we have:
\begin{itemize}
    \item One term of type (i): $x_n x_{n-1}^{\,n}$.
    \item One term of type (ii): $x_1^{\,n} x_0$.
    \item Type (iii) contributions: For  $1\leq j \leq  n-2 = 2k-1$,  the pairs $(a,b)$ satisfies $a+b \le n+1 = 2k+2$, $a,b \le n = 2k+1$, and $2 \mid (a+b+n+1) = a+b+2k+2$. The parity condition simplifies to $a+b$ being even. Let $a+b=2d$. Then $0\leq d\leq k+1$.
    A direct enumeration yields that the number of such terms equals  $2d(2k-1)+1$ when  $0\leq d\leq k$; and equals  $(2k-1)(2k+1)$ when  $d= k+1$. Then the total number of such terms equals
    \[
    (2k-1)(2k+1)+\sum_{d=0}^{k} (2d(2k-1)+1) =
     k+1 + (2k-1)(k^2+3k+1).
    \] 
\end{itemize}

    Adding the two special terms yields the total
    \[
    \Omega(D^{2k+1}(x_0))= k+3 + (2k-1)(k^2+3k+1).
    \]

\noindent\textbf{Case 2: $n = 2k$.}  
A similar counting argument, using the same lemma, gives
\[
\Omega(D^{2k}(x_0)) = (k+2)(2k^2-2k+1).
\]
The detailed enumeration is analogous to the odd case and is therefore omitted.
\end{proof}

Concerning the combinatorial interpretation for the coefficient $a_w$ in \eqref{eq:canonicalexp1}, we recast Theorem 1.1 of \cite{Foata-Han-2012} within the framework of our $q$-grammar terminology. In \cite{Foata-Han-2012}, the authors introduce the notion of {\it $t$-permutations} and define two associated statistics, denoted ``$\ides$" and ``$\imaj$". For the sake of preserving the continuity of our exposition, we do not reproduce these definitions here.
For each triple $(k,a,b)$, let $\mathcal{T}_{n,k,a,b}$ denote the set of all $t$-permutations $w$ for which $\ides w = k$, subject to a number of additional structural constraints and let   $\mathcal{T}^{-}_{n,k,a,b+1}$ denote the set of all $t$-permutations $w=(w_0,w_1,\ldots, w_m, w_{m+1})$  in $\mathcal{T}_{n,k,a,b}$ such that $w_{m+1}=\epsilon$. The precise formulation of these constraints is given in \cite[Subsections 1.3--1.5]{Foata-Han-2012}.

\begin{Theorem}\cite[Theorem 1.1]{Foata-Han-2012} \label{thm:Foata-Han}
Let
$$
A_{n,k,a,b}(q)=\sum_{w\in \mathcal{T}_{n,k,a,b}}
q^{\imaj w}
$$
and 
$$
B_{n,k,a,b}(q)=\sum_{w\in \mathcal{T}^{-1}_{n,k,a,b+1}}
q^{\imaj w}.
$$
Then
$$
D^n(x_0)=\sum_{k,a,b}A_{n,k,a,b}(q)
x_{k+1}^b\,x_k^a,
$$
and 
$$
D^n(y_0)=\sum_{k,a,b}B_{n,k,a,b}(q)\,y_{k+1}\,
x_{k+1}^b\,x_k^a,
$$
where  $0\le k\le n-1$ and $0\le a+b\le n+1$.
\end{Theorem}

Using Theorem \ref{thm:Foata-Han}, Foata and Han \cite{Foata-Han-2012} derived the following $q$-exponential generating functions.

\begin{Theorem}\cite[Theorem 1.5]{Foata-Han-2012} \label{thm:Foata-Hanaa}
Let \(\phi = \{ x_j \mapsto x, \, y_j \mapsto y\}\) denote the evaluation map that assigns to each variable \(x_j\) the variable \(x\) and each variable \(y_j\) the variable \(y\). We have 
\begin{align} \label{eq:gf:tan}
\sum_{n\geq 0} \phi\bigl(D^n(x_0)\bigr)\,\frac{u^n}{(q;q)_n} &=\;
\tan_q(u) \;+\; \frac{x\,\sec_q(u)\Sec_q(u)}{1 - x\,\tan_q(u)},\\[5pt]
\sum_{n\geq 0} \phi\bigl(D^n(y_0)\bigr)\,\frac{u^n}{(q;q)_n} &=\;
  \frac{y\sec_q(u)}{1 - x\,\tan_q(u)}. \label{eq:gf:sec}
\end{align}
\end{Theorem}
These two identities are  \(q\)-analogues of Hoffman's formula \cite{Hoffman1995}, which is commonly written in the form:
\begin{align*}
\sum_{n\geq 0} A_n(x)\,\frac{u^n}{n!}  &= 
\frac{x + \tan(u)}{1 - x\,\tan(u)},\\[5pt]
\sum_{n \geq 0} B_n(x) \frac{u^n}{n!} &= \frac{1}{\cos u - x \sin u}.
\end{align*}

In Subsection \ref{subsec:application}, we present a grammatical derivation of Theorem \ref{thm:Foata-Hanaa} using the $q$-grammatical calculus developed in Section \ref{secqgram}. As shown in Subsection \ref{subsec:application}, this approach relies solely on the definition of the $q$-grammar $G_{\tan\cup \sec}$ given by \eqref{defi:gramartansec} and the evaluation map $\phi$ defined in Theorem \ref{thm:Foata-Hanaa}.

\subsection{Further examples}
By analogy with the example discussed in the preceding subsection, the following \(q\)-grammars are investigated in \cite{Foata-Han-2012} within the framework of \(q\)-trigonometric functions.

\begin{align*}
G_{\tan} &= \bigl( \{x\},\;  \{x_j \rightarrow q^j (1 + x_j x_{j+1})\},\;  \DIO \bigr); \\
G_{\tan'} &= \bigl(\{x\},\;  \{x_j \rightarrow q^j (1 + x_j x_{j+1})\},\;  \LPO \bigr); \\
G_{\sec} &= \bigl( \{x,y\},\;  \{x_j \rightarrow q^j (1 + x_j x_{j+1}), \; y_j\rightarrow q^j x_j y_{j+1}\},\; \DIO \bigr); \\
G_{\sec'} &= \bigl(\{x,y\},\; \{x_j \rightarrow q^j (1 + x_j x_{j+1}), \;  y_j\rightarrow q^j x_j y_{j+1}  \},\;  \LPO \bigr); \\
G_{\Sec} &= \bigl( \{x,y\},\;  \{x_j \rightarrow q^j (1 + x_j x_{j+1}), \; y_j\rightarrow q^j y_j x_{j+1}\},\; \DIO \bigr); \\
G_{\Sec'} &= \bigl(\{x,y\},\; \{x_j \rightarrow q^j (1 + x_j x_{j+1}), \;  y_j\rightarrow q^j y_j x_{j+1}  \},\;  \LPO \bigr). \\\end{align*}
The rules in the aforementioned $q$-grammars are motivated by the action of the $q$-derivative operator for the $q$-trigonometric functions $\tan_q, \sec_q,$ and $\Sec_q$, as described in Proposition \ref{th:Dtanq}. Specifically, we set
\[
x_j := \tan_q(q^j u), \quad
y_j := \sec_q(q^j u) \ \text{ or }\ 
y_j := \Sec_q(q^j u).
\]
We do not reproduce here the results concerning these $q$-grammars. The reader is referred to \cite{Foata-Han-2012} for a complete description and rigorous proofs of these results. We shall mention only one of them, which is related to Mahonian statistics.
Consider the $q$-grammar $G_{\tan'}$. Then the coefficient of the monomial $x_0 x_1 x_2 \cdots x_n$ in $D^n(x_0)$ is given by
\[
(1+q)(1+q+q^2)\cdots (1+q+q^2+\cdots+q^{n-1}).
\]

\subsection{\texorpdfstring{Proposed new $q$-grammars}{Proposed new q-grammars}}
In this paper, we first develop a method for computing the \(q\)-exponential generating functions \eqref{defi:gf} associated with the corresponding \(q\)-grammars, which we refer to as \(q\)-grammatical calculus. As an illustration of this method, we provide a grammatical derivation of $q$-binomial inversion formula (see Example \ref{exa:qbininv}). We also provide a grammatical derivation of Theorem \ref{thm:Foata-Hanaa} (see Subsection \ref{subsec:application}).    We then introduce several new classes of \(q\)-grammars. Although the related problems posed in Section \ref{sec:3.8} could be investigated, our primary focus at the present paper is to construct a suitable evaluation map \(\phi\) such that \(\phi(D^k(s))\) yields several classical combinatorial polynomials, including two \(q\)-analogs of the Eulerian polynomials, \(q\)-analog of the Roselle polynomials and two \(q\)-analogs of the Andr\'e  polynomials. It is worth noting that, in most of our examples, the combinatorial interpretation of \(\phi(D^k(s))\) relies on classical Mahonian statistics for permutations or trees, such as the {\it inversion number} $(\inv)$ and the {\it major index} $(\maj)$.
 See Section \ref{sec:qgrammars} and Section \ref{sec:qAndre} for related definitions. More precisely,

Subsection \ref{subsec:qmajdes}. $(\des, \maj)$ on permutations:  
\begin{itemize}
	\item $G_{\maj}=(\{x,y\}, \{  x_j \rightarrow q^j x_0 y_0, \  y_j \rightarrow q^j x_0 y_0\}    ,  \LPO)$,
	\item $\phi =  \{ x_j \rightarrow xq^j, \  y_j \rightarrow yq^j \}$.
\end{itemize}

\smallskip

Subsection \ref{subsec:qinvdes}. (des, inv) on permutations: 
    \begin{itemize}
	\item $G_{\inv}=(\{x,y\},  \{ x_j \rightarrow q^j y_j x_{j+1}, \  y_j \rightarrow q^j y_j x_{j+1}   \}, \AIO )$,
	\item $\phi =  \{ x_j \rightarrow x, \   y_j \rightarrow y\}$.
\end{itemize}

\smallskip

Subsection \ref{subsec:qroselle}. cycle $q$-Roselle polynomials: 
\begin{itemize}
    \item $G_{\cyc}=(\{x,y,z,e\},  R, \KSO)$, where
    \item $R = \{   x_j \rightarrow q^j y_j x_{j+1} , 
\  y_j \rightarrow q^j y_j x_{j+1} ,
\  z_j \rightarrow q^j y_j x_{j+1} ,
\  e_j \rightarrow tq^j e_j z_{j+1} 
  \} $,
	\item $\phi =  \{ x_j \rightarrow x, \   y_j \rightarrow y,    \   z_j \rightarrow z, \   e_j \rightarrow e \}$.
\end{itemize}

\smallskip

Subsection \ref{sec:AndrI}. (des, inv) on Andr\'e I permutations: 
\begin{itemize}
\item $G_{{\rm AndI}}=(\{x,y\},  \{   x_j \rightarrow q^j x_j y_{j+1} , \  y_j \rightarrow q^j x_j \}    , \AIO)$,
	\item $\phi =  \{ x_j \rightarrow t, \   y_j \rightarrow 1\}$.
\end{itemize}

\smallskip

Subsection \ref{sec:q-AndreII}. (des, inv) on Andr\'e II permutations: 
\begin{itemize}
    \item	$G_{\rm{AndII}}=(\{x,y\},  \{  x_j \rightarrow q^j x_j y_{j+1} , \  y_j \rightarrow q^{j+1} x_{j+1}  \}    , \AIO)$,
	\item $\phi =  \{ x_j \rightarrow t, \   y_j \rightarrow 1\}$.
\end{itemize}

\section{\texorpdfstring{$q$-grammatical calculus}{q-grammatical calculus}} \label{secqgram}

 This section aims to develop $q$-grammatical calculus. Analogous to the grammatical calculus presented in Appendix \ref{sec:survey}, we define the $q$-exponential generating function of $f \in \mathbb{E}$ associated with the $q$-derivative associated with a $q$-grammar. More precisely, let $ G=(S, R, \rho)$ be a $q$-grammar, and let $D$ denote the  $q$-derivative associated with  $G$. For any $f\in \mathbb{E}$, we introduce the following $q$-exponential generating function:
\begin{equation}\label{def:qexp}
    \Gen_q^{(G)}(f;u)=\sum_{n\geq 0}D^n(f)\frac{u^n}{(q;q)_n}.
\end{equation}

Here we choose ${(q;q)_n}$ instead of $n!$. This choice is consistent with $q$-exponential structure.

Let $\phi$ be an evaluation map, define 
    \begin{equation}\label{defi:gf} 
    \mathrm{gen}_q(f,u) = \phi\left(\Gen_q^{(G)}(f;u)\right)=\sum_{k\geq 0} \phi(D^k(f)) \frac{u^k}{(q;q)_k}.
    \end{equation}

The formal power series $\Gen_q^{(G)}(f;u)$ and $\mathrm{gen}_q(f,u)$ are also referred to as the Eulerian generating function; see Goldman and Rota \cite{Goldman-Rota-1970}. This may be viewed as a $q$-analogue of the exponential generating function associated with the grammar defined in \eqref{defi:gfgram}, as detailed in Appendix \ref{sec:survey}. The exponential generating function associated with a classical grammar plays a fundamental role in the grammatical calculus developed by Chen \cite{Chen-1993}. To establish $q$-grammatical calculus, we investigate the properties of the $q$-exponential generating functions \eqref{def:qexp} and \eqref{defi:gf}. It transpires that the $q$-exponential generating function exhibits richer structure, particularly with regard to the multiplicative property \eqref{gramma-multiple}.

This section is structured as follows. In Subsection \ref{subsec:qderivative}, we establish several properties of the  $q$-derivative associated with a $q$-grammar. Specifically, we derive the $q$-Leibniz formula for $q$-derivatives associated with a special class of $q$-grammars (termed  $q$-linear grammars). As an application, we provide a grammatical derivation of the $q$-binomial inversion formula. Subsection \ref{subsec:qEGF} is devoted to studying the $q$-exponential generating function \eqref{def:qexp} and \eqref{defi:gf} using the properties of the  $q$-derivative developed in the previous subsection. We establish the multiplicative property of the $q$-exponential generating function  \eqref{defi:gf}  under restricted evaluation maps and $q$-grammars, as summarized in Theorem \ref{thm:product}. In Subsection \ref{subsec:application}, we illustrate the use of $q$-grammatical calculus by deriving $q$-Hoffman's formula (Theorem \ref{thm:Foata-Hanaa}) as an application of  Theorem \ref{thm:product}.

For the convenience of reference, we collect the relevant formulas concerning  $q$-derivative and its $q$-exponential generating function in
Appendix~\ref{sec:qderivativeform}.

\subsection{\texorpdfstring{Properties of the  $q$-derivative}{Properties of the  q-derivative}} \label{subsec:qderivative}
 
Recall that $\mathbb{E}$ denotes group algebra $\mathbb{K}[q][F(\mathbb{S})]$ on $F(\mathbb{S})$, where $F(\mathbb{S})$ denotes the free group on $\mathbb{S}$. Since $D$ is a linear operator (see Definition \ref{defi:q-derivative}),  the following proposition is obvious:

\begin{Proposition} For $f,g\in\mathbb{E}$ and $c \in  \mathbb{K}[q]$, we have 
\begin{align*}
D(c)&=0,   \\[5pt]
D(cf)&=cD(f),  \\[5pt]
D(f+g)&=D(f)+D(g).  
\end{align*}
\end{Proposition}

Unlike the usual product rule for derivatives, the  $q$-derivative satisfies the following modified product rule: 

\begin{Proposition}\label{lem: product_formula}
For   $f,g\in\mathbb{E}$, we have:
\begin{align}\label{eq:prod}
D(fg)&=D(f)\uparrow g+fD(g).  
\end{align}
\end{Proposition}

\begin{proof}
Since $D$ is a linear operator, it is sufficient to show that \eqref{eq:prod} holds for $f, g \in F(\mathbb{S})$. Assume that  $f=w_1w_2\cdots w_n$ and $g=w_{n+1}w_{n+2}\cdots w_{n+m}$ where each $w_i\in \mathbb{S}\cup\mathbb{S}^{-1}$. Then
\begin{align*}
&D(f)\uparrow g+fD(g)\\&=
D(w_1w_2\cdots w_n)\uparrow (w_{n+1}w_{n+2}\cdots w_{n+m})
+w_1w_2\cdots w_nD(w_{n+1}\cdots w_{n+m})
\\&=\sum_{j=1}^n  \Bigl( w_1w_2\cdots w_{j-1}\, R(w_j)\, \uparrow\! \bigl(w_{j+1}\cdots w_n\bigr) \Bigr)\uparrow (w_{n+1}w_{n+2}\cdots w_{n+m})
\\&~~~~+
w_1w_2\cdots w_n
\sum_{j=1}^m \Bigl( w_{n+1}w_{n+2}\cdots w_{n+j-1}\, R(w_{n+j})\, \uparrow\! \bigl(w_{n+j+1}\cdots w_{n+m}\bigr) \Bigr)
\\&=
\sum_{j=1}^{m+n} \Bigl( w_1w_2\cdots w_{j-1}\, R(w_j)\, \uparrow\! \bigl(w_{j+1}\cdots w_{m+n}\bigr) \Bigr)
\\&=
D(w_1w_2\cdots w_{m+n})
\\&=
D(fg),
\end{align*}
as desired. This completes the proof.
\end{proof}

By induction, one easily derives the following result.
\begin{Proposition}\label{prop:f_product} For any $f_i\in\mathbb{E}~(1\leq i\leq n)$, we have 
$$
D(f_1f_2\cdots f_n) =  \sum_{j=1}^n  f_1f_2\cdots f_{j-1}\, D(f_j)\, \uparrow\! \bigl(f_{j+1}\cdots f_n\bigr). 
$$
 
\end{Proposition}

The following proposition describes how to compute the $q$-derivative of the inverse of $f \in F(\mathbb{S})$.

\begin{Proposition}\label{lem:f_minus_one}
Let $f\in F(\mathbb{S})$, we have 
    \begin{align*}
D\left(f^{-1}\right)&=-f^{-1}{D(f)}\uparrow (f^{-1}).
\end{align*}
\end{Proposition}

\begin{proof}
Suppose that $f=w_1w_2\cdots w_n$ where each $w_i\in \mathbb{S}\cup\mathbb{S}^{-1}$. Then $f^{-1}=w_n^{-1}w_{n-1}^{-1}\cdots w_1^{-1}$ and so 
\begin{align*}
&f^{-1}{D(f)}\uparrow (f^{-1})\\&=
f^{-1}
\sum_{j=1}^n  w_1w_2\cdots w_{j-1}\, R(w_j)\, \uparrow\! \bigl(w_{j+1}\cdots w_n\bigr) 
\uparrow (w_n^{-1}w_{n-1}^{-1}\cdots w_1^{-1})
\\&=
\sum_{j=1}^n  w_n^{-1}w_{n-1}^{-1}\cdots w_{j+1}^{-1}\, w_{j}^{-1}R(w_{j})\, \uparrow\! \bigl(w_{j}^{-1}w_{j-1}^{-1}\cdots w_1^{-1}\bigr)
\\&=-
\sum_{j=1}^n  w_n^{-1}w_{n-1}^{-1}\cdots w_{j+1}^{-1}\, R(w_{j}^{-1})\, \uparrow\! \bigl(w_{j-1}^{-1}\cdots w_1^{-1}\bigr)
\\&=-
D(w_n^{-1}w_{n-1}^{-1}\cdots w_1^{-1})
\\&=-
D(f^{-1})
,
\end{align*}
which completes the proof.
\end{proof}

Because the  $q$-derivative does not adhere to the classical product rule, the $q$-Leibniz formula ($q$-analogue of \eqref{leibniz}) is not satisfied for general $q$-grammars. We find that it holds for a special case of $q$-grammars, which we define as $q$-linear grammars.

\begin{Definition}[$q$-Linear grammar]
 A $q$-grammar $G=(S,R,\rho)$ is called $q$-linear if $\rho=\textbf{KSO}$, and for each master variable $s \in \mathbb{S}$ and any $i\geq 0$, we have
  \begin{align}\label{eq:rule}
      R(s_{i+1})= R(\uparrow s_{i})=q\uparrow R(s_{i}).
  \end{align} 
\end{Definition}

For example, the $q$-grammar $G_{\cyc}$ defined in Subsection  \ref{subsec:qroselle} is $q$-linear. 

For a $q$-linear grammar, we have the following consequence. 

\begin{Proposition} \label{q_linear_uparrow-prop} Let $G$ be a $q$-linear grammar and let $D$ be $q$-derivative associated with $G$. For $m,~n\geq 1$ and $f\in \mathbb{E}$, we have
\begin{equation}\label{q_linearaa}
D^n(\uparrow^m f)=q^{nm} \uparrow^{m} \left(D^n(f)\right).
\end{equation}
\end{Proposition}
\begin{proof}
    Just need to prove the case $m=n=1$ and $f\in \mathbb{S}\cup\mathbb{S}^{-1}$, which is easy to verify by the definition of $q$-linear grammar and Proposition \ref{prop:f_product}.
\end{proof}

Using Proposition \ref{q_linear_uparrow-prop}, we derive the following $q$-Leibniz formula for the  $q$-derivative associated with a $q$-linear grammar. Recall that the $q$-binomial coefficient (also called the Gaussian polynomial) is defined by
\[
 {n + m \brack n}_q =
\begin{cases}
\dfrac{(1 - q^{n+m})(1 - q^{n+m-1}) \cdots (1 - q^{m+1})}{(1 - q^n)(1 - q^{n-1}) \cdots (1 - q)}, & \text{for } n, m \geq 0, \\[1em]
0, & \text{otherwise},
\end{cases}
\]
see Andrews \cite[Chapter~1]{Andrews-1976}. 

\begin{Proposition} \label{q_linear_prodrule-prop} Let $G$ be a $q$-linear grammar and let $D$ be  $q$-derivative associated with $G$.  For $n\geq 1$ and $f,g\in \mathbb{E}$, we have
\begin{equation}\label{q_linear_prodrule}
D^n(fg)=\sum_{k=0}^n {n \brack k}_q D^k(f) \uparrow^k \left(D^{(n-k)}(g)\right),
\end{equation}
\end{Proposition}

\begin{proof} We proceed by induction on $n$. The case $n=1$ follows from Proposition \ref{lem: product_formula}. Assume that the identity in \eqref{q_linear_prodrule} holds for some $n$, that is, 
\begin{equation}\label{q_linear_pf-prod}
D^{n}(fg)=\sum_{k=0}^{n} {n \brack k}_q D^k(f)\uparrow^k \left(D^{(n-k)}(g)\right).
\end{equation}
We now aim to show that the identity holds for $n+1$. To do this, we apply $q$-derivative $D$ to the both sides of \eqref{q_linear_pf-prod}:  
\begin{align*}
&D^{n+1}(fg)\\[5pt]
&=D\left(\sum_{k=0}^{n} {n \brack k}_q   D^k(f)\uparrow^k \left(D^{(n-k)} (g)\right)\right)\\[5pt]
&\overset{\eqref{q_linearaa}}{=}\sum_{k=0}^{n} {n \brack k}_q    D^{k+1} (f)\uparrow^{k+1} \left(D^{(n-k)} (g)\right) +\sum_{k=0}^{n} {n \brack k}_q    q^kD^k(f)\uparrow^{k} \left(D^{(n-k+1)} (g)\right) \\[5pt]
&=\sum_{k=1}^{n+1} {n \brack k-1}_q     D^{k} (f)  \uparrow^{k} 
 \left(D^{(n-k
+1)} (g) \right)+\sum_{k=0}^{n} {n \brack k}_q    q^kD^k(f)\uparrow^{k} \left(D^{(n-k+1)} (g)\right) 
\\[5pt]
&=D^{n+1} (f)\uparrow^{n+1} g+fD^{(n+1)} (g)\\[5pt]
&\quad \quad +\sum_{k=1}^{n} \left( {n \brack k-1}_q  +q^k {n \brack k}_q   \right)  D^{k} (f)\uparrow^{k} \left(D^{(n-k+1)} (g)\right) \\[5pt]
&=D^{n+1} (f)\uparrow^{n+1} g+fD^{(n+1)} (g)+ \sum_{k=1}^{n} {n+1 \brack  k}_q    D^{k} (f)\uparrow^{k} \left(D^{(n-k+1)} (g)\right) \\[5pt]
&=\sum_{k=0}^{n+1} {n+1 \brack  k}_q    D^{k} (f)\uparrow^{k} \left(D^{(n-k+1)} (g)\right), 
\end{align*}
as desired. This completes the proof. 
\end{proof}  

As an application of Proposition \ref{q_linear_prodrule-prop}, we provide a grammatical derivation of the following $q$-binomial inversion \cite[Corollary 3.38]{Aigner-1979}. 

\begin{Example}\label{exa:qbininv} Let $a_0, a_1, a_2,\ldots$ and $b_0,b_1,b_2,\ldots$ be two sequences. Then we have the following $q$-inversion pair: 
\begin{align}\label{q-inversion}
(i). \ a_n=\sum_{k=0}^n {n \brack k}_q  b_k  \quad  \Longleftrightarrow \quad 
(ii).\  b_n=\sum_{k=0}^n (-1)^{n-k}  {n \brack k}_q q^{n-k \choose 2} a_k .
\end{align}
\end{Example}
\begin{proof}
We consider the following $q$-linear grammar:  
\begin{equation}\label{def:grambinv}
G_{{\rm binv}} = \bigl(\{x,y\}, \{ x_j \mapsto q^{j}x_{j+1},\; y_j \mapsto q^{j}y_{j} \}, \KSO \bigr).
\end{equation}

Let $D$ be the $q$-derivative associated with $G_{{\rm binv}}$.
It is easy to check that for $n\geq 1$, 
\begin{align}\label{eq:inverfor1}
D^n(x_0)=q^{n\choose 2}x_n, \  D^n(y_0)=y_0.
\end{align}
Using Proposition \ref{lem:f_minus_one}, we have 
\[D(y_i^{-1})=-y_i^{-1}D(y_i)\uparrow (y_{i}^{-1})=-q^iy_i^{-1}y_iy_{i+1}^{-1}=-q^iy_{i+1}^{-1}.\] 
So, by induction,  we obtain that for $n\geq 0$, 
\begin{equation}\label{eq:inverfor2}
D^n(y_0^{-1})=(-1)^nq^{n\choose 2}y_n^{-1}.
\end{equation}
 
Since $G_{{\rm binv}}$ is a $q$-linear grammar, by Proposition \ref{q_linear_prodrule-prop}, we find that 
$$
D^n(x_0y_0)=\sum_{k=0}^n {n \brack k}_q D^k(x_0)\uparrow^k \left(D^{n-k}(y_0)\right)
=\sum_{k=0}^n {n \brack k}_q q^{k\choose 2} x_k y_k. 
$$
Denote $b_i$ by $q^{i\choose 2}x_iy_i$. Then, the first identity in \eqref{q-inversion} can be rewritten as 
\begin{equation}\label{pf-qinv}
D^n(x_0y_0)=a_n.
\end{equation}
Now, suppose it is true. From Proposition \ref{q_linear_prodrule-prop}, and using \eqref{eq:inverfor1} and \eqref{eq:inverfor2}, we deduce that   
\begin{align*}
b_n&=  y_n D^n(x_0)=y_nD^n(x_0y_0y^{-1}_0)\nonumber\\[5pt]
&=y_n\sum_{k=0}^n {n \brack k}_qD^k(x_0y_0)\uparrow^k \left(D^{n-k}(y^{-1}_0)\right) \nonumber \\[5pt]
& =y_n\sum_{k=0}^n {n \brack k}_q a_k(-1)^{n-k}q^{n-k \choose 2}y^{-1}_n \nonumber \\[5pt]
&=\sum_{k=0}^n {n \brack k}_q (-1)^{n-k} q^{n-k \choose 2}  a_k. 
\end{align*}
The converse can be proved similarly. 
\end{proof}

\subsection{\texorpdfstring{$q$-Exponential generating function}{q-Exponential generating function}} \label{subsec:qEGF} This section is devoted to studying the $q$-exponential generating functions of $D^n(f)$ and $\phi(D^n(f))$, making use of the $q$-derivative properties established in the previous section.

\begin{Proposition}\label{pro:q_gramma}Let $G=(S,R,\rho)$ be a $q$-grammar and let $D$ be  $q$-derivative associated with $G$. For  $f,g\in \mathbb{E}$, we have 
\begin{align} 
{\rm Gen}_q^{(G)}(f+g;u) &= {\rm Gen}_q^{(G)}(f;u)+{\rm Gen}_q^{(G)}(g;u), \label{q_gramma-add}\\[5pt]
D_q{\rm Gen}_q^{(G)}(f;u) &=
{\rm Gen}_q^{(G)}(D(f);u), \label{q_gramma-dif}
\end{align}
where $D_q$ is the real $q$-derivative operator defined in \eqref{defi:qderiv}.
\end{Proposition}
\begin{proof} 
By the linearity of $q$-derivative associated with $q$-grammar, it is straightforward to verify \eqref{q_gramma-add}. 
By definition, we have
\begin{align*}
D_q {\rm Gen}_q^{(G)}(f;u) 
&= \frac{1}{u} \left(
\sum_{n\geq 0}D^n(f)\frac{u^n}{(q;q)_n}
-
\sum_{n\geq 0}D^n(f)\frac{(uq)^n}{(q;q)_n}
\right)\\
&= \frac{1}{u}  \sum_{n\geq 0}D^n(f)\frac{(1-q^n)u^n}{(q;q)_n}\\
&=  \sum_{n\geq 1}D^n(f)\frac{u^{n-1}}{(q;q)_{n-1}}\\
&=   \sum_{n\geq 0}D^{n+1}(f)\frac{u^n}{(q;q)_{n}}\\
&={\rm Gen}_q^{(G)}(D(f);u).
\end{align*}
\end{proof}

As said in Subsection \ref{subsec:qderivative}, the $q$-Leibniz formula does not hold for $q$-derivatives associated with general $q$-grammars, but we show that it is valid for $q$-linear grammars. This will lead to the multiplicative property of $q$-exponential generating functions \eqref{def:qexp}  associated with the $q$-linear grammar. We begin with the following proposition.

\begin{Proposition} Let $G$ be a $q$-linear grammar. For $m\geq 1$ and $f\in \mathbb{E}$,
$$
{\rm Gen}^{(G)}_q(\uparrow^m f;u) = \uparrow^m {\rm Gen}^{(G)}_q(f;uq^m).
$$
\end{Proposition}
\begin{proof} Using Proposition \ref{q_linear_uparrow-prop},  we have 
\begin{align*}
{\rm Gen}^{(G)}_q(\uparrow^m f;u)
&=\sum_{n\geq 0}D^n(\uparrow^m f)\frac{u^n}{(q;q)_n}\\
&=\sum_{n\geq 0}(q^{nm} \uparrow^m D^n(f))\frac{u^n}{(q;q)_n}\\
&=\sum_{n\geq 0} \uparrow^m D^n(f)\frac{(uq^m)^n}{(q;q)_n}
\\&=\uparrow^m {\rm Gen}^{(G)}_q(f;uq^m),
\end{align*}
as required. 
\end{proof}

The next proposition establishes the multiplicative property of $q$-exponential generating functions \eqref{def:qexp} associated with  $q$-linear grammars. 

\begin{Proposition}
Let $G$ be a $q$-linear grammar and   let $D$ be $q$-derivative associated with $G$. For $f,g\in \mathbb{E}$, we have 
    \begin{align}\label{eq:multprop}
{\rm Gen}_q^{(G)}(fg;u)= \sum_{k\geq 0} D^k(f) \frac{u^k}{(q; q)_k}  \uparrow^k
  {\rm Gen}_q^{(G)}(g;u).
\end{align}
\end{Proposition}
\begin{proof}
By Proposition \ref{q_linear_prodrule-prop}, we have
\begin{align*}
&{\rm Gen}_q^{(G)}(fg;u)= \sum_{n=0}^{+\infty} D^n(fg) \frac{u^n}{(q; q)_n}
\\&
= \sum_{n=0}^{+\infty} \left( \sum_{k=0}^n {n \brack  k}_q D^k(f) \uparrow^k \left(D^{n-k} (g) \right) \right) \frac{u^n}{(q; q)_n}
\\&
= \sum_{k=0}^{+\infty} D^k(f) \frac{u^k}{(q; q)_k} \sum_{n=k}^{+\infty} \uparrow^k \left(D^{n-k} (g) \right)\frac{u^{n-k}}{(q; q)_{n-k}}
\\&
= \sum_{k=0}^{+\infty} D^k(f) \frac{u^k}{(q; q)_k} \sum_{n=0}^{+\infty} \uparrow^k \left( D^n (g)\right) \frac{u^n}{(q; q)_n}
\\&
=\sum_{k\geq 0} D^k(f) \frac{u^k}{(q; q)_k}  \uparrow^k
  {\rm Gen}_q^{(G)}(g;u),
\end{align*}
as desired. 
\end{proof}

 We proceed to establish the multiplicative property of $q$-exponential generating functions \eqref{defi:gf}  for $q$-linear grammars under a restricted class of evaluation maps.

\begin{Definition}[Master-linear   evaluation]
  An evaluation map $\phi \colon  \mathbb{S} \rightarrow \mathbb{V}$ is called master-linear if $\phi (s_i)= \phi (s_j)$ for each master variable $s \in {S}$ and any $i,j\geq 0$. 
\end{Definition}

For illustration, the evaluation maps considered in Subsections \ref{subsec:qinvdes}, \ref{subsec:qroselle}, \ref{sec:AndrI} and \ref{sec:q-AndreII} are master-linear, while the evaluation map considered in Subsection \ref{subsec:qmajdes} fails to satisfy this condition.

By applying a master-linear evaluation $\phi$ to both sides of \eqref{eq:multprop}, we derive the following multiplicative property:

\begin{Proposition}\label{prop: q_linear_gramma-mul*}
    Let $G$ be a $q$-linear grammar and let $\phi$ be a master-linear evaluation. For $f,g\in \mathbb{E}$, we have
    \begin{align}\label{q_linear_gramma-mul*}
\phi\left({\rm Gen}_q^{(G)}(fg;u)\right) &= \phi\left({\rm Gen}_q^{(G)}(f;u)\right)\cdot \phi\left({\rm Gen}_q^{(G)}(g;u)\right). 
\end{align}
Furthermore, by induction, we obtain that for any $n>0$,  
\begin{align}\label{q_linear_gramma-mul2}
\phi\left({\rm Gen}_q^{(G)}(\prod_{k=1}^n f_k;u)\right) &= 
\prod_{k=1}^n\phi\left({\rm Gen}_q^{(G)}(f_k;u)\right). 
\end{align}
\end{Proposition}

From Proposition \ref{prop: q_linear_gramma-mul*}, we see that reordering the variables in $f \in \mathbb{E}$  does not change the $q$-exponential generating function of  $\phi(D^n(f))$. This observation implies that the product formula  \eqref{q_linear_gramma-mul*}  is not limited solely to $q$-linear grammar. In fact, it holds for more general classes of $q$-grammar than  $q$-linear grammar when $\phi$ is a master-linear evaluation map, which yields the central result of our $q$-grammatical calculus.

 \begin{Theorem} \label{thm:product} Let $G=(S,R,\rho)$ be a $q$-grammar satisfying \eqref{eq:rule}, that is,  for each   $s_i \in \mathbb{S}$ and any $i\geq 0$, we have $
      R(s_{i+1})= R(\uparrow s_{i})=q\uparrow R(s_{i}). $ 
Let $\phi$ be a master-linear evaluation. For $f,g\in \mathbb{E}$, we have
\begin{align}
\phi\left({\rm Gen}_q^{(G)}(fg;u)\right) &= \phi\left({\rm Gen}_q^{(G)}(f;u)\right)\cdot \phi\left({\rm Gen}_q^{(G)}(g;u)\right).
\end{align}
\end{Theorem}

Theorem \ref{thm:product}  is a consequence of Proposition \ref{prop: q_linear_gramma-mul*} together with the proposition below.

\begin{Proposition}\label{prop: q_linear_gramma-mul-2}
    Let $G=(S,R,\rho)$ be a $q$-grammar such that \eqref{eq:rule} holds and let $\hat{G}=(S,R,\KSO)$ be its corresponding $q$-linear grammar. Let 
   $\phi$ be a master-linear evaluation. For $f \in \mathbb{E}$, we have
    \begin{align}\label{q_linear_gramma-mul-2}
\phi\left({\rm Gen}_q^{(G)}(f;u)\right) &= 
\phi\left({\rm Gen}_q^{(\hat{G})}(f;u)\right). 
\end{align} 
\end{Proposition}

\begin{proof}
Let $D_{G}$ and $D_{\hat{G}}$ be the $q$-derivatives  associated with $G$ and $\hat{G}$ respectively. By definition, we have
$$
D_{ G}=\rho \circ D_{\hat{G}}.
$$
   To prove \eqref{q_linear_gramma-mul-2}, it suffices  to show that
 for   $f \in \mathbb{E}$, 
    $$
\phi\left(D_{ G}^{n}(f)\right)= \phi\left(D_{\hat{G}}^{n}(f)\right)
$$
holds for $n\geq 0$.

We will prove a more general result:
\begin{align}\label{q_linear_gramma-mul-3}
\phi\left({\rm Gen}_q^{(\hat{G})}(D_{ G}^{n}(f);u)\right) &= 
\phi\left({\rm Gen}_q^{(\hat{G})}(D_{ \hat{G}}^{n}(f);u)\right). 
\end{align}

The case $n=0$ is trivial.
Suppose that \eqref{q_linear_gramma-mul-3} holds for some $n$. Then we consider the $n+1$ case.
First,  \eqref{q_linear_gramma-mul-3} implies 
\begin{align}\label{q_linear_gramma-mul-4}
\phi\left({\rm Gen}_q^{(\hat{G})}(D_{\hat{G}}\circ D_{ G}^{n}(f);u)\right) &= 
\phi\left({\rm Gen}_q^{(\hat{G})}(D_{\hat{G}}\circ D_{ \hat{G}}^{n}(f);u)\right)
\nonumber \\&=
\phi\left({\rm Gen}_q^{(\hat{G})}(D_{\hat{G}}^{n+1}(f);u)\right). 
\end{align}
On the other hand, by Proposition \ref{prop: q_linear_gramma-mul*}, we find that
$$
\phi\left({\rm Gen}_q^{(\hat{G})}(\rho \circ D_{ \hat{G}}\circ D_{ G}^{n}(f);u)\right) = 
\phi\left({\rm Gen}_q^{(\hat{G})}(D_{ \hat{G}}\circ D_{ G}^{n}(f);u)\right),
$$
which implies that
\begin{align}\label{q_linear_gramma-mul-5}
\phi\left({\rm Gen}_q^{(\hat{G})}( D_{ G}^{n+1}(f);u)\right) &= 
\phi\left({\rm Gen}_q^{(\hat{G})}(D_{\hat{G}}\circ D_{ G}^{n}(f);u)\right).
\end{align}
Putting \eqref{q_linear_gramma-mul-4} and \eqref{q_linear_gramma-mul-5} together, we obtain
\begin{align}\label{q_linear_gramma-mul-6}
\phi\left({\rm Gen}_q^{(\hat{G})}(D_{ G}^{n+1}(f);u)\right) &= 
\phi\left({\rm Gen}_q^{(\hat{G})}(D_{\hat{G}}^{n+1}(f);u)\right),
\end{align}
which completes the proof.
\end{proof}

\subsection{\texorpdfstring{A grammatical derivation of $q$-Hoffman's formula}{A grammatical derivation of q-Hoffman's formula}} \label{subsec:application}

Here we provide grammatical derivations of Theorem \ref{thm:Foata-Hanaa} based on the $q$-grammatical calculus developed in this section.

Let $D$ be $q$-derivative associated with $G_{\tan \cup \sec}$ defined in \eqref{defi:gramartansec}. Using Proposition \ref{lem:f_minus_one},  we have 
\begin{align*}
D(y_0^{-1})&=-y_0^{-1}D(y_0)\uparrow (y_0^{-1})\\
&=-y_0^{-1}(x_0y_1)y_1^{-1}\\
&=-y_0^{-1}x_0.
\end{align*}
Thus, by Proposition \ref{lem: product_formula}, we have 
\begin{align*}
D(y_0^{-1}x_0)&=D(y_0^{-1})x_1+y_0^{-1}D(x_0)\\
&=-y_0^{-1}x_0x_1+y_0^{-1}(1+x_0x_1)\\
&=y_0^{-1}. 
\end{align*}
By induction, we derive that  for $n\geq 0$, 
    \begin{align*} 
        D^{2n}(y_0^{-1}) =(-1)^ny_0^{-1}, \quad  D^{2n+1}(y_0^{-1}) =(-1)^{n+1}y_0^{-1}x_0
    \end{align*}  
 and 
    \begin{align*} 
        D^{2n}(y_0^{-1}x_0) =(-1)^ny_0^{-1}x_0, \quad    D^{2n+1}(y_0^{-1}x_0) =(-1)^{n}y_0^{-1}. 
    \end{align*} 
Consequently,  
 \begin{align}
{\rm Gen}_q(y_0^{-1};u)&=\sum_{n=0}^{+\infty}D^{n} (y_0^{-1}) \frac{u^{n}}{(q; q)_{n}} \nonumber \\[5pt]
&= \sum_{n\geq 0}(-1)^ny_0^{-1}\frac{u^{2n}}{(q;q)_{2n}}+\sum_{n\geq 0}(-1)^{n+1}y_0^{-1}x_0\frac{u^{2n+1}}{(q;q)_{2n+1}} \nonumber  \\[5pt]
&= y^{-1}_0(\cos_q(u) - x_0\,\sin_q(u)) \label{eq:qHoffpfa}
\end{align}
and 
\begin{align}
{\rm Gen}_q(y_0^{-1}x_0;u)&=\sum_{n\geq 0}D^{n} (y_0^{-1}x_0) \frac{u^{n}}{(q; q)_{n}} \nonumber \\[5pt]
&= \sum_{n\geq 0}(-1)^ny_0^{-1}x_0\frac{u^{2n}}{(q;q)_{2n}}+\sum_{n\geq 0}(-1)^{n}y_0^{-1}\frac{u^{2n+1}}{(q;q)_{2n+1}}\nonumber \\[5pt]
&= y^{-1}_0(x_0\,\cos_q(u) + \sin_q(u)). \label{eq:qHoffpfb}
\end{align}

\begin{proof}[Grammatical Proofs of Theorem \ref{thm:Foata-Hanaa}]
Let $\phi$ denote the evaluation map that sends $x_j$ to $x$ and $y_j$ to $y$. It is straightforward to verify that $\phi$ is a master-linear evaluation. Note that the grammar $G_{\tan \cup \sec}$ defined in \eqref{defi:gramartansec} is not a $q$-linear grammar, as the order $\rho=\DIO$  violates the defining conditions of $q$-linearity; however, this grammar still satisfies \eqref{eq:rule}.  Hence, by 
Theorem~\ref{thm:product}, we conclude that
\begin{align*}
1&=\phi({\rm Gen}_q(y_0  y_0^{-1};u)) = \phi({\rm Gen}_q(y_0;u))\phi({\rm Gen}_q(y^{-1}_0;u)). 
\end{align*}
It implies that 
\begin{align}
\phi({\rm Gen}_q(y_0;u))=&\sum_{k=0}^{+\infty} \phi(D^k(y_0)) \frac{u^k}{(q; q)_k} \nonumber \\[5pt]
&=
\frac{1}{\phi({\rm Gen}(y^{-1}_0;u))} \nonumber \\[5pt]
&\overset{\eqref{eq:qHoffpfa}}{=}\frac{y}{\cos_q(u) - x\,\sin_q(u)}, \label{eq:qHoffpfc}
\end{align}
which can be shown to be equivalent to \eqref{eq:gf:sec}. 

For \eqref{eq:gf:tan}, we find that 
\begin{align*}
\phi({\rm Gen}_q(x_0;u))&=\phi({\rm Gen}_q(y_0y_0^{-1}x_0;u))\\[5pt]
&= \phi({\rm Gen}_q(y_0;u))\phi({\rm Gen}_q(y_0^{-1}x_0;u))   
\\[5pt] 
& \overset{\eqref{eq:qHoffpfb}\eqref{eq:qHoffpfc}}{=}\frac{y}{\cos_q(u) - x\,\sin_q(u)}\times y^{-1}(x\,\cos_q(u) + \sin_q(u))\\[5pt]
&=\frac{x\cos_q(u)+\sin_q(u)}{\cos_q(u)-x\sin_q(u)}\\[5pt]
&=\frac{x+\tan_q(u)}{1-x\tan_q(u)}.  
\end{align*}
One can show that this is equivalent to \eqref{eq:gf:tan}. This completes the proof.  
\end{proof}

\section{\texorpdfstring{$q$-Grammars for  permutations}{q-Grammars for  permutations}} \label{sec:qgrammars}

In this section, we construct $q$-analogues of the two grammars established by Dumont \cite{Dumont-1996}, which address enumeration problems connected to the classical Eulerian polynomials defined on the set of permutations.

 Let $\mathfrak{S}_n$ denote the set of permutations on $[n] := \{1, 2, \dots, n\}$. For a permutation $\sigma = \sigma_1 \cdots \sigma_n \in \mathfrak{S}_n$, we assume that a zero is patched at the beginning and at the end, that is, $\sigma_0 = \sigma_{n+1} = 0$.  An index $i$ is a descent    if   $\sigma_i > \sigma_{i+1}$; otherwise, $i$ is called an ascent.  Let $\des(\sigma)$ count the number of descents of $\sigma$ and let $\asc(\sigma)$ count the number of ascents of $\sigma$.   For $n\ge1$, the bivariate Eulerian polynomials are defined by
\begin{equation}\label{defi:euler}
A_n(x,y)=\sum_{\sigma \in \mathfrak{S}_n} x^{\mathrm{asc}(\sigma)} y^{\mathrm{des}(\sigma)},
\end{equation}
and the exponential generating function of Eulerian polynomials is well known:
\begin{equation}\label{gf-Euler-a}
\sum_{n=0}^{\infty} A_n(x,y)\frac{u^n}{n!}
= \frac{x-y}{1-x^{-1}y\,e^{(x-y)u}}.
\end{equation}
 Dumont \cite{Dumont-1996} found the grammar \eqref{gramEuler} to generate the bivariate Eulerian polynomials $A_n(x,y)$ and 
Chen and Fu \cite{Chen-Fu-2017}  provided a grammatical derivation of \eqref{gf-Euler-a}, see Appendix \ref{sec:survey}.

For a permutation $\sigma = \sigma_1 \cdots \sigma_n \in \mathfrak{S}_n$, an index $1 \le i \le n$ is called an excedance if $\sigma_i > i$, or a drop if $\sigma_i < i$, or a fixed point if $\sigma_i = i$. Clearly, $n$ cannot be an excedance and $1$ cannot be a drop. The number of excedances, the number of drops and the number of fixed points of $\sigma$ are denoted by $\exc(\sigma)$, $\drop(\sigma)$ and $\fix(\sigma)$, respectively. A drop of a permutation is also called an anti-excedance.

The joint distribution of (exc, fix) was determined by Foata-Sch\"utzenberger \cite{Foata-Sch-1970}, see also Shin-Zeng \cite{Shin-Zeng-2012}. For $n \ge 1$, define
\[
F_n(x,z) = \sum_{\sigma \in \mathfrak{S}_n} x^{\exc(\sigma)} z^{\fix(\sigma)}
\]
and define $F_0(x,z) = 1$. Then
\begin{equation}\label{gf:rossel}
\sum_{n=0}^{\infty} F_n(x,z) \frac{u^n}{n!} = \frac{(1-x)e^{zu}}{e^{xu} - x e^u}. 
\end{equation}
Writing
\begin{equation*}\label{defi:Rossele}
F_n(x,y,z) = \sum_{\sigma \in \mathfrak{S}_n} x^{\exc(\sigma)} y^{\drop(\sigma)} z^{\fix(\sigma)}
\end{equation*}
and $F_0(x,y,z) = 1$, \eqref{gf:rossel} can be converted into the homogeneous form
\begin{equation*}\label{gf:birossel}
\sum_{n=0}^{\infty} F_n(x,y,z) \frac{u^n}{n!} = \frac{(y-x)e^{zu}}{y e^{xu} - x e^{yu}}.  
\end{equation*}

The cycle Roselle polynomials $F_n(x,y,z|\beta)$ are defined by 
\begin{equation}\label{defi:cycRossele}
F_n(x,y,z|\beta) = \sum_{\sigma \in \mathfrak{S}_n} x^{\exc(\sigma)} y^{\drop(\sigma)} z^{\fix(\sigma)}\beta^{{\cyc}(\sigma)}.
\end{equation}
Using the exponential formula, Ksavrelof-Zeng  \cite{Ksavrelof-Zeng-2003} found that 
\begin{equation}\label{gf:birosselcyc}
\sum_{n=0}^{\infty} F_n(x,y,z|\beta) \frac{u^n}{n!} = \left(\frac{(y-x)e^{zu}}{y e^{xu} - x e^{yu}}\right)^{\beta}. 
\end{equation}

Dumont \cite{Dumont-1996} showed that the grammar given in \eqref{gram-Rosselle} generates the polynomials  $F_n(x,y,z)$, which he referred to as  Roselle polynomials.  Chen and Fu \cite{Chen-Fu-2023a} provided a grammatical derivation of \eqref{gf:birossel}.   Ma-Ma-Yeh-Yeh \cite{MaMaYehYeh2024}  found the grammar to generate cycle Roselle polynomials $F_n(x,y,z|\beta)$.

This section is devoted to establishing $q$-grammars for $q$-analogues of Eulerian polynomials \eqref{defi:euler} and cycle Roselle polynomials \eqref{defi:cycRossele}.

Several $q$-analogs of Eulerian polynomials with combinatorial meanings have been studied in the literature (see \cite{Carlitz1975, Carlitz1954, Garsia1979,  Petersen2015, ShareshianWachs2007,ShareshianWachs2011}). In this section, we mainly provide $q$-grammars for $q$-maj-Eulerian polynomials and $q$-inv-Eulerian polynomials.   

Recall that the \textit{major index} $\maj(\sigma)$ of a permutation $\sigma \in \mathfrak{S}_n$ is the sum of all the descents of $\sigma$, i.e., 
\[
\maj(\sigma) := \sum_{i \colon \sigma_i > \sigma_{i+1}} i.
\]
The inversion statistic is defined by 
\[
\inv(\sigma)= \#\bigl\{ (i,j) \colon i < j,\,  \sigma_i > \sigma_j \bigr\}.
\]

The $q$-maj-Eulerian polynomials are defined by: 
\begin{equation}\label{defi:qEulermaj}
A^{ \maj}_n(q;x,y)=\sum_{\sigma \in \mathfrak{S}_n}q^{\operatorname{maj} (\sigma)} x^{\operatorname{asc} (\sigma)} y^{\des (\sigma)}.
\end{equation}

The $q$-inv-Eulerian polynomials are defined by: 
\begin{equation}\label{defi:qEulerinv}
A^{  \inv}_n(q;x,y)=\sum_{\sigma \in \mathfrak{S}_n}q^{\inv (\sigma)} x^{\operatorname{asc} (\sigma)} y^{\des (\sigma)}.
\end{equation}

To introduce the $q$-analogue of cycle Roselle polynomials, let us first interpret Roselle polynomials $F_n(x,y,z)$ based on the one-line representation of permutations.  Let $\sigma = \sigma_1 \cdots \sigma_n \in \mathfrak{S}_n$.  We first place a bar after each right-to-left minimum of $\sigma$ and place a bar at the beginning of $\sigma$. Recall that a right-to-left minimum of $\sigma$ is an element $\sigma_i$ such that $\sigma_i < \sigma_j$ for every $j > i$. For the permutation $\sigma = 5\,4\,1\,2\,7\,3\,6\,9\,10\,8$, we have
\[
\sigma =\mid 5\,4\,1 \mid 2 \mid 7\,3 \mid 6 \mid 9\,10\,8 \mid.
\]
If there is only one element $\sigma_j$ between two bars, then we call $\sigma_j$ an isolated element. The number of isolated elements of $\sigma$ is denoted by $\operatorname{isol}(\sigma)$. 

In order to single out ascents that are not isolated elements, we say that an index $0 \leq i \leq n-1$ of $\sigma$ is a non-isolated ascent if $i$ is an ascent and $\sigma_{i+1}$ is not isolated. The number of non-isolated ascents of $\sigma$ is denoted by $\operatorname{iasc}(\sigma)$. It is easy to check that 
\[
\operatorname{iasc}(\sigma)+\operatorname{isol}(\sigma)=\operatorname{asc}(\sigma).
\]

We now construct a bijection $\psi$ on $\mathfrak{S}_n$. Let $\sigma \in \mathfrak{S}_n$. 
Take the cycle decomposition of $\sigma$ such that cycles are written with their smallest element last and the cycles are written in increasing order of their smallest element. Then erasing the parentheses yields $\tau=\psi(\sigma)$. 

For example, let $\sigma= 5\,2\,7\,1\,4\,6\,3\,9\,10\,8$ and we adopt the following cycle form of $\sigma$: $(5\,4\,1)\, (2) \, (7\,3)\, (6) \, ( 9\,10\,8)$. Then $\tau=\psi(\sigma)=5\,4\,1\,2\,7\,3\,6\,9\,10\,8$. 

It is easy to check that 
\begin{align*}
\operatorname{drop}(\sigma)&=\des(\tau)-1,\\[2pt]
\operatorname{exc}(\sigma)&=\operatorname{iasc}(\tau),\\[2pt]
\operatorname{fix}(\sigma)&=\operatorname{isol}(\tau),\\[2pt]
\operatorname{cyc}(\sigma)&=\operatorname{RLmin}(\tau).
\end{align*}
Note that the map $\psi$  can be viewed as a variation of Foata's first fundamental transformation \cite{Foata-Sch-1970}. 

Hence, cycle Roselle polynomials \eqref{defi:cycRossele} can be interpreted as 
\begin{equation}\label{defi:Rossele-lin}
F_n(x,y,z|\beta) = \sum_{\sigma \in \mathfrak{S}_n} x^{\operatorname{iasc}(\sigma)} y^{\des(\sigma)-1} z^{\operatorname{isol}(\sigma)}\beta^{\operatorname{RLmin}(\sigma)}.
\end{equation}

Based on the definition \eqref{defi:Rossele-lin} of cycle Roselle polynomials, we consider the following  cycle $q$-inv-Roselle polynomials
\begin{align}\label{def:qRoselle}
F^{\inv}_n(q;x,y,z|\beta) 
&=\sum_{\sigma \in \mathfrak{S}_n} q^{\inv(\sigma)} x^{\operatorname{iasc}(\sigma)}z^{\operatorname{isol}(\sigma)} y^{\des(\sigma)-1} \beta^{\operatorname{RLmin}(\sigma)}.
\end{align}

The main objective of this section is to establish $q$-grammars for the $q$-maj-Eulerian polynomials \eqref{defi:qEulermaj},   the $q$-inv-Eulerian polynomials \eqref{defi:qEulerinv} and the $q$-inv-Roselle polynomials \eqref{def:qRoselle}.  Our proofs rely on the grammatical labeling technique introduced by Chen and Fu \cite{Chen-Fu-2017}. As outlined in Appendix \ref{sec:survey},   grammatical labeling exhibits how the substitution rules in context-free grammar arise in the construction of the combinatorial structures.  Furthermore, this technique carries over naturally to the $q$-grammar context.  

Harnessed by these $q$-grammars, we develop a $q$-calculus to derive the $q$-exponential generating functions for the $q$-inv-Eulerian polynomials \eqref{defi:qEulerinv} and the $q$-inv-Roselle polynomials \eqref{def:qRoselle} based on the framework developed in Section \ref{secqgram}. The resulting identities serve as the $q$-analogues of \eqref{gf-Euler-a} and \eqref{gf:birosselcyc}.   Setting $x=1$ in Theorem \ref{thm:qEuler}, we recover the $q$-analogue of the generating function of Eulerian polynomials due to Stanley  \cite{Stanley-1976}.

\begin{Theorem}\label{thm:qEuler} We have 
\begin{equation}
\sum_{n\geq 0} A^{  \inv}_n(q;x,y)\frac{u^n}{(q;q)_n}=\frac{x-y}{1-x^{-1}y e_q((x-y)u)},
\end{equation}
where $ e_q(u)$ is given by \eqref{defi:expfor}.  
\end{Theorem}

\begin{Theorem}\label{thm:qRossel} We have 
\begin{align}
&\sum_{n\geq 0} F^{\inv}_n(q;x,y,z|\beta)\frac{u^n}{(q;q)_n}\nonumber \\[5pt]
&\quad = \prod_{k=0}^{+\infty}
    \frac{1-x^{-1}y e_q((x-y)uq^{k+1})}{1-\beta uq^k(z-y)-x^{-1}y(1+\beta uq^k(x-z)) e_q((x-y)uq^{k+1})}, 
\end{align}
where $ e_q(u)$ is defined in \eqref{defi:expfor}. 
\end{Theorem}

The rest of this section is organized as follows. Subsection~\ref{subsec:qmajdes} presents a $q$-grammar for the $q$-maj-Eulerian polynomials \eqref{defi:qEulermaj}, as shown  in Theorem~\ref{q-maj-mult}; 
Subsection~\ref{subsec:qinvdes}  provides a $q$-grammar for the $q$-inv-Eulerian polynomials \eqref{defi:qEulerinv}, as described in  Theorem~\ref{q-inv-mult}; Based on Theorem~\ref{q-inv-mult}, Subsection~\ref{subsec:qcalinvdes} delivers  a grammatical derivation of Theorem \ref{thm:qEuler}. 
Subsection~\ref{subsec:qroselle}  introduces a  $q$-grammar for the $q$-inv-Roselle polynomials \eqref{def:qRoselle} (see Theorem~\ref{q-Roselle-thm}). Subsequently, relying on the result of Theorem~\ref{q-Roselle-thm},   Subsection~\ref{subsec:qcalroselle} derives  Theorem \ref{thm:qRossel} in a purely grammatical manner within the $q$-calculus framework.

\subsection{\texorpdfstring{$q$-Grammar for $q$-maj-Eulerian polynomials}{q-Grammar for q-maj-Eulerian polynomials}} \label{subsec:qmajdes}
We have the following consequence. 

\begin{Theorem}\label{q-maj-mult}
Let $G_{\maj}$ be the $q$-grammar defined by
\begin{equation}\label{def:grammaj}
G_{\maj} = \bigl(\{x,y\}, \{ x_j \mapsto q^j x_0 y_0,\; y_j \mapsto q^j x_0 y_0 \}, \LPO\bigr).
\end{equation}
Let $D$ be the $q$-derivative associated with $G_{\maj}$ and define the evaluation map $\phi$ by
$\phi(x_j)=xq^j$ and $\phi(y_j)=yq^j$. Then
\begin{equation*}
\phi\bigl(D^n(x_0)\bigr) = A_n^{\maj}(q;x,y).
\end{equation*}
\end{Theorem}

\begin{proof}
Let $\sigma=\sigma_1\cdots\sigma_n\in\mathfrak{S}_n$. We first define the grammatical labeling of $\sigma$. Note that in the grammatical labeling of $\sigma$, descent positions are labeled by~$y$ and ascent positions are labeled by~$x$. The subscripts of the $x$-labels are nonincreasing from left to right, and the subscripts of the $y$-labels are nondecreasing from left to right.  The weight $\omega(\sigma)$ of $\sigma$ is defined as the product of the labels in its grammatical labeling, taken according to the Letter-Priority Order (LPO).

The grammatical labeling of $\sigma$ is defined recursively below.
Successively remove the entries $n,n-1,\dots,m+1$ from $\sigma$
to obtain the increasing sequence $\sigma^{(0)}=1,2,\dots,m$.
Adjoin  $0$ at both ends of $\sigma^{(0)}$.

First, for $1\le i\le m$, label the position immediately before $\sigma^{(0)}_i$ by $x_0$,
and label the position after $\sigma^{(0)}_m$ by $y_0$. That is,
\[
\sigma^{(0)}=
\begin{array}{ccccccccccccccccccc}
0 &_{x_0}& \sigma^{(0)}_1 &_{x_0}& \sigma^{(0)}_2 &_{x_0}&
\cdots & \sigma^{(0)}_m &_{y_0}& 0
\end{array}.
\]
In this case, the weight of 
$\sigma^{(0)}$ is 
\[
\omega(\sigma^{(0)})=x^m_0y_0.
\]

Let $\sigma^{(1)}$ be the permutation obtained from $\sigma$ by removing $n,n-1,\dots,m+2$,
and suppose that $m+1$ is inserted into $\sigma^{(0)}$ at position $p$ to yield $\sigma^{(1)}$.

There are two cases:

\noindent{\bf Case 1:} If $1\le p\le m$, then position $p$ of $\sigma^{(1)}$ is labeled $x_0$
and position $p+1$ is labeled $y_0$.
The subscripts of all labels before position $p$ are increased by~$1$,
and $y_0$ is changed to $y_1$. Explicitly,
\[
\sigma^{(1)}=
\begin{array}{ccccccccccccccccccc}
0 &_{x_1}& \sigma^{(0)}_1 &\cdots&_{x_1}& m+1 &_{y_0}&
\sigma^{(0)}_p &_{x_0}& \cdots& \sigma^{(0)}_m &_{y_1}& 0
\end{array}.
\]
In this case, the weight of 
$\sigma^{(1)}$ is 
\[
\omega(\sigma^{(1)})=x^{p}_{0}x^{m-p}_1y_0y_1.
\]

\noindent{\bf Case 2:} If $p=m+1$, then position $m+1$ of $\sigma^{(1)}$ is labeled $x_0$
and position $m+2$ is labeled $y_0$, while the subscripts of all remaining labels remain unchanged. That is,
\[
\sigma^{(1)}=
\begin{array}{ccccccccccccccccccc}
0 &_{x_0}& \sigma^{(0)}_1 &_{x_0}& \sigma^{(0)}_2 &_{x_0}&
\cdots& \sigma^{(0)}_m &_{x_0}& m+1 &_{y_0}& 0
\end{array}.
\]
In this case, the weight of 
$\sigma^{(1)}$ is 
\[
\omega(\sigma^{(1)})=x^{m+1}_{0}y_0.
\]

Now let $\sigma^{(l)}$ be the permutation obtained from $\sigma$
by removing $n,n-1,\dots,m+l+1$, where $0\le l<n-m$,
and assume that $\sigma^{(l)}$ has been labeled recursively. Assume that the weight of $\sigma^{(l)}$ is 
\[
\omega(\sigma^{(l)})=q^{s(\sigma^{(l)})}\omega_1(\sigma^{(l)})\cdots\omega_{m+l+1}(\sigma^{(l)}).
\]

We now define the grammatical labeling of $\sigma^{(l+1)}$. Note that $\sigma^{(l+1)}$ has $m+l+2$ positions to be labeled. As before, descent positions are labeled by~$y$ and ascent positions are labeled by~$x$. The subscripts are determined as follows. Suppose that $m+l+1$ is inserted into $\sigma^{(l)}$ at position~$p$ to form $\sigma^{(l+1)}$. Then
we consider two cases.

\noindent{\bf Case 1:} Suppose position $p$ in the labeling of $\sigma^{(l)}$ carries the label $x_i$, i.e., $\omega_j=x_i$ for some~$j$. Then the subscripts of all labels of $\sigma^{(l)}$ before position~$p$ are increased by~$1$, and the subscripts of all $y$-labels of $\sigma^{(l)}$ after position~$p$ are increased by~$1$. Collect all resulting labels together with $x_0$ and $y_0$, and rearrange them so that the $x$-subscripts are nonincreasing from left to right and the $y$-subscripts are nondecreasing from left to right. Label $q^i$ below $\sigma_p^{(l)}$.

\noindent{\bf Case 2:} Suppose position $p$ in the labeling of $\sigma^{(l)}$ carries the label $y_i$, i.e., $\omega_j=y_i$ for some~$j$. Then the subscripts of all $y$-labels of $\sigma^{(l)}$ after position~$p$ are increased by~$1$. Again, collect all resulting labels together with $x_0$ and $y_0$, and rearrange them so that the $x$-subscripts are nonincreasing from left to right and the $y$-subscripts are nondecreasing from left to right. Label $q^i$ below $\sigma_p^{(l)}$.

Below we illustrate the grammatical labeling for the permutation
$\sigma=8\,1\,6\,2\,3\,5\,7\,4\in\mathfrak{S}_8$:
\begin{align*}
 \sigma^{(0)}&=  \begin{array}{ccccccccccccccccccc}
   0&_{x_0}&1&_{x_0}&2&_{x_0}&3&_{x_0}&4&_{y_0}&0 
     \end{array}
 \\[5pt]
 \sigma^{(1)}&=  \begin{array}{ccccccccccccccccccc}
   0&_{x_1}&1&_{x_1}&2&_{x_1}&3&_{x_0}&\fbox{\color{blue} 5}&_{y_0}&4&_{y_1}&0 
     \end{array}  \\[5pt]
\sigma^{(2)}&=  \begin{array}{ccccccccccccccccccc}   0&_{x_2}&1&_{x_1}&\fbox{\color{blue}  6}&_{y_0}&{2}&_{x_0}&3&_{x_0}&5&_{y_1}&4&_{y_2}&0 \\[5pt]
     &&&&&&q&&&&&&&&&
     \end{array}  \\[5pt]
\sigma^{(3)}&=  \begin{array}{ccccccccccccccccccc}   0&_{x_2}&1&_{x_1}&6&_{y_0}&2&_{x_0}&3&_{x_0}&5&_{x_0}&\fbox{\color{blue}7}&_{y_0}&4&_{y_3}&0 \\[5pt]
     &&&&&&q&&&&&&&&q&
     \end{array}   \\[5pt]
\sigma^{(4)}&=  \begin{array}{ccccccccccccccccccc}   0&_{x_1}&\fbox{\color{blue}8}&_{y_0}&1&_{x_0}&6&_{y_1}&2&_{x_0}&3&_{x_0}&5&_{x_0}&7&_{y_1}&4&_{y_4}&0 \\[5pt]
     &&&&q^2&&&&q&&&&&&&&q&
     \end{array}       
\end{align*}

The weight $\omega(\sigma)$ of $\sigma$ is defined as the product of all labels, read from right to left. For the above example, we have
\[
\omega(\sigma)=q^4 x_0 x_0 x_0 x_0 x_1 y_0 y_1 y_1 y_4.
\]

Recall that the morphism map $\phi$ is defined by $\phi(x_j)=xq^j$ and $\phi(y_j)=yq^j$.
It is clear that 
\[
\phi\bigl(\omega(\sigma^{(0)})\bigr)
=q^{\maj(\sigma^{(0)})} x^{\asc(\sigma^{(0)})}y^{\des(\sigma^{(0)})}
\]
and 
\[
\phi\bigl(\omega(\sigma^{(1)})\bigr)
=q^{\maj(\sigma^{(1)})} x^{\asc(\sigma^{(1)})}y^{\des(\sigma^{(1)})}.
\]
By induction, we show that for all $0\le l\le n-m$,
\begin{equation}\label{relmaj:phi}
\phi\bigl(\omega(\sigma^{(l)})\bigr)
=q^{\maj(\sigma^{(l)})} x^{\asc(\sigma^{(l)})}y^{\des(\sigma^{(l)})}.
\end{equation}

With this grammatical labeling, we now prove Theorem~\ref{q-maj-mult}.
By \eqref{relmaj:phi}, it suffices to establish the identity:
for $n\ge1$,
\begin{equation}\label{equ-maj-equiv}
D^n(x_0)=\sum_{\sigma\in\mathfrak{S}_n}\omega(\sigma),
\end{equation}
where $D$ is the $q$-derivative associated with  the $q$-grammar \eqref{def:grammaj}.

Represent a permutation $\sigma=\sigma_1\cdots\sigma_n\in\mathfrak{S}_n$
by adjoining $0$ at both ends, creating $n+1$ positions between consecutive elements.
These positions are where we insert $n+1$ to generate permutations in $\mathfrak{S}_{n+1}$.
Assume that the weight of $\sigma$ is 
\[
\omega(\sigma)=q^{s(\sigma)}\omega_1(\sigma)\cdots\omega_{n+1}(\sigma).
\]

Let $\pi^{(i)}\in\mathfrak{S}_{n+1}$ be the permutation obtained
by inserting $n+1$ into $\sigma$ at the position with label $\omega_i(\sigma)$.
By the $q$-product formula,
\begin{equation}\label{qLeibnitz-maj}
D\bigl(\omega(\sigma)\bigr)
=\sum_{i=1}^{n+1} q^{s(\sigma)}\rho\Bigl(
\omega_1(\sigma)\cdots\omega_{i-1}(\sigma)\,
R(\omega_i(\sigma))\,
\uparrow\!\bigl(\omega_{i+1}(\sigma)\cdots\omega_{n+1}(\sigma)\bigr)
\Bigr).
\end{equation}
From the above grammatical labeling, it is straightforward to verify that
\begin{align*}\label{rel-inv}
&q^{s(\sigma)}\rho\Bigl(
\omega_1(\sigma)\cdots\omega_{i-1}(\sigma)\,
R(\omega_i(\sigma))\,
\uparrow\!\bigl(\omega_{i+1}(\sigma)\cdots\omega_{n+1}(\sigma)\bigr)
\Bigr)
\nonumber\\
&\quad = q^{s(\pi^{(i)})}\omega(\pi^{(i)}).
\end{align*}
Thus \eqref{qLeibnitz-maj} holds,
which implies that \eqref{equ-maj-equiv} is valid for $n+1$.
Theorem~\ref{q-maj-mult} now follows immediately from   \eqref{relmaj:phi} and \eqref{equ-maj-equiv}. \end{proof}

\subsection{\texorpdfstring{$q$-Grammar for  $q$-inv-Eulerian polynomials }{q-Grammar for  q-inv-Eulerian polynomials }} \label{subsec:qinvdes}

\begin{Theorem}\label{q-inv-mult}
Let $G_{\inv}$ be the $q$-grammar defined by
\begin{equation}\label{def:graminv}
G_{\inv} = \bigl(\{x,y\}, \{ x_j \rightarrow q^j y_j x_{j+1} , \,  y_j \rightarrow q^{j} y_{j}x_{j+1} \}, \AIO \bigr).
\end{equation}
Let $D$ be the $q$-derivative associated with $G_{\inv}$ and define the
evaluation map $\phi$ by $
\phi(x_j)=x$ and $\phi(y_j)=y$. 
Then
\begin{equation}\label{thm:q-maj}
\phi\bigl(D^n(x_0)\bigr) = A^{\inv }_n(q;x,y).
\end{equation} 
\end{Theorem}

\begin{proof}
Let $\sigma=\sigma_1\cdots \sigma_n \in \mathfrak{S}_n$. For $1\leq i\leq n+1$,  recall that the position $i$ is said to be the position immediately before $\sigma_i$, whereas the position $n+1$ is meant to be the position after $\sigma_n$. We patch   $0$ to $\sigma$ at both ends so that there are $n+1$ positions between two adjacent elements. For $1\leq i\leq n+1$, we label the position $i$ of $\sigma \in \mathfrak{S}_n$ as follows:
\begin{itemize}
 \item If $i$ is an ascent, then label it by $x_{n+1-i}$;
  \item If $i$ is a descent, then label it  by $y_{n+1-i}$.
\end{itemize}
Below shows the grammatical labeling of a permutation $\sigma \in \mathfrak{S}_8$.   
\begin{equation*}\label{exam}
\begin{aligned}
 \sigma= & 7\quad 1\quad 2\quad 8\quad 3\quad 6\quad 5\quad 4\\[5pt]
  &{\rightarrow} \ \begin{array}{ccccccccccccccccccc}
   0&x_8&7&y_7 &1&x_6&2&x_5&8&y_4 &3&x_3&6&y_2 &5&y_1 &4&y_0&0 
     \end{array}
\end{aligned}
\end{equation*}
The weight $\omega$ of $\sigma$ is defined to be the product of all the labels from right to left. For the example above, we see that the weight of $\sigma$  is 
\[\omega(\sigma)=y_0y_1y_2x_3y_4x_5x_6y_7x_8.\]

We observe that  
\begin{equation*} 
\phi\left(\sum_{\sigma \in \mathfrak{S}_{n}} q^{\inv (\sigma)}\omega(\sigma)\right)=A^{  \inv}_n(q;x,y).
\end{equation*}
Thus,  Theorem \ref{q-inv-mult} follows once we confirm the assertion: For $n\geq 1$,
\begin{equation}\label{equ-qEulerab}
D^n (x_0)=\sum_{\sigma \in \mathfrak{S}_{n}} q^{\inv (\sigma)}\omega(\sigma),
\end{equation}
where $D$  is $q$-derivative associated with  $q$-grammar \eqref{def:graminv}. 

We proceed by induction on $n$. For $n=1$, the statement is evident. Assume that this statement holds for $n$,  that is, the relation \eqref{equ-qEulerab} is valid for $n$. To demonstrate that it also holds for $n+1$,   it suffices to show, by \eqref{equ-qEulerab}, that
\begin{equation}\label{gram-inductaa1b}
D\left(\sum_{\sigma \in \mathfrak{S}_{n}}q^{\inv (\sigma)} \omega(\sigma)\right)=\sum_{\pi \in \mathfrak{S}_{n+1}}q^{\inv (\pi)}\omega(\pi).
\end{equation}

For a permutation $\sigma \in \mathfrak{S}_n$, we define the weight 
\[\omega(\sigma)=\omega_1(\sigma)\omega_2(\sigma) \cdots \omega_{n+1}(\sigma), \]
which is the product of the right-to-left labels in the grammatical labeling of $\sigma$. 

Similarly, for the permutation $\pi  \in \mathfrak{S}_{n+1}$,  the weight is given by 
\[\omega(\pi)=\omega_1(\pi)\omega_2(\pi) \cdots \omega_{n+2}(\pi), \]
which is the product of the right-to-left labels in the grammatical labeling of $\pi$. 

Note that in this grammar, the order AIO is equivalent to KSO (Keep Sequence Order), and so by the $q$-product  formula, we have 
\[D(\omega(\sigma))=\sum_{i=1}^{n+1}  \omega_1(\sigma)\cdots \omega_{i-1}(\sigma)\, R(\omega_i(\sigma))\, \uparrow\! \bigl(\omega_{i+1}(\sigma)\cdots \omega_{n+1}(\sigma)\bigr). \] 

Next, we represent a permutation $\sigma=\sigma_1\cdots \sigma_n$ in $\mathfrak{S}_n$ by patching $0$ at both ends, resulting in  $n+1$ positions between adjacent elements. These positions allow us to insert $n+1$ into $\sigma$ to generate $n+1$ permutations in $\mathfrak{S}_{n+1}$. Let $\pi^{(i)}$ be the permutation  in $\mathfrak{S}_{n+1}$ obtained  by inserting the element $n+1$ into  $\sigma$ at the position $n+1-i$, where $0\leq i\leq n$. It is easy to check that 
\[{\inv (\pi^{(i)})}=\inv (\sigma)+i.\]
To prove \eqref{gram-inductaa1b},  it is enough to show that 
\begin{align}\label{rel-inv}
&q^{\inv (\sigma)}\omega_1(\sigma)\cdots \omega_{i-1}(\sigma)\, R(\omega_i(\sigma))\, \uparrow\! \bigl(\omega_{i+1}(\sigma)\cdots \omega_{n+1}(\sigma)\bigr) \nonumber\\
&\quad =q^{\inv (\pi^{(i)})}\omega(\pi^{(i)})=q^{\inv (\sigma)+i}\omega_1(\pi^{(i)})\omega_2(\pi^{(i)})\cdots \omega_{n+2}(\pi^{(i)}).
\end{align}

We consider  two cases:

\noindent{\bf Case 1:} Suppose $\omega_i(\sigma) = x_{i}$ (i.e., position $n+1-i$ in the grammatical labeling of $\sigma$ is assigned the label $x_{i}$). For the labeling of $\pi^{(i)}$:
\begin{itemize}
\item Position $n+1-i$ is labeled $x_{i+1}$;

\item Position $n+2-i$ is labeled $y_{i}$; 

\item All labels of $\pi^{(i)}$ at positions after $n+2-i$ are identical to those of $\sigma$; 

\item  Each label of $\pi^{(i)}$ at positions before $n+1-i$ has a subscript that is exactly one greater than the subscript of the corresponding label in $\sigma$.  

Below shows the grammatical labeling of a permutation $\pi^{(i)}$ in this case.   

\end{itemize}
\begin{align*}
&\quad \quad \begin{array}{cccccccccc}
 \sigma=0 & _{x_{n}} &\sigma_1\cdots\ &\sigma_{n-i}&  _{x_{i}} & \sigma_{n-i+1} &  \cdots & \sigma_n & _{y_0} & 0
\end{array}\\
&\Rightarrow   \begin{array}{ccccccccccccc}
\pi^{(i)}= 0 & _{x_{n+1}} &\sigma_1\cdots &\sigma_{i-1}&  _{x_{i+1}} & n+1 & _{y_{i}} &\sigma_{n-i+1} &  \cdots  & \sigma_n & _{y_0} & 0
\end{array}
\end{align*}                                                                   
Since $R(x_i) = q^{i}y_{i}x_{i+1}$, the relation \eqref{rel-inv} follows immediately.

\noindent{\bf Case 2:} Suppose $\omega_i(\sigma) = y_{i}$ (i.e., position $n+1-i$ in the grammatical labeling of $\sigma$ is assigned the label $y_{i}$). Then $\pi^{(i)}$ has $x_{i+1}$ at position $n+1-i$, $y_{i}$ at position $n+2-i$ in its labeling. As in Case 1, labels of $\pi^{(i)}$ for all positions after $n+2-i$ coincide with those of  $\sigma$, and for each position before $n+1-i$, the subscript of the label of $\pi^{(i)}$ is one greater than the subscript of the corresponding label of $\sigma$. 
Below shows the grammatical labeling of a permutation $\pi^{(i)}$ in this case.   
\begin{align*}
&\quad \quad \begin{array}{cccccccccc}
 \sigma=0 & _{x_{n}} &\sigma_1\cdots\ &\sigma_{n-i}&  _{y_{i}} & \sigma_{n-i+1} &  \cdots & \sigma_n & _{y_0} & 0
\end{array}\\
&\Rightarrow   \begin{array}{ccccccccccccc}
\pi^{(i)}= 0 & _{x_{n+1}} &\sigma_1\cdots &\sigma_{i-1}&  _{x_{i+1}} & n+1 & _{y_{i}} &\sigma_{n-i+1} &  \cdots  & \sigma_n & _{y_0} & 0
\end{array}
\end{align*}    
 Given $R(y_i) = q^{i}y_{i}x_{i+1}$, the relation \eqref{rel-inv} is likewise valid for this case,  \eqref{gram-inductaa1b}  follows at once.
 
Summing the results from these two cases shows that the assertion \eqref{rel-inv} holds, and thus \eqref{gram-inductaa1b} is valid for $n+1$. Theorem  \ref{q-inv-mult} follows immediately from \eqref{gram-inductaa1b}. This completes the proof.  
\end{proof}

\subsection{\texorpdfstring{$q$-grammatical calculus for $q$-inv-Eulerian polynomials}{q-grammatical calculus for q-inv-Eulerian polynomials}} \label{subsec:qcalinvdes}

Based on the framework of $q$-grammatical calculus developed in Section \ref{secqgram}, this subsection uses Theorem \ref{q-inv-mult} to derive a grammatical proof of Theorem \ref{thm:qEuler}.

\begin{Proposition}\label{lem:xyminus} 
Let $G_{\inv}$ be the $q$-grammar defined in Theorem \ref{q-inv-mult}, and let $D$ be the  $q$-derivative associated with $G_{\inv}$.  
For $n\geq 1$ and $i\geq 0$, 
\begin{align}\label{Dqn-Xi-minus}
D^n(x_i^{-1}) &= -q^{ni}  x_i^{-1}y_i(x_{i+1}-y_{i+1})^{n-1},\\[5pt]
D^n(y_i^{-1}) &= -q^{\binom{n}{2}+ni} \prod_{k=1}^{n-1}(y_{i+k}-x_{i+k})x_{i+n}y^{-1}_{i+n}.   \label{Dqn-yi-minus}
\end{align}
\end{Proposition}

\begin{proof} Observe that 
\begin{equation}
D\left(x_i^{-1}\right)=-x_i^{-1}R(x_i)\uparrow (x_i^{-1})
=-q^ix_i^{-1}y_{i}x_{i+1}x_{i+1}^{-1}=-q^ix_i^{-1}y_i, \label{eq:Dqn-xi-minusaa}
\end{equation}
and 
\begin{equation}
D\left(y_i^{-1}\right)=-y_i^{-1}R(y_i)\uparrow (y_i^{-1})
=-q^iy_i^{-1}y_{i}x_{i+1}y_{i+1}^{-1}=-q^ix_{i+1}y_{i+1}^{-1}.  \label{eq:Dqn-yi-minusbb}
\end{equation}
Moreover, 
\begin{align*} 
D\left(x_i^{-1}y_i\right)&=\operatorname{\AIO} \bigl(R(x_i^{-1})y_{i+1}+x_i^{-1}R(y_i)\bigr) \nonumber \\[5pt]
&=-q^{i}x_{i}^{-1}y_{i}y_{i+1}+q^{i} x_{i}^{-1}y_{i}x_{i+1} \nonumber \\[5pt]
&=q^{i}x_{i}^{-1}y_i(x_{i+1}-y_{i+1})
\end{align*}
and 
\begin{align*} 
D\left(x_{i+1}y_{i+1}^{-1}\right)&=\operatorname{\AIO}\bigl(R(x_{i+1})y_{i+2}^{-1}+x_{i+1}R(y_{i+1}^{-1})\bigr) \nonumber \\[5pt]
&=q^{i+1}y_{i+1}x_{i+2}y_{i+2}^{-1}-q^{i+1} x_{i+1}x_{i+2}y^{-1}_{i+2} \nonumber \\[5pt]
&=q^{i+1}(y_{i+1}-x_{i+1})x_{i+2}y_{i+2}^{-1}.
\end{align*}
Since $D(x_i-y_i)=0$ for all $i\ge 0$ in the $q$-grammar defined by \eqref{def:graminv}, by induction, we derive that for $n\geq 0$, 
\begin{align} 
D^n\left(x_i^{-1}y_i\right)&=q^{ni}x_{i}^{-1}y_i(x_{i+1}-y_{i+1})^{n},\\[5pt]
D^n\left(x_{i+1}y_{i+1}^{-1}\right)&=q^{ni+{n+1\choose 2}}\prod_{k=1}^{n}(y_{i+k}-x_{i+k})x_{i+n+1}y_{i+n+1}^{-1}.
\end{align}
Combining these identities with \eqref{eq:Dqn-xi-minusaa} and \eqref{eq:Dqn-yi-minusbb} immediately establishes the identities \eqref{Dqn-Xi-minus} and \eqref{Dqn-yi-minus}, respectively.  
\end{proof}

\begin{Corollary}\label{coro:xyminus} 
Let $G_{\inv}$ be the $q$-grammar given in Theorem \ref{q-inv-mult} with $\phi$ the corresponding evaluation map. Then for $i\geq 0$, 
\begin{equation}
\phi({\rm Gen}^{(G_{\inv})}_q(x^{-1}_i;u))=\frac{1-x^{-1}y e_q((x-y)uq^i)}{x-y}
\end{equation}
and 
\begin{equation}
\phi({\rm Gen}^{(G_{\inv})}_q(y^{-1}_i;u))=\frac{1-xy^{-1} E_q((y-x)uq^i)}{y-x},
\end{equation}
where $ e_q(u)$ and $ E_q(u)$ are as defined in \eqref{defi:expfor} and \eqref{defi:Expfor}, respectively.
\end{Corollary}

\begin{proof}
Applying \eqref{Dqn-Xi-minus}, we deduce that 
\begin{align*}
\phi({\rm Gen}^{(G_{\inv})}_q(x_i^{-1};u))&=\sum_{n=0}^{+\infty} \phi(D^n(x_i^{-1})) \frac{u^n}{(q; q)_n} \nonumber \\[5pt]
&=x^{-1}-\sum_{n=1}^{+\infty} x^{-1}y (x-y)^{n-1} q^{ni} \frac{u^n}{(q; q)_n} \nonumber \\[5pt]
&=x^{-1}-\frac{x^{-1}y}{x-y}\left(e_q((x-y)uq^i)-1\right) \nonumber \\[5pt]
&=\frac{1-x^{-1}y e_q((x-y)uq^i)}{x-y}.
\end{align*}

Similarly, using \eqref{Dqn-Xi-minus}, we obtain 
\begin{align*}
\phi({\rm Gen}^{(G_{\inv})}_q(y_i^{-1};u))&=\sum_{n=0}^{+\infty} \phi(D^n(y_i^{-1})) \frac{u^n}{(q; q)_n} \nonumber \\[5pt]
&=y^{-1}-\sum_{n=1}^{+\infty} q^{{n\choose 2}+ni} xy^{-1} (y-x)^{n-1}  \frac{u^n}{(q; q)_n} \nonumber \\[5pt]
&=y^{-1}-\frac{xy^{-1}}{y-x}\left(E_q((y-x)uq^i)-1\right) \nonumber \\[5pt]
&=\frac{1-xy^{-1} E_q((y-x)uq^i)}{y-x}. 
\end{align*}
\end{proof}

\begin{Theorem} \label{q-inv-mult-grammar} Let $G_{\inv}$ be the $q$-grammar defined in Theorem \ref{q-inv-mult} and let $\phi$ be  the evaluation map as defined therein. Then 
\begin{equation}\label{eq:q-inv-mult-grammar}
\phi({\rm Gen}^{(G_{\inv})}_q(x_i;u))=\frac{x-y}{1-x^{-1}y e_q((x-y)uq^i)}.
\end{equation}

\end{Theorem}

 \begin{proof}
It is straightforward to check that $\phi$ in Theorem \ref{q-inv-mult} is a master-linear evaluation and the grammar $G_{\inv}$ defined in \eqref{def:graminv} satisfies  \eqref{eq:rule}. Consequently, by Theorem \ref{thm:product}, we deduce that
\begin{align*}
1&=\phi({\rm Gen}^{(G_{\inv})}_q(x_i  x_i^{-1};u)) = \phi({\rm Gen}^{(G_{\inv})}_q(x_i;u))\phi({\rm Gen}^{(G_{\inv})}_q(x_i^{-1};u)). 
\end{align*}
Hence, applying Corollary \ref{coro:xyminus}, we arrive at \eqref{eq:q-inv-mult-grammar}. 
\end{proof}

By combining  Theorem \ref{q-inv-mult} and Theorem \ref{q-inv-mult-grammar}, we recover Theorem \ref{thm:qEuler}. 

\subsection{\texorpdfstring{$q$-Grammar for cycle $q$-Roselle polynomials}{q-Grammar for cycle q-Roselle polynomials}} \label{subsec:qroselle}

\begin{Theorem} \label{q-Roselle-thm}  Let $G_{\cyc}=(\{x,y,z,e\},  R, \KSO)$ be the $q$-grammar, where
 \[R = \{   x_j \rightarrow q^j y_j x_{j+1} , 
\  y_j \rightarrow q^j y_j x_{j+1} ,
\  z_j \rightarrow q^j y_j x_{j+1} ,
\  e_j \rightarrow \beta q^j e_j z_{j+1} 
  \} .\] 
Let $D$ be the $q$-derivative associated with $G_{\cyc}$ and define the evaluation map $\phi$ by 
$\phi(x_j)=x$, $\phi(y_j)=y$,  $\phi(z_j)=z$ and $\phi(e_j)=e.$ Then  
$$
	\phi(D^n(e_0)) =  e F^{\inv}_n(q;x,y,z|\beta).
$$
\end{Theorem}

\proof Let $\sigma=\sigma_1\cdots \sigma_n \in \mathfrak{S}_n$. For $1\leq i\leq n+1$,  recall that the position $i$ is said to be the position immediately before $\sigma_i$, whereas the position $n+1$ is meant to be the position after $\sigma_n$. We patch   $0$ to $\sigma$ at both ends so that there are $n+1$ positions between two adjacent elements. For $1\leq i\leq n+1$, we label the position $i$ of $\sigma \in \mathfrak{S}_n$ as follows:
\begin{itemize}
 \item If $i$ is a descent, then label it by $y_{n+1-i}$;
  \item If $\sigma_i$ is an isolated element, then label $i$  by $z_{n+1-i}$;  
   \item If $i$ is a non-isolated ascent, then label it  by $x_{n+1-i}$;  

   \item If $i=n+1$, then label it by $e_0$. 
\end{itemize}
Below  shows the grammatical labeling of a permutation $\sigma= 5\, 4\,1\, 2\,7\, 3\,6\, 9\,  8 \in \mathfrak{S}_9$.   
\begin{equation*}\label{exama}
\begin{array}{rrrrrrrrrrrrrrrr}
   0&y_9\,5&x_8 \,4&x_7\,1&|&z_6\,2&|&y_5 \,7&x_4\,3&|&z_3\,6&|&y_2 \,9&x_1\,8&e_0&0 
     \end{array}
\end{equation*}
The weight $\omega$ of $\sigma$ is defined to be the product of all these labels taken from right to left. For the example above, its weight is given by 
\[\omega(\sigma)=e_0\, x_1\, y_2\, z_3\, x_4\, y_5\, z_6\, x_7\, x_8\, y_9.\]
Note that 
\begin{equation*} 
\phi\left(\sum_{\sigma \in \mathfrak{S}_{n}}q^{\inv(\sigma)} \beta^{\RLmin(\sigma)}\omega(\sigma)\right)=eF^{\inv}_n(q;x,y,z|\beta),
\end{equation*}
and consequently, Theorem \ref{q-Roselle-thm}  follows immediately upon verifying the following assertion:  For $n\geq 1$,
\begin{equation}\label{equ-qEulera}
D^n (e_0)=\sum_{\sigma \in \mathfrak{S}_{n}}q^{\inv(\sigma)}\beta^{\RLmin(\sigma)}\omega(\sigma).
\end{equation}

We proceed by induction on $n$. For $n=1$, the statement is evident. Assume that \eqref{equ-qEulera} holds for $n$. To demonstrate that it also holds for $n+1$, it suffices to show, by \eqref{equ-qEulera}, that
\begin{equation}\label{gram-inductaa1}
D\left(\sum_{\sigma \in \mathfrak{S}_{n}}q^{\inv(\sigma)}\beta^{\RLmin(\sigma)}\omega(\sigma)\right)=\sum_{\pi \in \mathfrak{S}_{n+1}}q^{\inv(\pi)}\beta^{\RLmin(\pi)}\omega(\pi).
\end{equation}
 
For a permutation $\sigma \in \mathfrak{S}_n$, we define the weight 
\[\omega(\sigma)=\omega_1(\sigma)\omega_2(\sigma) \cdots \omega_{n+1}(\sigma), \]
which is the product of the right-to-left labels in the grammatical labeling of $\sigma$. 

Similarly, for the permutation $\pi  \in \mathfrak{S}_{n+1}$,  the weight is given by 
\[\omega(\pi)=\omega_1(\pi)\omega_2(\pi) \cdots \omega_{n+2}(\pi), \]
which is the product of the right-to-left labels in the grammatical labeling of $\pi$. 

By the $q$-product formula and note that $\rho$  is $\KSO$ (Keep Sequence Order), we have 
\[D(\omega(\sigma))=\sum_{i=1}^{n+1}  \omega_1(\sigma)\cdots \omega_{i-1}(\sigma)\, R(\omega_i(\sigma))\, \uparrow\! \bigl(\omega_{i+1}(\sigma)\cdots \omega_{n+1}(\sigma)\bigr). \]

Next,  we represent a permutation $\sigma=\sigma_1\cdots \sigma_n$ in $\mathfrak{S}_n$ by patching   $0$  at both ends, resulting in  $n+1$ positions between adjacent elements.  These positions allow us to insert   $n+1$ into $\sigma$ to generate $n+1$ permutations in $\mathfrak{S}_{n+1}$. Let $\pi^{(i)}$ be permutation  in $\mathfrak{S}_{n+1}$ obtained  by inserting the element $n+1$ into  $\sigma$ at the position $n+1-i$, where $0\leq i\leq n$. To prove \eqref{gram-inductaa1},  it is enough to show that 
\begin{align}\label{pfbb}
&\beta^{\operatorname{RLmin}(\sigma)}q^{\inv (\sigma)}\omega_1(\sigma)\cdots \omega_{i-1}(\sigma)\, R(\omega_i(\sigma))\, \uparrow\! \bigl(\omega_{i+1}(\sigma)\cdots \omega_{n+1}(\sigma)\bigr)\nonumber \\[5pt]
&=\beta^{\operatorname{RLmin}(\pi^{(i)})}q^{\inv (\pi^{(i)})}\omega_1(\pi^{(i)})\omega_2(\pi^{(i)})\cdots \omega_{n+2}(\pi^{(i)}).
\end{align}

We consider  the following four cases:

\noindent{\bf Case 1:}  Suppose $\omega_i(\sigma) = x_{i}$ (i.e., position $n+1-i$ in the grammatical labeling of $\sigma$ is assigned the label $x_{i}$). Then in the labeling of $\pi^{(i)}$, position $n+1-i$ is labeled $x_{i+1}$ and position $n+2-i$ is labeled $y_{i}$; labels of $\pi^{(i)}$ after position $n+2-i$ coincide with those of  $\sigma$,  and each label of $\pi^{(i)}$ before position $n+1-i$ has a subscript one greater than the corresponding subscript in the labeling of $\sigma$. Furthermore, 
\begin{equation}\label{rose-rel}
\operatorname{RLmin}(\pi^{(i)}) = \operatorname{RLmin}(\sigma)\quad  {\text and} \quad  \inv(\pi^{(i)}) = \inv(\sigma) + i.
\end{equation}
Since $R(x_{i}) = q^{i}y_{i}x_{i+1}$, relation \eqref{pfbb} follows immediately.

\noindent{\bf Case 2:}  Suppose $\omega_i(\sigma) = y_{i}$ (i.e., position $n+1-i$ in $\sigma$'s grammatical labeling is $y_{i}$). Then $\pi^{(i)}$ has $x_{i+1}$ at position $n+1-i$, $y_{i}$ at position $n+2-i$ in its labeling. As in Case 1, labels of $\pi^{(i)}$ for all positions after $n+2-i$ coincide with those of  $\sigma$,  and for each position before $n+1-i$,   the subscript of the label of    $\pi^{(i)}$ is one greater than the subscript of the corresponding label of $\sigma$. Given $
 R(y_{i}) = q^{i}y_{i}x_{i+1}$
 and relation \eqref{rose-rel} is likewise valid for this case,  \eqref{pfbb}  follows at once.

\noindent{\bf Case 3:} If $\omega_i(\sigma)=z_{i}$ (position $n+1-i$ in $\sigma$'s grammatical labeling is $z_{i}$),  the labeling of $\pi^{(i)}$ satisfies $x_{i+1}$ at  $n+1-i$, $y_{i}$ at  $n+2-i$,  labels coincide with $\sigma$ for position $>n+2-i$, and a subscript increment of 1 for labels at position $<n+1-i$ (as in Case 1). Since $R(z_{i}) = q^{i}y_{i}x_{i+1}$ and   \eqref{rose-rel} is also valid here, \eqref{pfbb} is an immediate consequence.

\noindent{\bf Case 4:} If position $n+1-i$ in the labeling of $\sigma$ is assigned $e_0$, then $i=0$ and $\omega_{1}(\sigma)=e_0$. In the labeling of $\pi^{(n+1)}$, position $n+1$ is labeled $z_1$ and position $n+2$ is labeled $e_0$; for every position preceding $n+1$, the subscript of the label of $\pi^{(0)}$ is one greater than the subscript of the corresponding label of $\sigma$.  We have $\RLmin(\pi^{(0)}) = \RLmin(\sigma)+1$ and $\inv(\pi^{(0)}) = \inv(\sigma)$. Together with $R(e_0)=\beta e_0z_1$, relation \eqref{pfbb} follows.

Summing the four cases proves \eqref{pfbb}, so \eqref{equ-qEulera} holds for 
$n+1$, and Theorem \ref{q-Roselle-thm} follows accordingly.   \qed

\subsection{\texorpdfstring{$q$-grammatical calculus for cycle $q$-Roselle polynomials}{q-grammatical calculus for cycle q-Roselle polynomials} } \label{subsec:qcalroselle}  This subsection is devoted to providing a grammatical proof of Theorem \ref{thm:qRossel}. By  Theorem \ref{q-Roselle-thm}, it suffices to establish the following consequence.    

\begin{Theorem} \label{q-inv-cycle-grammar} Let $G_{\cyc}$ be the $q$-grammar defined in Theorem \ref{q-Roselle-thm} and let $\phi$ be  the evaluation map as defined therein. Then 
\begin{align*}
    \phi({\rm Gen}_q^{(G_{\cyc})}(e_0 ;u))=e\prod_{k=0}^{+\infty}
    \frac{1}{1-\beta uq^k\phi({\rm Gen}_q^{(G_{\cyc})}(z_1 ;uq^k))},
\end{align*}
where 
\begin{align}
\phi({\rm Gen}_q^{(G_{\cyc})}(z_1 ;uq^k))=\frac{z-y+x^{-1}y(x-z)e_q((x-y)uq^{k+1})}{1-x^{-1}y e_q((x-y)uq^{k+1})}. 
\end{align}

\end{Theorem}
  
\begin{proof}

Using Proposition \ref{pro:q_gramma}, we have 
\begin{align*}
D_q{\rm Gen}_q^{(G_{\cyc})}(e_0;u) &=
{\rm Gen}_q^{(G_{\cyc})}(D(e_0);u).
\end{align*}
By definition, we observe that
\[{\rm Gen}_q^{(G_{\cyc})}(D(e_0);u)={\rm Gen}_q^{(G_{\cyc})}(\beta e_0z_1;u)=\beta{\rm Gen}_q^{(G_{\cyc})}(e_0z_1;u)\]
and 
\[D_q{\rm Gen}_q^{(G_{\cyc})}(e_0;u)=\frac{{\rm Gen}_q^{(G_{\cyc})}(e_0 ;u)-{\rm Gen}_q^{(G_{\cyc})}(e_0 ;uq)}{u}.\]
Therefore, 
\begin{equation*}
{\rm Gen}_q^{(G_{\cyc})}(e_0 ;u)-{\rm Gen}_q^{(G_{\cyc})}(e_0 ;uq)
=\beta u{\rm Gen}_q^{(G_{\cyc})}(e_0z_1 ;u).
\end{equation*}
Applying the evaluation map $\phi$ to both sides, we obtain 
\begin{equation}\label{pf:gfcycle}
\phi({\rm Gen}_q^{(G_{\cyc})}(e_0 ;u))-\phi({\rm Gen}_q^{(G_{\cyc})}(e_0 ;uq))
=\beta u\phi({\rm Gen}_q^{(G_{\cyc})}(e_0z_1 ;u)).
\end{equation}
Note that  $\phi$ in Theorem \ref{q-Roselle-thm} is a master-linear evaluation and the grammar $G_{\cyc}$ defined in~Theorem \ref{q-Roselle-thm} satisfies condition~\eqref{eq:rule}.
By Theorem~\ref{thm:product}, we obtain the multiplicative identity: 
\begin{align}\label{pf:gfcycleb}
\phi({\rm Gen}_q^{(G_{\cyc})}(e_0z_1 ;u))=\phi({\rm Gen}_q^{(G_{\cyc})}(e_0;u))\phi({\rm Gen}_q^{(G_{\cyc})}(z_1;u)).
\end{align}
Substituting \eqref{pf:gfcycleb} into \eqref{pf:gfcycle} yields  
\begin{align*}
    \phi({\rm Gen}_q^{(G_{\cyc})}(e_0 ;u))
    &=\phi({\rm Gen}_q^{(G_{\cyc})}(e_0 ;uq))\cdot\frac{1}{1-\beta u\phi({\rm Gen}_q^{(G_{\cyc})}(z_1 ;u))}
    \\&=
       e\prod_{k=0}^{+\infty}
    \frac{1}{1-\beta uq^k\phi({\rm Gen}^{(G_{\cyc})}_q(z_1 ;uq^k))}.
\end{align*}
From Theorem \ref{q-inv-mult-grammar}, we have 
\begin{align}
\phi({\rm Gen}^{(G_{\cyc})}_q(z_1 ;uq^k))&=\phi({\rm Gen}^{(G_{\inv})}_q(x_1 ;uq^k))-x+z \nonumber \\[5pt]
&=\frac{x-y}{1-x^{-1}y e_q((x-y)uq^{k+1})}-x+z \nonumber  \\[5pt]
&=\frac{z-y+x^{-1}y(x-z)e_q((x-y)uq^{k+1})}{1-x^{-1}y e_q((x-y)uq^{k+1})}. 
\end{align}
This completes the proof. 
\end{proof}

\section{\texorpdfstring{$q$-Grammar for Andr\'e permutations and their generalizations}{q-Grammar for Andr\'e permutations and their generalizations}} \label{sec:qAndre}

This section aims to provide $q$-analogues of the grammar established by Dumont  \cite{Dumont-1996} for Andr\'e polynomials. Recall that the Andr\'e polynomials are defined in terms of 0-1-2 increasing trees.
An increasing tree on $[n]:=\{1,2,\ldots, n\}$ is a rooted tree with vertex set $[n]$ in which the
labels of the vertices are increasing along any path from the root. Note that $1$ is the
root. A 0-1-2 increasing tree is an increasing tree in which the degree of any vertex is
at most two. The degree of a vertex in a rooted tree is meant to be the number of its
children. Given a 0-1-2 increasing tree $T$, let $l(T)$ denote the number of leaves of $T$,
and let $u(T)$ denote the number of vertices of $T$ with degree $1$. The Andr\'e polynomial
$E_n(x,y)$ is defined by
\[
E_n(x,y) = \sum_{T} x^{l(T)} y^{u(T)},
\]
where the sum ranges over 0-1-2 increasing trees on $[n]$.

Setting $x = y = 1$, $E_n(x,y)$ reduces to the $n$-th Euler number $E_n$, which counts
0-1-2 increasing trees on $[n]$ as well as alternating permutations on $[n]$ and Andr\'e permutations on $[n]$, see \cite{FoataSchutzenberger1973Nombres, kuznetso1994increasing, stanley2010survey}.

Dumont \cite{Dumont-1996} found the grammar \eqref{gram-Andretree}  to generate the Andr\'e polynomials $E_n(x,y)$ and 
Chen and Fu \cite{Chen-Fu-2017}  provided a grammatical derivation of the generating function formula for $E_n(x,y)$, see Appendix \ref{sec:survey} for details. 

In this section, we'll provide $q$-analogues of the grammar \eqref{gram-Andretree}. It is natural to define $q$-Andr\'e polynomials via Andr\'e permutations, incorporating statistics for descents and inversions.

Andr\'e permutations were first introduced by Foata and Sch\"utzenberger and further studied by Strehl \cite{Str74} and Foata and Strehl \cite{FSt74, FSt76}.   For clarity, we will work with permutations of length $n$ for which each permutation is a sequence of $n$ distinct integers not necessarily from 1 to $n$.  The empty word $e$ and any single-letter word are defined as both {\it Andr\'e I permutations} and {\it Andr\'e II permutations}.
For a permutation $\sigma=\sigma_1\sigma_2\cdots \sigma_n$ ($n\geq 2$) of length $n$, we decompose it as $\sigma=\tau\,\min(\sigma)\,\tau'$. Here $\sigma$ is the concatenation of a left factor~$\tau$, followed by the minimum letter  $\min(\sigma)$, and a right factor $\tau'$.  Then, $\sigma$ is called an {\it Andr\'e I permutation} (resp. {\it Andr\'e II permutation}) if both $\tau$ and $\tau'$ are   Andr\'e I permutations (resp. Andr\'e II permutations), and the maximum letter of the subword $\tau\tau'$ lies in $\tau'$ (resp. the minimum letter of $\tau\tau'$ lies in $\tau'$).
We denote by $\AndI_n$ the set of all Andr\'e I permutations of $[n]:=\{1,2,\ldots, n\}$, and by $\AndII_n$ the set of all Andr\'e II permutations of $[n]$. This inductive definition immediately establishes a connection to the Euler numbers, as it can be shown that the number of Andr\'e I permutations and Andr\'e II permutations on the set $[n]$ are equal, i.e., $E_n=|\AndI_n| = |\AndII_n|$. 

\smallskip
Andr\'e I permutations for $n\leq 5$ are listed below:

$n=1$:\quad 1;\qquad $n=2$:\quad 12;\qquad
$n=3$:\quad 123, 213;

$n=4$:\quad 1234, 1324, 2314, 2134, 3124;

$n=5$:\quad 12345, 12435, 13425, 23415, 13245, 14235, 34125,
24135,\hfil\break
\indent\hphantom{$n=5$:\quad}23145, 21345, 41235, 31245, 21435, 32415, 41325, 31425.

\smallskip
Andr\'e II permutations for $n\leq 5$ are listed below:

$n=1$:\quad 1;\qquad $n=2$:\quad 12;\qquad
$n=3$:\quad 123, 312;

$n=4$:\quad 1234, 1423, 3412, 4123, 3124;

$n=5$:\quad 12345, 12534, 14523, 34512, 15234, 14235, 34125,
45123,\hfil\break
\indent\hphantom{$n=5$:\quad}35124, 51234, 41235, 31245, 51423, 53412, 41523, 31524.

We introduce the following polynomials defined on Andr\'e permutations, which incorporate statistics for descents and inversions.
\begin{align}
F^{I}_n(q;t) &= \sum_{\sigma \in \AndI_n} t^{\des(\sigma)} q^{\inv{(\sigma)}},\\[5pt]
F^{I\!I}_n(q;t)  &= \sum_{\sigma \in \AndII_n} t^{\des(\sigma)} q^{\inv{(\sigma)}},
\end{align}
where $\inv{(\sigma)}$ counts the number of inversions of $\sigma$ and $\des(\sigma)$ counts the number of descents of $\sigma$; see the beginning of Section \ref{sec:qgrammars} for their definitions. 

We have
\begin{align*}
F^{I}_1(q;t)  &=F^{I}_2(q;t) =t,\\[3pt]
F^{I}_3(q;t)  &=t + qt^2,\\[3pt]
F^{I}_4(q;t)  &=t + (2q + 2q^2 )t^2,\\[3pt]
F^{I}_5(q;t)  &=t + (3q+4q^2+3q^3+q^4) t^2 + (q^2+q^3+2q^4) t^3,
\end{align*}

and
\begin{align*}
F^{I\!I}_1(q;t)  &=F^{I\!I}_2(q;t)  =t,\\[3pt]
F^{I\!I}_3(q;t)  &=t + q^2t^2,\\[3pt]
F^{I\!I}_4(q;t)  &=t +  (2q^2+q^3+q^4)t^2,\\[3pt]
F^{I\!I}_5(q;t)  &=t+(2q^6 + q^5 + 3q^4 + 2q^3 + 3q^2)t^2 + (q^8 + q^6 + q^5 + q^4) t^3. 
\end{align*}

It is well known that there exists a bijection $\Psi$ between the set of permutations on $[n]$ and the set of increasing binary trees on $[n]$, see \cite[Chapter 1]{Stanley-EC1-2012}. Recall that an increasing binary tree is an increasing tree in which each vertex has at most two children, and the two children are distinguished as left and right (i.e., the tree is ordered). 

\begin{Definition}[The map $\Psi$] \label{defi:mappsi}
Let $\pi=\pi_1\pi_2\cdots \pi_n$  be a sequence of $n$ distinct letters not necessarily from $1$ to $n$. Define a binary tree $T $ as follows. If $\pi = \emptyset$, then $T  = \emptyset$. If $\pi \neq \emptyset$, then let $i$ be the least letter of $\pi$. Thus $\pi$ can be factored uniquely in the form $\pi = \sigma i \tau$. Now let $i$ be the root of $T $, and let $T_\sigma$ and $T_{\tau}$ be the left and right subtrees obtained by removing $i$ (see 
Figure~\ref{fig:inddef}). This yields an inductive definition of $T $.   

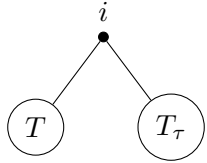
\begin{figure}[ht]
    \centering
    \begin{tikzpicture}[
    scale=0.6,
    vertex/.style={shape=circle, draw, inner sep=1.3pt, fill=black},
    subtree/.style={shape=circle, draw},
    sibling distance=1cm,
    level distance=20mm,
    auto
    ]
    \node[vertex, label=above:$i$] (i) at (0,0) {};
    \node[subtree] at (-1.5,-2) (left) {$T_\sigma$};
    \node[subtree] at (1.5,-2) (right) {$T_{\tau}$};
    \draw (i) -- (left);
    \draw (i) -- (right);
    \end{tikzpicture}
    \caption{An inductive definition of $T$.} \label{fig:inddef}
\end{figure}

\end{Definition}

As noted by Stanley \cite[Chapter 1]{Stanley-EC1-2012}, numerous permutation statistics, such as descents, double descents, peaks, and valleys, correspond to statistics on increasing binary trees under the bijection $\Psi$.  We find that the inversion number of permutations corresponds to the following inversion statistic defined on increasing binary trees under the bijection $\Psi$.  It should be pointed out that this inversion statistic on increasing binary trees is different from the notion of inversion in rooted labeled trees (see, e.g.,  Mallows and Riordan \cite{MallowsRiordan1968} and Gessel \cite{GesselSaganYeh1995}).

 \begin{Definition}\label{defi-inv-3Andre} Let $T$ be an increasing binary tree on the set $[n]$. An inversion in $T$ is a pair of vertices $(i,j)$ where $i>j$, and either:
\begin{itemize}
\item[(1)] $j$ lies to the right of the path from the root {\rm (}labeled $1${\rm )} to $i$, or

\item[(2)]  $j$ is on the path from the root {\rm (}labeled $1${\rm )} to $i$, and the left child of $j$ is contained in this path.
\end{itemize}
\end{Definition}

    Let $\inv _i(T)$ denote the number of vertices $j$ of $T$ such that $(i,j)$ is an inversion of $T$, and the inversion number of $T$ is defined by 
    \begin{equation}\label{invAndreiia}
    \inv(T)= \sum_{i=1}^n \inv _i(T). 
    \end{equation}

For example, the tree $T_1$ given in Figure \ref{figAnd-ary} has two inversions: $(2,1)$ and $(3,1)$, so $\inv(T_1)=2$.
In contrast, the tree $T_2$ given in Figure \ref{figAnd-ary} has four inversions: $(3,1)$, $(3,2)$, $(4,1)$, and $(4,2)$, so $\inv(T_2)=4$.

Applying the bijection $\Psi$ to the trees $T_1$ and $T_2$ in Figure \ref{figAnd-ary}, we obtain 
\[\Psi(T_1)=2314 \quad  \text{and} \quad \Psi(T_2)=3412.\]
It is straightforward to verify that $\inv(2314)=2$ and $\inv(3412)=4$, which is consistent with the tree inversion statistics.

As observed by Foata and the first author \cite{FH01}, when restricted to Andr\'e  permutations, the bijection $\Psi$ induces a bijection sending Andr\'e permutations to special cases of increasing binary trees, which we refer to as Andr\'e trees.

\begin{Definition}
An Andr\'e I tree is an increasing binary tree satisfying the maxima of the sibling subtrees are in increasing order {\rm (}by convention, maximum of empty tree is $0${\rm )}, whereas, an Andr\'e II tree is an increasing binary tree satisfying the minima of $k$ sibling subtrees are in decreasing order {\rm (}by convention,  minimum of empty tree is $+\infty${\rm )}.
 
\end{Definition}

The sets of Andr\'e I trees on $[n]$ and Andr\'e II trees on $[n]$ are denoted by $\mathcal{T}_n^{I}$ and $\mathcal{T}_n^{I\!I}$, respectively.   

It is easy to see that Andr\'e I trees can be derived from  0-1-2-increasing trees by requiring the maxima of the two sibling subtrees to be in increasing order,  whereas Andr\'e II trees can be derived from  0-1-2-increasing trees by requiring the minima of the two sibling subtrees to be in decreasing order. With the conventions that the maximum of an empty tree is $0$ and the minimum of an empty tree is $+\infty$. It follows that 
\[E_n=|\mathcal{T}_n^{I}|=|\mathcal{T}_n^{I\!I}|.\]

Using the bijection $\Psi$, it is straightforward to verify that Andr\'e I trees on $[n]$ are in bijection with Andr\'e I permutations on $[n]$, and Andr\'e II trees on $[n]$ are in bijection with Andr\'e II permutations on $[n]$. Moreover, the descent number of an Andr\'e permutation is determined by the number of leaves of the corresponding Andr\'e tree.

\begin{figure}
\centering
\begin{subfigure}[t]{0.3\textwidth}
\centering
\subcaption*{Andr\'e I tree $T_1$}
\begin{tikzpicture}
[vertex/.style={shape=circle, draw, inner sep=1.5pt, fill=black},
subtree/.style={shape=ellipse, draw,minimum width=7mm, minimum height=.5cm},
every fit/.style={ellipse,draw,inner sep=-2pt},
sibling distance=7mm,level distance=7mm,
leaf/.style={label={[name=#1]below:$ $}},auto]
\node [vertex, label=0:{1}]{}
child { node [vertex, label=0:{2}]{}
child { {}    edge from parent [white] }
child { {}   edge from parent [white] }
child { node [vertex, label=0:{3}]{}
child { {}    edge from parent [white] }
child { {}   edge from parent [white] }
child { {}  edge from parent [white] } } }
child { {} edge from parent [white] }  
child { node [vertex, label=0:{4}]{}   
child { {}    edge from parent [white] }
child { {}   edge from parent [white] }
child { {}  edge from parent [white] } };
\end{tikzpicture}
\end{subfigure}
 \begin{subfigure}
     [t]{0.3\textwidth}
     \centering
     \subcaption*{Andr\'e II tree $T_2$}
\begin{tikzpicture}
[vertex/.style={shape=circle, draw, inner sep=1.5pt, fill=black},
subtree/.style={shape=ellipse, draw,minimum width=7mm, minimum height=.5cm},
every fit/.style={ellipse,draw,inner sep=-2pt},
sibling distance=7mm,level distance=7mm,
leaf/.style={label={[name=#1]below:$ $}},auto]
\node [vertex, label=0:{1}]{}
child { node [vertex, label=0:{3}]{}  
child { {}    edge from parent [white] } child { {}   edge from parent [white] } child { node [vertex, label=0:{4}]{}  
child { {}    edge from parent [white] } child { {}   edge from parent [white] } child { {}  edge from parent [white] } } }
child { {} edge from parent [white] }
child { node [vertex, label=0:{2}]{}  
child { {}    edge from parent [white] } child { {}   edge from parent [white] } child { {}  edge from parent [white] } };
\end{tikzpicture}
 \end{subfigure}
 \caption{ Two Andr\'e  trees  on $[4]$.} 
\label{figAnd-ary}
\end{figure}
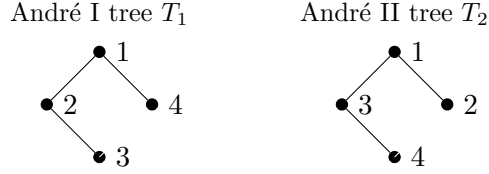

\begin{Proposition}\label{Andreperincr} The bijection $\Psi$ defined in Definition \ref{defi:mappsi} is a bijection between the set $\mathcal{T}^I_n$ of Andr\'e  I trees on $[n]$ {\rm (}resp. the set $\mathcal{T}^{I\!I}_n$ of Andr\'e  II trees on $[n]${\rm )} and the set ${\rm And}^{I}_n$ of Andr\'e  I permutations on $[n]$ {\rm (}resp. the set ${\rm And}^{I\!I}_n$ of Andr\'e  II permutations on $[n]$ {\rm )}. Moreover, for any $T\in \mathcal{T}^I_n$ {\rm (}resp. $T\in \mathcal{T}^{I\!I}_n$ {\rm )}
, if $\pi=\Psi(T)$, then 
    \[l(T)=\des(\pi) \quad \text{and} \quad \inv(T)=\inv(\pi).\]
\end{Proposition}

Using Proposition \ref{Andreperincr}, we are now able to express the polynomials $F^{I}_n(q;t)$ 
 and $F^{I\!I}_n(q;t)$ in terms of Andr\'e trees, which we refer to $q$-Andr\'e  polynomials.  
\begin{align} \label{eq:AndreIO}
E^{I}_n(q;x,y) &= \sum_{\sigma \in \mathcal{T}^I_n} x^{l(T)} y^{u(T)} q^{\inv{(T)}},\\
E^{I\!I}_n(q;x,y) &= \sum_{\sigma \in \mathcal{T}^{I\!I}_n} x^{l(T)} y^{u(T)} q^{\inv{(T)}}, \label{eq:AndreIIO}
\end{align}
where $l(T)$, $u(T)$ and $\inv(T)$ denote the number of leaves, the number of vertices of degree 1, and the inversion number of $T$, respectively.

From Proposition \ref{Andreperincr}, we see that upon setting $y=1$ and $x=t$, the polynomials $E^{I}_n(q;x,y)$ and $E^{I\!I}_n(q;x,y)$ reduce to $F^{I}_n(q;t) $ and $F^{I\!I}_n(q;t) $ respectively.

The remainder of this section is structured as follows. Subsection~\ref{sec:AndrI} presents the $q$-grammar for the $q$-Andr\'e I polynomials \eqref{eq:AndreIO} (see Theorem~\ref{thm:q-AndreI}); Subsection~\ref{sec:q-AndreII} is devoted to the $q$-grammar for the $q$-Andr\'e II polynomials \eqref{eq:AndreIIO} (see Theorem~\ref{thm:q-AndreII}).  We remark that this construction extends naturally to yield the $q$-grammar for $q$-analogues of $k$-Andr\'e trees. These $q$-analogues are obtained by generalizing the inversion statistic from increasing binary trees to the setting of increasing $k$-ary trees.

\subsection{\texorpdfstring{$q$-Grammar for $q$-Andr\'e I polynomials }{q-Grammar for q-Andr\'e I polynomials }} \label{sec:AndrI}

\begin{Theorem}\label{thm:q-AndreI}
Let $G_{{\rm AndI}}$ be the $q$-grammar defined by
\begin{equation}\label{def:AndreIGRA}
G_{{\rm AndI}}=(\{x,y\},  \{   x_j \rightarrow q^j x_j y_{j+1} , \  y_j \rightarrow q^j x_j \}    , \AIO).
\end{equation}
Let $D$ be the $q$-derivative associated with $G_{{\rm AndI}}$ and define the
evaluation map $\phi$ by $\phi(x_j)=x$ and $\phi(y_j)=y.$ 
Then
\begin{equation}\label{eq:AndreI}
\phi\bigl(D^n(x_0)\bigr) =E^{I}_{n+1}(q;x,y).
\end{equation} 
\end{Theorem}

We are ready to prove Theorem \ref{thm:q-AndreI}  by using the grammatical labeling. To do so, we first introduce several statistics on Andr\'e I trees.

For $T \in \mathcal{T}^{I}_n$ and $v\in T$, let $L_v$   be the set of vertices on the path from the root 1 to $v$, say $L_v=\{v_0:=1<v_1<\cdots <v_{m-1}<v_{m}:=v\}$.  Let $N_v$ denote the set of vertices in $L_v$ whose right child is not contained in  $L_v$. For $v_i \in N_v$,  let $T^r_{v_i}$ denote the subtree rooted at the right child of $v_i$ and let  $n_r(v_i)$ be the number of  vertices in $T^r_{v_i}$ with the convention that $n_r(v_i)=0$ if $T^r_{v_i}=\emptyset$.  We define the following   statistic for $v\in T$: 
    \begin{align*}
   \Delta^{T}_I(v)&=\sum_{v_i \in N_v} n_{r}(v_i).
    \end{align*}
    When the context of the tree 
    $T$  is clear, we simplify the notation to  $\Delta_I(v)$.  

\begin{figure}
    \centering
    \begin{tikzpicture}
    [scale = 0.7, vertex/.style={shape=circle, draw, inner sep=1.5pt, fill=black},
    subtree/.style={shape=ellipse, draw,minimum width=1.5cm, minimum height=.5cm},
    every fit/.style={ellipse,draw,inner sep=-2pt},
    sibling distance=4cm,level distance=12mm,
    leaf/.style={label={[name=#1]below:$ $}},auto]
    
    \node[vertex, name=1, label=30:{$1$}]{}[grow=down]
    child {node [vertex, name=3, label=-180:{$3$}]{}[level distance=10mm, sibling distance = 25mm]
    child {node [vertex, label=-90:{$8$}]{}}
    child {node [vertex, label=-90:{$10$}]{}}
    edge from parent [solid] }
    child{node[vertex, name=2, label=0:{$2$}]{}[level distance=10mm, sibling distance = 25mm]
    child {node [vertex, label=-90:{$9$}]{}}
    child {node [vertex, name=4, label=0:{$4$}]{}
    [level distance=10mm, sibling distance = 30mm]
    child {node [vertex, name=5, label=180:{$5$}]{}
    [level distance=10mm, sibling distance = 15mm]
    child {node [vertex, name=6, label=180:{$6$}]{}
    child { {}edge from parent [solid,white]}
    child {node [vertex, label=-90:{$12$}]{}}
    }
    child {node [vertex, label=0:{$13$}]{}
    [level distance=10mm, sibling distance = 12mm]
    child {node {} edge from parent [white]}
    child {node [vertex, label=0:{$14$}]{}
    child {node [vertex, label=-90:{$15$}]{}}
    child {node [vertex, label=-90:{$17$}]{}}}}}
    child {node [vertex, label=0:{$7$}]{}
    [level distance=10mm, sibling distance = 15mm]
    child {node [vertex, label=0:{$11$}]{}
    child {{} edge from parent [white]}
    child {node [vertex, label=0:{$16$}]{}
    }}
    child {node [vertex, label=0:{$18$}]{}}
    edge from parent [solid] }}};
    
    \draw[red, thick] (1) -- (2);
    \draw[red, thick] (2) -- (4);
    \draw[red, thick] (4) -- (5);
    \draw[red, thick] (5) -- (6);

    \draw[red,thick] (6) circle (5pt);
    \end{tikzpicture}
    \caption{An Andr\'e I tree $T$ on $[18]$.} \label{orintre-ex}
\end{figure}
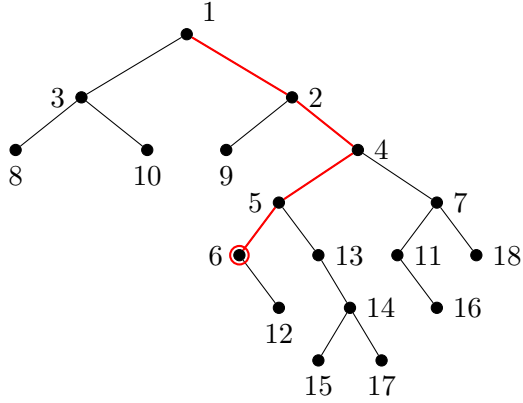

From Fig.~\ref{orintre-ex}, we see that 
$L_6=\{1,2,4,5,6\}$, where $T^r_1=T^r_2=\emptyset$ and $T^r_4, T^r_5$ and $T^r_6$ are given in Fig. \ref{subtreeAnd}. Hence, 
\[N_v=\{4,5,6\} \quad \text{and} \quad  \Delta_I(6)=n_r(4)+n_r(5)+n_r(6)=9.\]

\begin{figure}
    \centering    
 \begin{tikzpicture}
 [vertex/.style={shape=circle, draw, inner sep=1.5pt, fill=black},
 subtree/.style={shape=ellipse, draw,minimum width=1.5cm, minimum height=.5cm},
 every fit/.style={ellipse,draw,inner sep=-2pt},
sibling distance=4cm,level distance=12mm,
 leaf/.style={label={[name=#1]below:$ $}},auto]
 
\node [vertex, label=0:{$7$}]{}
 [level distance=10mm, sibling distance = 15mm]
 child {node [vertex, label=180:{$11$}]{}
 child {{} edge from parent [white]}
 child {node [vertex, label=-90:{$16$}]{}
 }}
 child {node [vertex, label=0:{$18$}]{}};

 \draw(0,-3.5) node{$T_4^r$};
\end{tikzpicture}
 \quad \quad  \quad \quad 
 \begin{tikzpicture}
 [vertex/.style={shape=circle, draw, inner sep=1.5pt, fill=black},
 subtree/.style={shape=ellipse, draw,minimum width=1.5cm, minimum height=.5cm},
 every fit/.style={ellipse,draw,inner sep=-2pt},
sibling distance=4cm,level distance=12mm,
 leaf/.style={label={[name=#1]below:$ $}},auto]
 
\node [vertex, label=0:{$13$}]{}
  [level distance=10mm, sibling distance = 12mm]
 child {node {} edge from parent [white]}
 child {node [vertex, label=0:{$14$}]{}
 child {node [vertex, label=-90:{$15$}]{}}
 child {node [vertex, label=-90:{$17$}]{}}};
\draw(0.5,-3.5) node{$T_5^r$};
\end{tikzpicture}
 \quad \quad  \quad \quad 
\begin{tikzpicture}
    \draw [fill=black] (0,-1)circle(2pt);
    \draw (0.1,-1) node[right]{$12$};
    \draw(0.2,-3.5) node{$T_6^r$};
\end{tikzpicture}
 \caption{Subtrees rooted at the right children of $4,5,6$.} \label{subtreeAnd}
    \end{figure}
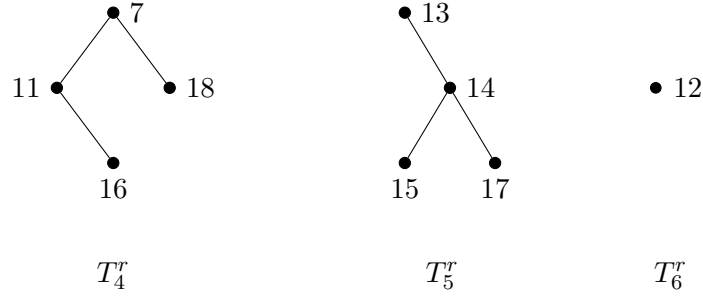

From the definition of $\Delta_I$, it is not difficult to show that 

\begin{Proposition}\label{Andreperprop} Let  $T$ be an  Andr\'e I tree   in $ \mathcal{T}^I_n$ and let $u, v $ be two vertices in $T$  each having  at most one child. If  $u$  lies on or to the left of the path from the root 1 to $v$, then $\Delta_I(u)> \Delta_I(v)$; otherwise $\Delta_I(u)<\Delta_I(v)$. 
\end{Proposition}

 \noindent{ \bf An insertion algorithm for Andr\'e I trees:} 
Let $T$ be an Andr\'e I tree.  Suppose $T$ has  $m$ vertices with at most one child. We 
present an insertion algorithm to generate $m$ Andr\'e I trees on $[n+1]$ by inserting $n+1$ into $T$.

Let $v$ be the leaf or the vertex with exactly one child in $T$. We define $\phi_v$ that transforms $T$ into an Andr\'e I tree on $[n+1]$ via operations on $v$.

Recall that $L_{v}$ denotes the set of vertices on the path from the root 1 to~$v$: 
\[L_{v}=\{v_0:=1<v_1<\cdots <v_j:=v \}.\]
Let $v^r_{1},\ldots, v^r_{k}$ be  the $k$ vertices in $L_v$ whose right children are not in $L_v$, that is, 
\[N_v=\{v^r_{1},\ldots, v^r_{k}\}.\]
Denote by  $T^r_{v^r_i}$ the subtree rooted at the right child of $v^r_i$ and let 
$a_i$ be  largest vertex in the subtree $T^r_{v^r_i}$. Define
\[M(L_v):=\{a_1>a_2>\cdots>a_k\}.\]
Relabel the elements of $M(L_v)$ according to the permutation 
\[\left( \begin{array} {ccccccc}
a_1&a_2 &a_3&\cdots & a_k\\
n+1&a_1&a_2 &\cdots & a_{k-1}
\end{array} 
\right).
\]
If $v$ is a leaf of $T$, then assign $a_k$ as the right child of $v$ (see Fig. \ref{inserAndrea}); If $v$ has one child of $T$, then assign $a_k$ as the left child of $v$ (see Fig. \ref{inserAndreb}). This yields an Andr\'e I tree  $\widetilde{T}:=\phi^{I}_v(T)$. It can be readily seen that the insertion algorithm is reversible.

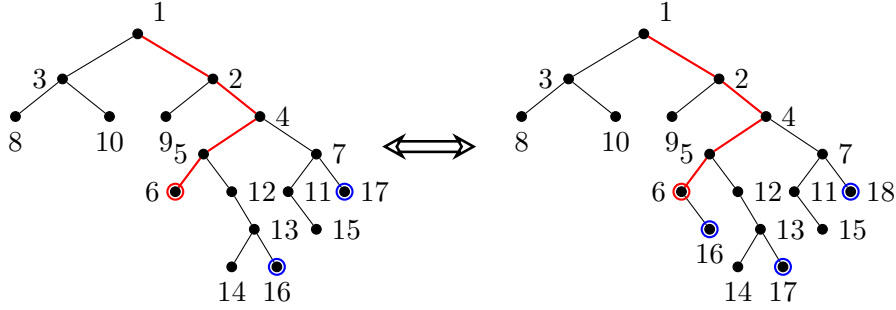
\begin{figure}
    \centering
    \begin{subfigure}[t]{0.4\textwidth}
    \begin{tikzpicture}
    [scale = 0.5, vertex/.style={shape=circle, draw, inner sep=1.3pt, fill=black},
    subtree/.style={shape=ellipse, draw,minimum width=1.5cm, minimum height=.5cm},
    every fit/.style={ellipse,draw,inner sep=-2pt},
    sibling distance=4cm,level distance=12mm,
    leaf/.style={label={[name=#1]below:$ $}},auto]
    
    \node[vertex, name=1, label=30:{$1$}]{}[grow=down]
    child {node [vertex, name=3, label=-180:{$3$}]{}[level distance=10mm, sibling distance = 25mm]
    child {node [vertex, label=-90:{$8$}]{}}
    child {node [vertex, label=-90:{$10$}]{}}
    edge from parent [solid] }
    child{node[vertex, name=2, label=0:{$2$}]{}[level distance=10mm, sibling distance = 25mm]
    child {node [vertex, label=-90:{$9$}]{}}
    child {node [vertex, name=4, label=0:{$4$}]{}
    [level distance=10mm, sibling distance = 30mm]
    child {node [vertex, name=5, label=180:{$5$}]{}
    [level distance=10mm, sibling distance = 15mm]
    child {node [vertex, name=6, label=180:{$6$}]{}}
    child {node [vertex, label=0:{$12$}]{}
    [level distance=10mm, sibling distance = 12mm]
    child {node {} edge from parent [white]}
    child {node [vertex, label=0:{$13$}]{}
    child {node [vertex, label=-90:{$14$}]{}}
    child {node [vertex, name = 16, label=-90:{$16$}]{}}}}}
    child {node [vertex, label=0:{$7$}]{}
    [level distance=10mm, sibling distance = 15mm]
    child {node [vertex, label=0:{$11$}]{}
    child {{} edge from parent [white]}
    child {node [vertex, label=0:{$15$}]{}
    }}
    child {node [vertex, name=17, label=0:{$17$}]{}}
    edge from parent [solid] }}};
    
    \draw[red, thick] (1) -- (2);
    \draw[red, thick] (2) -- (4);
    \draw[red, thick] (4) -- (5);
    \draw[red, thick] (5) -- (6);

    \draw[red,thick] (6) circle (6pt);
    \draw[blue,thick]  (17) circle (6pt) (16) circle (6pt);
    
\draw[double distance=3pt, 
      <->, 
      >={Stealth[open, length=9pt, width=10pt]},
      line width=1pt] 
      (6.5,-3) -- (9,-3);
\end{tikzpicture}
 \end{subfigure}  \qquad\qquad 
 \begin{subfigure}[t]{0.4 \textwidth}
    \begin{tikzpicture}
    [scale = 0.5, vertex/.style={shape=circle, draw, inner sep=1.3pt, fill=black},
    subtree/.style={shape=ellipse, draw,minimum width=1.5cm, minimum height=.5cm},
    every fit/.style={ellipse,draw,inner sep=-2pt},
    sibling distance=4cm,level distance=12mm,
    leaf/.style={label={[name=#1]below:$ $}},auto]
    
    \node[vertex, name=1, label=30:{$1$}]{}[grow=down]
    child {node [vertex, name=3, label=-180:{$3$}]{}[level distance=10mm, sibling distance = 25mm]
    child {node [vertex, label=-90:{$8$}]{}}
    child {node [vertex, label=-90:{$10$}]{}}
    edge from parent [solid] }
    child{node[vertex, name=2, label=0:{$2$}]{}[level distance=10mm, sibling distance = 25mm]
    child {node [vertex, label=-90:{$9$}]{}}
    child {node [vertex, name=4, label=0:{$4$}]{}
    [level distance=10mm, sibling distance = 30mm]
    child {node [vertex, name=5, label=180:{$5$}]{}
    [level distance=10mm, sibling distance = 15mm]
    child {node [vertex, name=6, label=180:{$6$}]{}
    child { {} edge from parent [solid,white]}
    child {node [vertex, name=16, label=-90:{$16$}]{}}}
    child {node [vertex, label=0:{$12$}]{}
    [level distance=10mm, sibling distance = 12mm]
    child {node {} edge from parent [white]}
    child {node [vertex, label=0:{$13$}]{}
    child {node [vertex, label=-90:{$14$}]{}}
    child {node [vertex, name = 17, label=-90:{$17$}]{}}}}}
    child {node [vertex, label=0:{$7$}]{}
    [level distance=10mm, sibling distance = 15mm]
    child {node [vertex, label=0:{$11$}]{}
    child {{} edge from parent [white]}
    child {node [vertex, label=0:{$15$}]{}
    }}
    child {node [vertex, name=18, label=0:{$18$}]{}}
    edge from parent [solid] }}};
    
    \draw[red, thick] (1) -- (2);
    \draw[red, thick] (2) -- (4);
    \draw[red, thick] (4) -- (5);
    \draw[red, thick] (5) -- (6);
    \draw[red, thick] (6) circle(6pt);
    \draw[blue, thick] (16) circle(6pt) (17) circle(6pt) (18) circle(6pt);
    \end{tikzpicture}
    \end{subfigure}
     \caption{The bijection $\phi_6$, where  $M(L_6)=\{17,16\}$.} \label{inserAndrea}
\end{figure}

\begin{figure}
    \centering
 \begin{subfigure}[t]{0.4 \textwidth}
    \begin{tikzpicture}
    [scale = 0.5, vertex/.style={shape=circle, draw, inner sep=1.3pt, fill=black},
    subtree/.style={shape=ellipse, draw,minimum width=1.5cm, minimum height=.5cm},
    every fit/.style={ellipse,draw,inner sep=-2pt},
    sibling distance=4cm,level distance=12mm,
    leaf/.style={label={[name=#1]below:$ $}},auto]
    
    \node[vertex, name=1, label=30:{$1$}]{}[grow=down]
    child {node [vertex, name=3, label=-180:{$3$}]{}[level distance=10mm, sibling distance = 25mm]
    child {node [vertex, label=-90:{$8$}]{}}
    child {node [vertex, label=-90:{$10$}]{}}
    edge from parent [solid] }
    child{node[vertex, name=2, label=0:{$2$}]{}[level distance=10mm, sibling distance = 25mm]
    child {node [vertex, label=-90:{$9$}]{}}
    child {node [vertex, name=4, label=0:{$4$}]{}
    [level distance=10mm, sibling distance = 30mm]
    child {node [vertex, name=5, label=180:{$5$}]{}
    [level distance=10mm, sibling distance = 15mm]
    child {node [vertex, name=6, label=180:{$6$}]{}
    child { {} edge from parent [white]}
    child {node [vertex, name=12, label=-90:{$12$}]{}}}
    child {node [vertex, label=0:{$13$}]{}
    [level distance=10mm, sibling distance = 12mm]
    child {node {} edge from parent [white]}
    child {node [vertex, label=0:{$14$}]{}
    child {node [vertex, label=-90:{$15$}]{}}
    child {node [vertex, name = 17, label=-90:{$17$}]{}}}}}
    child {node [vertex, label=0:{$7$}]{}
    [level distance=10mm, sibling distance = 15mm]
    child {node [vertex, label=0:{$11$}]{}
    child {{} edge from parent [white]}
    child {node [vertex, label=0:{$16$}]{}
    }}
    child {node [vertex, name=18, label=0:{$18$}]{}}
    edge from parent [solid] }}};
    
    \draw[red, thick] (1) -- (2);
    \draw[red, thick] (2) -- (4);
    \draw[red, thick] (4) -- (5);
    \draw[red, thick] (5) -- (6);
    \draw[red, thick] (6) circle(6pt);
    \draw[blue, thick] (12) circle(6pt) (17) circle(6pt) (18) circle(6pt) ;
    \draw[double distance=3pt, 
      <->, 
      >={Stealth[open, length=9pt, width=10pt]},
      line width=1pt] 
      (6.5,-3) -- (9,-3);
    \end{tikzpicture}
    \end{subfigure}\qquad\qquad 
     \begin{subfigure}[t]{0.4\textwidth}
    \begin{tikzpicture}
    [scale = 0.5, vertex/.style={shape=circle, draw, inner sep=1.3pt, fill=black},
    subtree/.style={shape=ellipse, draw,minimum width=1.5cm, minimum height=.5cm},
    every fit/.style={ellipse,draw,inner sep=-2pt},
    sibling distance=4cm,level distance=12mm,
    leaf/.style={label={[name=#1]below:$ $}},auto]
    
    \node[vertex, name=1, label=30:{$1$}]{}[grow=down]
    child {node [vertex, name=3, label=-180:{$3$}]{}[level distance=10mm, sibling distance = 25mm]
    child {node [vertex, label=-90:{$8$}]{}}
    child {node [vertex, label=-90:{$10$}]{}}
    edge from parent [solid] }
    child{node[vertex, name=2, label=0:{$2$}]{}[level distance=10mm, sibling distance = 25mm]
    child {node [vertex, label=-90:{$9$}]{}}
    child {node [vertex, name=4, label=0:{$4$}]{}
    [level distance=10mm, sibling distance = 30mm]
    child {node [vertex, name=5, label=180:{$5$}]{}
    [level distance=10mm, sibling distance = 15mm]
    child {node [vertex, name=6, label=180:{$6$}]{}
    child { node [vertex, name=12,label=180:{$12$}]{}}
    child {node [vertex, name=17, label=-90:{$17$}]{}}}
    child {node [vertex, label=0:{$13$}]{}
    [level distance=10mm, sibling distance = 12mm]
    child {node {} edge from parent [white]}
    child {node [vertex, label=0:{$14$}]{}
    child {node [vertex, label=-90:{$15$}]{}}
    child {node [vertex, name = 18, label=-90:{$18$}]{}}}}}
    child {node [vertex, label=0:{$7$}]{}
    [level distance=10mm, sibling distance = 15mm]
    child {node [vertex, label=0:{$11$}]{}
    child {{} edge from parent [white]}
    child {node [vertex, label=0:{$16$}]{}
    }}
    child {node [vertex, name=19, label=0:{$19$}]{}}
    edge from parent [solid] }}};
    
    \draw[red, thick] (1) -- (2);
    \draw[red, thick] (2) -- (4);
    \draw[red, thick] (4) -- (5);
    \draw[red, thick] (5) -- (6);

    \draw[red,thick] (6) circle (6pt);
    \draw[blue,thick] (17) circle (6pt) (12) circle (6pt) (19) circle (6pt) (18) circle (6pt);
    

\end{tikzpicture}
 \end{subfigure}   
 \caption{The bijection $\phi_6$, where $M(L_6)=\{18,17,12\}$.}  \label{inserAndreb}
\end{figure}
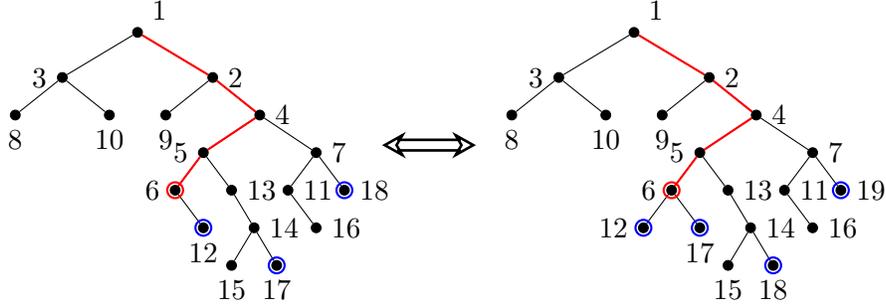

From the construction of the insertion algorithm, it is not difficult to show the following proposition. 
 
\begin{Proposition}\label{Andreperpropb}
Let $T$ be an Andr\'e I tree on $[n]$ and let $v$ be a vertex in $T$ with at most one child.
Suppose that $\widetilde{T}=\phi_v(T)$. Then
\begin{equation}\label{insrelfa}
\inv(\widetilde{T}) - \inv(T) = \Delta^{T}_I(v).
\end{equation}
Moreover, for any vertex $u$ in $T$ with at most one child, we have the following results.
\begin{itemize}
\item[(a)] If $u$ lies on or to the left of the path from root $1$ to $v$, then
\begin{equation*}
\Delta^{\widetilde{T}}_I(u) = \Delta^{T}_I(u) + 1.
\end{equation*}

\item[(b)] If $u$ lies to the right of the path from root $1$ to $v$ and $u \not\in M(L_v)$, then
\begin{equation*}
\Delta^{\widetilde{T}}_I(u) = \Delta^{T}_I(u).
\end{equation*}

\item[(c)] If $u\in M(L_v):=\{a_1>a_2>\cdots>a_k\}$, then
\begin{align*}
\Delta^{\widetilde{T}}_I(a_k) = \Delta^{T}_I(v), \quad
\Delta^{\widetilde{T}}_I(n+1) = \Delta^{T}_I(a_1) = 0,
\end{align*}
and for $1\leq i\leq k-1$,
\begin{equation*}
\Delta^{\widetilde{T}}_I(a_i) = \Delta^{T}_I(a_{i+1}).
\end{equation*}

\item[(d)] If $v$ is a leaf of $T$, then
\begin{equation*}
\Delta^{\widetilde{T}}_I(v) = \Delta^{T}_I(v) + 1.
\end{equation*}
\end{itemize}
\end{Proposition}

We are ready to prove Theorem \ref{thm:q-AndreI} by using the grammatical labeling.  

Let $T$ be the Andr\'e I trees on $[n]$. For $1\leq i\leq n$, we label the vertex $i$ in $T$ as follows: 
\begin{itemize}
\item If $i$ is a leaf and $\Delta_I(i)=k$, then label it by $x_k$; 
\item If $i$ has only one child and $\Delta_I(i)=k$, then label it by $y_k$; 
\item If $i$ has two children, then it is left unlabeled.  
\end{itemize}

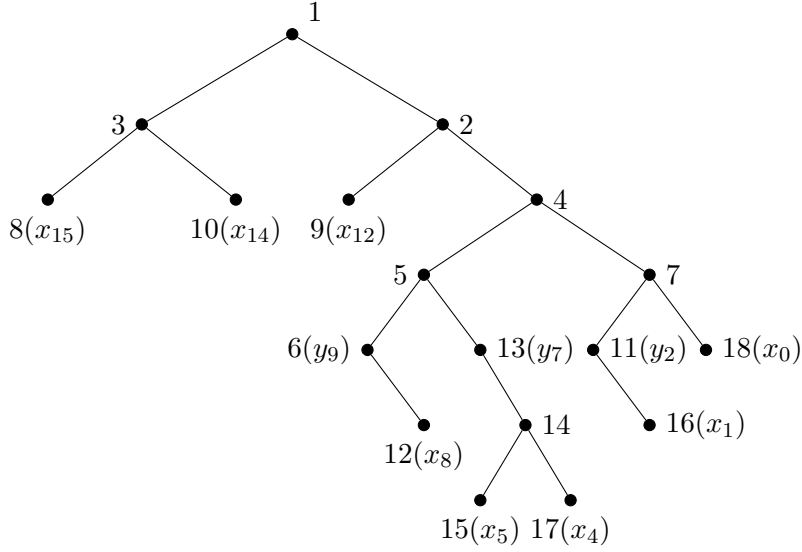
\begin{figure}
    \centering
 \begin{tikzpicture}
 [vertex/.style={shape=circle, draw, inner sep=1.5pt, fill=black},
 subtree/.style={shape=ellipse, draw,minimum width=1.5cm, minimum height=.5cm},
 every fit/.style={ellipse,draw,inner sep=-2pt},
sibling distance=4cm,level distance=12mm,
 leaf/.style={label={[name=#1]below:$ $}},auto]
 
\node[vertex,label=30:{$1$}]{}[grow=down]
  child {node [vertex, label=-180:{$3$}]{}[level distance=10mm, sibling distance = 25mm]
  child {node [vertex, label=-90:{$8(x_{15})$}]{}}
  child {node [vertex, label=-90:{$10(x_{14})$}]{}}
 edge from parent [solid] }
 child{node[vertex, label=0:{$2$}]{}[level distance=10mm, sibling distance = 25mm]
 child {node [vertex, label=-90:{$9(x_{12})$}]{}}
 child {node [vertex, label=0:{$4$}]{}
 [level distance=10mm, sibling distance = 30mm]
 child {node [vertex, label=180:{$5$}]{}
 [level distance=10mm, sibling distance = 15mm]
 child {node [vertex, label=180:{$6(y_9)$}]{}
  child { {}edge from parent [solid,white]}
 child {node [vertex, label=-90:{$12(x_8)$}]{}}
 }
 child {node [vertex, label=0:{$13(y_7)$}]{}
  [level distance=10mm, sibling distance = 12mm]
 child {node {} edge from parent [white]}
 child {node [vertex, label=0:{$14$}]{}
 child {node [vertex, label=-90:{$15(x_5)$}]{}}
 child {node [vertex, label=-90:{$17(x_4)$}]{}}}}}
 child {node [vertex, label=0:{$7$}]{}
 [level distance=10mm, sibling distance = 15mm]
 child {node [vertex, label=0:{$11(y_2)$}]{}
 child {{} edge from parent [white]}
 child {node [vertex, label=0:{$16(x_1)$}]{}
 }}
 child {node [vertex, label=0:{$18(x_0)$}]{}}
 edge from parent [solid] }}};
\end{tikzpicture}
\caption{The labeling of an Andr\'e I trees $T$ on [18].} \label{labelAndre}
    \end{figure}

The weight $\omega$ of $T \in \mathcal{T}^I_{n}$ is defined to be the product of all labels of $T$ arranged in $\AIO$ order (that is, $x_0,y_0,x_1,y_1,\ldots$).
 
For example, Fig. \ref{labelAndre}  shows the grammatical labeling of the Andr\'e I trees $T$ on $[18]$.   We see that the weight of $T$  is 
\[\omega(T)=x_0x_1y_2x_4x_5y_7x_8y_9x_{12}x_{14}x_{15}.\]

To prove Theorem  \ref{thm:q-AndreI}, we aim to show that   the following assertion: For $n\geq 0$,
\begin{equation}\label{equ-Andre}
D^n (x_0)=\sum_{T \in \mathcal{T}^I_{n}}   q^{\inv(T)}\omega(T),
\end{equation}
where $D$ is $q$-derivative  associated with $q$-grammar \eqref{def:AndreIGRA}. 

\begin{proof}[Proof of Theorem   \ref{thm:q-AndreI}] We proceed by induction on $n$. For $n=0$, the statement is evident. Assume that this statement holds for $n$,  that is, the relation \eqref{equ-Andre} is valid for $n-1$. To demonstrate that it also holds for $n$,   it suffices to show, by \eqref{equ-Andre}, that
\begin{equation}\label{gram-anre-inva}
D\left(\sum_{T \in \mathcal{T}^I_{n}}   q^{\inv(T)}\omega(T)\right)=\sum_{\widetilde{T} \in \mathcal{T}^I_{n+1}}   q^{\inv(\widetilde{T})}\omega(\widetilde{T}).
\end{equation}
Suppose $T$ contains $m$ vertices with at most one child and the weight of $T$ is given by  
\[\omega(T)= \omega_1(T)\omega_2(T) \cdots \omega_m(T).\] 

By the $q$-product formula, we have  
\begin{equation}\label{qLeibnitz-AndreI}
D\bigl(\omega(T)\bigr)
=\sum_{i=1}^{m} \AIO\Bigl(
\omega_1(T)\cdots\omega_{i-1}(T)\,
R(\omega_i(T))\,
\uparrow\!\bigl(\omega_{i+1}(T)\cdots\omega_{m}(T)\bigr)
\Bigr). 
\end{equation}
 Let $v_i$ be the vertex labeled by $\omega_i(T)$ in $T$, and let $\widetilde{T}^{(i)}=\phi_{v_i}(T)$.  To prove \eqref{gram-anre-inva}, it is enough to show that 
\begin{align}\label{pfbb-maj-a}
& q^{\inv  (T)}\AIO\Bigl(
\omega_1(T)\cdots\omega_{i-1}(T)\,
R(\omega_i(T))\,
\uparrow\!\bigl(\omega_{i+1}(T)\cdots\omega_{m}(T)\bigr)
\Bigr)\nonumber \\
&\quad =q^{\inv  (\widetilde{T}^{(i)})}\omega(\widetilde{T}^{(i)}).
\end{align}

Since the letters of $\omega(T)$ are ordered as  $x_0,y_0,x_1,y_1,\ldots$, by  
Proposition~\ref{Andreperprop}, vertices labeled  $\omega_1(T), \ldots, \omega_{i-1}(T)$ lie to the right of the path from 1 to $v_i$ in $T$, while the vertices labeled $\omega_{i+1}(T), \ldots, \omega_{m}(T)$ lie either on this path or to its left. 
We consider two cases:

\begin{itemize}
\item[Case 1:]  If $\omega_i(T)=x_k$, then $v_i$ is a leaf of $T$ with $\Delta_I(v_i)=k$. By Proposition \ref{Andreperpropb}, we see that \eqref{pfbb-maj-a} holds since $D(x_k)=q^k x_k y_{k+1}.$

\item[Case 2:]  If $\omega_i(T)=y_k$, then $v_i$ is a vertex of $T$ with exactly one child and $\Delta_I(v_i)=k$.  In light of  Proposition \ref{Andreperpropb},   we find that  \eqref{pfbb-maj-a} is satisfied since $D(y_k)=q^k x_k.$

\end{itemize}
 
Summing the results from these two cases shows that the assertion \eqref{gram-anre-inva} holds, and thus \eqref{equ-Andre} is valid for $n+1$. This completes the proof of Theorem~\ref{thm:q-AndreI}. 
\end{proof}

Next, we will provide a grammatical derivation of the following identity for $E^{I}_{n+1}(q;x,y)$, which specializes to the relation for $E(x,y)$ established by Chen and Fu \cite{Chen-Fu-2017}. 

\begin{Theorem} \label{thm:AndrIrec} Set $E^{I}_{0}(q;x,y)=x$.  For $n\geq 1$, 
\begin{align} \label{eq:AndrIrec}
&E^{I}_{n+1}(q;x,y) \nonumber \\[5pt]
&\quad =yE^{I}_{n}(q;x,y) +\sum_{k=0}^{n-2} q^{n-k-1} {n-1 \brack k}_q  E^{I}_{k+1}(q;x,y)E^{I}_{n-k-1}(q;x,y).
\end{align}
\end{Theorem}

\begin{proof}
Observe that the map $\phi$ from Theorem \ref{thm:q-AndreI} is a master-linear evaluation, and the grammar $G_{{\rm AndI}}$ defined in \eqref{def:AndreIGRA} satisfies condition \eqref{eq:rule}. In fact, from the proof of Theorem \ref{thm:q-AndreI},   the order $\AIO$ in the grammar $G_{{\rm AndI}}$ is equivalent to the order $\KSO$. This implies that   $G_{{\rm AndI}}$ is $q$-linear grammar. 

Applying Theorem \ref{thm:product}, we immediately obtain the following multiplicative identity:
\begin{align*}
\phi\big(\mathrm{Gen}_q^{(G_{{\rm AndI}})}(x_0y_1;u)\big)
=\phi\big(\mathrm{Gen}_q^{(G_{{\rm AndI}})}(x_0;u)\big)\phi\big(\mathrm{Gen}_q^{(G_{{\rm AndI}})}(y_1;u)\big).
\end{align*}
This implies
\begin{align}\label{pf:AndreIrecc}
\phi\big(D^{n}(x_0y_1)\big)
=\sum_{k=0}^n {n\brack k}_q \phi\big(D^k(x_0)\big) \phi\big(D^{n-k}(y_1)\big),
\end{align}
where $D$ is the $q$-derivative associated to the grammar $G_{{\rm AndI}}$.

By Theorem \ref{thm:q-AndreI}, for all integers $k\geq 0$, we have
\begin{equation}\label{pf:AndreIreca}
\phi\big(D^k(x_0)\big) =E^I_{k+1}(q;x,y).
\end{equation}

Since  the grammar $G_{{\rm AndI}}$ is $q$-linear, by Proposition \ref{q_linear_uparrow-prop}, we derive  that  for   $k\geq 1$,
\begin{align}\label{pf:AndreIrecb}
\phi\big(D^k(y_1)\big)&=\phi\big(D^{k-1}(qx_1)\big) \nonumber \\
&=q\,\phi\big(D^{k-1}(\uparrow x_0)\big)\nonumber  \\
&=q^{k} \phi\big(\uparrow D^{k-1}(x_0) \big)\nonumber  \\
&=q^{k}E^{I}_{k}(q;x,y).
\end{align}
Substituting \eqref{pf:AndreIreca} and \eqref{pf:AndreIrecb} into \eqref{pf:AndreIrecc} yields
\begin{align}\label{pf:AndreIrece}
\phi\big(D^{n}(x_0y_1)\big)
&=y\,E^I_{n+1}(q;x,y) \nonumber \\
&\quad+\sum_{k=0}^{n-1} q^{n-k} {n\brack k}_q
E^I_{k+1}(q;x,y)\, E^{I}_{n-k}(q;x,y).
\end{align}
On the other hand, another application of Theorem \ref{thm:q-AndreI} shows that for $n\geq 1$,
\begin{equation}\label{pf:AndreIrecd}
\phi\big(D^{n}(x_0)\big)
=\phi\big(D^{n-1}(x_0y_1)\big)
=E^I_{n+1}(q;x,y).
\end{equation}
Combining relations \eqref{pf:AndreIrece} and \eqref{pf:AndreIrecd}, we finally establish the recurrence \eqref{eq:AndrIrec}. This completes the proof. 
\end{proof}

 Finally, we remark that the $q$-grammar  $G_{{\rm AndI}}$ defined in \eqref{def:AndreIGRA} contains several well-known integer sequences, as stated next. 

Let $D$ be the $q$-derivative associated with  $G_{{\rm AndI}}$ given by \eqref{def:AndreIGRA} and define 
\begin{equation}\label{eq:canonicalexp}
D^{\,n} (x_0)=\sum_{w \in F(\mathbb{S})} a_w\, w.
\end{equation}
It is easy to see that the word $w$ in \eqref{eq:canonicalexp} is strictly characterized by a set of indices $\mathcal{I} \subseteq \{1, 2, \dots, n\}$  such that:
\begin{equation}\label{eq:canonicalexpaa}
    w = x_0 \prod_{i \in \mathcal{I}} v_i,
\end{equation}
where the products are taken in increasing order of indices and $v_i=x_i$ or $y_i$.  For a given word $w$ in \eqref{eq:canonicalexpaa}, we define a step sequence $S = (s_1, s_2, \dots, s_n)$ of length $n$ as follows:
\begin{itemize}
    \item $s_i = U$ (Up step, $(1,1)$) if $i \in \mathcal{I}$ and $v_i = x_i$;
    \item $s_i = L$ (Level step, $(1,0)$) if $i \in \mathcal{I}$ and $v_i = y_i$;
    \item $s_i = D^*$ (Down step, $(1,-1)$) if $i \notin \mathcal{I}$.
\end{itemize}

\begin{Proposition}\label{prop:Motzkin}
A word $w$ is a term in \eqref{eq:canonicalexp} if and only if its associated step sequence $S$ is a Motzkin path of length $n$ (i.e., the path stays non-negative and ends at height 0). Consequently, the number of distinct terms in $D^n(x_0)$ is the $n$-th Motzkin number $M_n$, that is,
\[\Omega(D^n(x_0))=M_n.\] 
\end{Proposition}

\begin{proof}
We proceed by induction on $n$.

\noindent{\bf Base Case:} For $n=1$, $D(x_0) = x_0 y_1$. The index set is $\mathcal{I}=\{1\}$ with $v_1=y_1$. The sequence is $S=(L)$, which is the unique Motzkin path of length 1. Thus $M_1 = 1$.

\noindent{\bf Inductive Step:} Assume the proposition holds for $n-1$. Any term $w' \in D^n(x_0)$ is generated by applying the derivative operator $D$ to some term $w \in D^{n-1}(x_0)$. By the definition of $D$, for a word $w = w_0 w_1 \dots w_k$:
\[
D(w) = \sum_{j=0}^k \rho \Bigl( w_0 \dots w_{j-1} R(w_j) \uparrow (w_{j+1} \dots w_k) \Bigr).
\]
Let $S = (s_1, \dots, s_{n-1})$ be the Motzkin path corresponding to $w$. We analyze the effect of $R$ on each step:

 \noindent{\bf Case 1:  Derivation of $x_0$ ($j=0$).} 
    Since $R(x_0) = x_0 y_1$, the new word starts with $x_0 y_1$ and all subsequent indices are shifted by 1. The new sequence becomes $S' = (L, s_1, s_2, \dots, s_{n-1})$. Since $S$ is a Motzkin path, prefixing a Level step $L$ results in a valid Motzkin path of length $n$.
    
 \noindent{\bf Case 2: Derivation of $x_m$ ($m \in \mathcal{I}, v_m=x_m$).}
    The rule $R(x_m) = x_m y_{m+1}$ replaces $x_m$ with $x_m y_{m+1}$ and shifts subsequent indices. In terms of steps, the original $U$ step at position $m$ remains $U$, and a new $L$ step is inserted at $m+1$. The sequence becomes $S' = (s_1, \dots, s_m, L, s_{m+1},\break  \dots, s_{n-1})$. Inserting $L$ after an Up step preserves non-negativity and the final zero height.
    
 \noindent{\bf Case 3: Derivation of $y_m$ ($m \in \mathcal{I}, v_m=y_m$).}
    The rule $R(y_m) = x_m$ replaces $y_m$ with $x_m$ and shifts subsequent indices. Note that the index $m+1$ is now skipped in the resulting word because no rule produced an index $m+1$. According to our characterization, a skipped index corresponds to a Down step $D^*$. Thus, the original $L$ step at position $m$ is replaced by the pair $(U, D^*)$. The sequence becomes $S' = (s_1, \dots, s_{m-1}, U, D^*, s_{m+1}, \dots, s_{n-1})$. Replacing $L$ with $(U, D^*)$ increases the intermediate height by 1 at one point and returns to the original height, thus preserving the Motzkin property.

\noindent{\bf Bijection:} The mapping $\Phi: w \mapsto S$ is injective because $S$ uniquely reconstructs the set $\mathcal{I}$ and the variables $v_i$. To show surjectivity, any Motzkin path of length $n$ can be reduced to a path of length $n-1$ by reversing one of the three operations above (removing a leading $L$, removing an $L$ that follows a $U$, or collapsing a $(U, D^*)$ pair into an $L$). By the inductive hypothesis, all Motzkin paths of length $n$ are covered. This completes the proof. 
\end{proof}

\begin{Proposition} 
Let $\phi^{nc}$ denote the map sending $x_j \mapsto x$ and $y_j \mapsto y$, where $x$ and $y$ are non-commutative.   Then the number of distinct terms in $\phi^{nc}(D^n(x_0))$ is the $(n+1)$-th Fibonacci number $F_{n+1}$, that is, 
 \[\Omega(\phi^{nc}(D^n(x_0)))=F_{n+1}.\]  
\end{Proposition}

\begin{proof}
Let $\mathcal{W}_n$ be the set of terms in $D^n(x_0)$. From the previous result, we know there is a bijection between $\mathcal{W}_n$ and the set of Motzkin paths of length $n$. Let $P = (s_1, s_2, \dots, s_n)$ be a Motzkin path where $s_i \in \{U, L, D^*\}$ ($U$: Up-step, $L$: Level-step, $D^*$: Down-step).

\noindent{\bf Step 1. Structural Mapping:}
According to the characterization of terms, each step $s_i$ in the Motzkin path corresponds to the variables in the word $w \in D^n(x_0)$ as follows:
\begin{itemize}
    \item $s_i = U \iff$ the variable at index $i$ is $x_i$.
    \item $s_i = L \iff$ the variable at index $i$ is $y_i$.
    \item $s_i = D^* \iff$ index $i$ is  skipped (no variable exists with index $i$).
\end{itemize}
When we perform the mapping $x_j \to x$ and $y_j \to y$ and ignore indices, a term $w$ is uniquely determined by the sequence of $x$'s and $y$'s it contains. Note that $D^*$ steps (skipped indices) do not contribute a variable to the word. Thus, the word $w$ (excluding $x_0$) is the sequence of $x$ and $y$ variables corresponding to the $U$ and $L$ steps in the path, in order.

\noindent{\bf Step 2. Combinatorial Counting:}
Let $k-1$ be the number of $U$ steps in a Motzkin path of length $n$. To be a valid Motzkin path ending at height 0, there must also be exactly $k-1$ $D^*$ steps. The remaining $n - 2(k-1)$ steps must be $L$ steps.
The resulting word $w$ after mapping will have:
\begin{itemize}
    \item $k-1$ instances of variable $x$.
    \item $n - 2k + 2$ instances of variable $y$.
\end{itemize}
The total number of variables in the word $w$ is $(k-1) + (n-2k+2) = n-k+1$. 

In the non-commutative case, the number of distinct words is the number of ways to arrange $k-1$ $x$'s and $(n-k+1) - (k-1) = n-2k+2$ $y$'s such that the underlying Motzkin path is valid. 

The number of ways to choose the positions of $k-1$ units of $x$ (each effectively  covering two indices in the derivation history) out of a reduced sequence of $n-k+1$ positions is given by the binomial coefficient:
\[
\binom{n-k+1}{k-1}.
\]
Summing over all possible numbers of $x$ variables (from $k-1=0$ up to the maximum possible $\lfloor n/2 \rfloor$) to get 
\[
\Omega(\phi^{nc}(D^n(x_0))) = \sum_{k=1}^{\lfloor n/2 \rfloor + 1} \binom{n-k+1}{k-1},
\]
which equals $F_{n+1}$ because 
\[
F_{n+1} = \sum_{m=0}^{\lfloor n/2 \rfloor} \binom{n-m}{m}.
\]
 This completes the proof. 
\end{proof}

\begin{Proposition}\label{coro:Euler}
Let $\phi$ denote the map sending $x_j \mapsto 1$,  $y_j \mapsto 1$ and $q \mapsto 1$. Then   $\phi(D^n(x_0))=E_{n+1}$, the $(n+1)$-th Euler numbers. 
\end{Proposition}
\begin{proof}
This follows immediately from Theorem \ref{thm:q-AndreI}.  
\end{proof}

\subsection{\texorpdfstring{$q$-Grammar for $q$-Andr\'e II polynomials}{q-Grammar for q-Andr\'e II polynomials}} \label{sec:q-AndreII}

\begin{Theorem}\label{thm:q-AndreII}
Let $G_{\rm{AndII}}$ be the $q$-grammar defined by
\begin{equation}\label{def:AndreII}
G_{\rm{AndII}}=(\{x,y\},  \{  x_j \rightarrow q^j x_j y_{j+1} , \  y_j \rightarrow q^{j+1} x_{j+1}  \}    , \AIO).
\end{equation}
Let $D$ be the $q$-derivative associated with $G_{\rm{AndII}}$ and define the
evaluation map $\phi$ by $\phi(x_j)=x$ and $\phi(y_j)=y$. Then
\begin{equation}\label{eq:AndreII}
\phi\bigl(D^n(x_0)\bigr) =E^{I\!I}_{n+1}(q;x,y).
\end{equation} 
\end{Theorem}

To prove Theorem \ref{thm:q-AndreII}, we first introduce the following statistic on Andr\'e II trees.

For $T \in \mathcal{T}^{I\!I}_n$ and $v\in T$, let $L_v$   be the set of vertices on the path from the root 1 to $v$, say $L_v=\{v_0:=1<v_1<\cdots <v_{m-1}<v_{m}:=v\}$.  Let $N_v$ denote the set of vertices in $L_v$ whose right child is not contained in  $L_v$. For $v_i \in N_v$, let $T^r_{v_i}$ denote the subtree rooted at this right child and let  $n_r(v_i)$ be the number of  vertices in $T^r_{v_i}$ with convention that $n_r(v_i)=0$ if $T^r_{v_i}=\emptyset$. 

We define the following statistics for $v\in T$: 
\begin{align*}
\Delta^{T}_{I\!I}(v)&=\sum_{v_i \in N_v} n_{r}(v_i)+|N_v|.
\end{align*}
When the context of the tree $T$ is clear, we simplify the notation to $\Delta_{I\!I_3}(v)$.

For example, consider an Andr\'e II tree $T$ given in Fig. \ref{labelAndre2}. By definition, we see that 
\[\Delta^{T}_{I\!I}(7)=12.\]

\begin{Proposition}\label{Andreperprop2} Let $T$ be an Andr\'e II tree in $\mathcal{T}^{I\!I}_n$ and let $u, v $ be two vertices in $T$  each having  at most one child. If $u$ lies on or to the left of the path from the root $1$ to $v$, then $\Delta_{I\!I}(u)> \Delta_{I\!I}(v)$; otherwise $\Delta_{I\!I}(u)<\Delta_{I\!I}(v)$. 
\end{Proposition}

\noindent{ \bf An insertion algorithm for Andr\'e II trees:} 
Let $T$ be an Andr\'e II tree on $[n]$.  Suppose $T$ has $m$ vertices with at most one child. We 
present an insertion algorithm to generate $m$ Andr\'e II trees on $[n+1]$ by inserting $n+1$ into $T$.

If $v$ is a leaf of $T$, then assign $n+1$ as the right child of $v$; If $v$ has one child of $T$, then assign $n+1$ as the left child of $v$. 
This yields an  Andr\'e II tree    $\widetilde{T}:=\phi^{I\!I}_v(T)$. 
It can be readily seen that the insertion algorithm is reversible.

From the construction of the insertion algorithm, it is not difficult to show  the following proposition. 

\begin{Proposition}\label{Andreperpropb2} Let $T$ be an Andr\'e II tree on $[n]$ and let $v$ be a vertex in $T$ with at most one child. Suppose that $\widetilde{T}=\phi^{I\!I}_v(T)$. Then 
\begin{equation}\label{insrelf}
 \inv (\widetilde{T})-\inv (T)=\Delta^{T}_{I\!I}(v).
 \end{equation} 
Moreover, for any vertex $u$ in $T$ with at most one child, we have the following results.
\begin{itemize}
\item[(a)] If $u$ lies on or to the left of the path from root 1 to $v$, then
\begin{equation*} 
\Delta^{\widetilde{T}}_{I\!I}(u)=\Delta^{T}_{I\!I}(u)+1.
\end{equation*}

\item[(b)] If $u$   lies to the right of the path from root 1 to $v$, then 
\begin{equation*} 
\Delta^{\widetilde{T}}_{I\!I}(u)=\Delta^{T}_{I\!I}(u).
\end{equation*}
 
\item[(c)]  If $v$ is a  leaf, then 
\begin{equation*}
\Delta^{\widetilde{T}}_{I\!I}(v)=\Delta^{T}_{I\!I}(v)+2.
\end{equation*}

\item[(d)]  If $v$ has only one child, then 
\begin{equation*}
\Delta^{\widetilde{T}}_{I\!I}(n+1)=\Delta^{T}_{I\!I}(v).
\end{equation*}

\end{itemize}
    
\end{Proposition}

We are ready to prove Theorem \ref{thm:q-AndreII} by using the grammatical labeling.  

Let $T$ be an Andr\'e II tree on $[n]$. For $1\leq i\leq n$, we label the vertex $i$ in $T$ as follows: 
\begin{itemize}
\item If $i$ is a leaf and $\Delta^{T}_{I\!I}(i)=k$, then label it by $x_k$; 
\item If $i$ has only one child and $\Delta^{T}_{I\!I}(i)=k$, then label it by $y_{k-1}$; 
\item If $i$ has two children, then it is left unlabeled.  
\end{itemize}
Fig. \ref{labelAndre2}  shows the grammatical labeling of the Andr\'e II tree $T$.

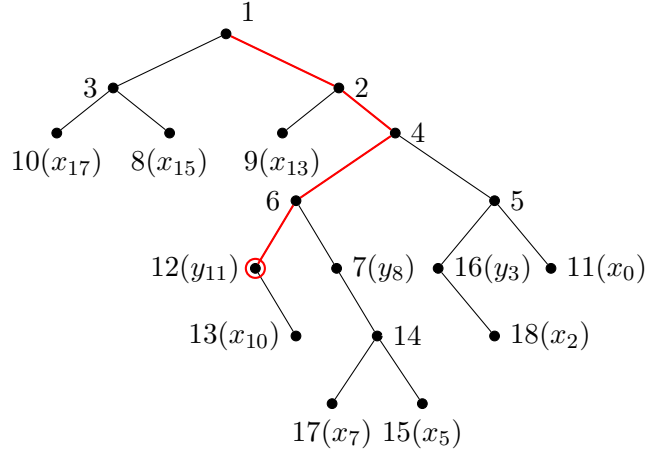
\begin{figure}
    \centering
  \begin{tikzpicture}
    [scale = 0.6, vertex/.style={shape=circle, draw, inner sep=1.3pt, fill=black},
    subtree/.style={shape=ellipse, draw,minimum width=1.5cm, minimum height=.5cm},
    every fit/.style={ellipse,draw,inner sep=-2pt},
    sibling distance=5cm,level distance=12mm,
    leaf/.style={label={[name=#1]below:$ $}},auto]
    
    \node[vertex, name=1, label=30:{$1$}]{}[grow=down]
    child {node [vertex, name=3, label=-180:{$3$}]{}[level distance=10mm, sibling distance = 25mm]
    child {node [vertex, label=-90:{$10(x_{17})$}]{}}
    child {node [vertex, label=-90:{$8(x_{15})$}]{}}
    edge from parent [solid] }
    child{node[vertex, name=2, label=0:{$2$}]{}[level distance=10mm, sibling distance = 25mm]
    child {node [vertex, label=-90:{$9(x_{13})$}]{}}
    child {node [vertex, name=4, label=0:{$4$}]{}
    [level distance=15mm, sibling distance = 44mm]
    child {node [vertex, name=6, label=180:{$6$}]{}
    [level distance=15mm, sibling distance = 18mm]
    child {node [vertex, name=12, label=180:{$12(y_{11})$}]{}
    child { {} edge from parent [white]}
    child {node [vertex,  label=-180:{$13(x_{10})$}]{}}}
    child {node [vertex, label=0:{$7(y_8)$}]{}
    [level distance=15mm, sibling distance = 18mm]
    child {node {} edge from parent [white]}
    child {node [vertex, label=0:{$14$}]{}
    [level distance=15mm, sibling distance = 20mm]
    child {node [vertex, label=-90:{$17(x_7)$}]{}}
    child {node [vertex, name = 15, label=-90:{$15(x_{5})$}]{}}}}}
    child {node [vertex, label=0:{$5$}]{}
    [level distance=15mm, sibling distance = 25mm]
    child {node [vertex, label=0:{$16(y_3)$}]{}
    child {{} edge from parent [white]}
    child {node [vertex, label=0:{$18(x_2)$}]{}
    }}
    child {node [vertex,  label=0:{$11(x_0)$}]{}}
    edge from parent [solid] }}};
    
    \draw[red, thick] (1) -- (2);
    \draw[red, thick] (2) -- (4);
    \draw[red, thick] (4) --  (6);
    \draw[red, thick] (12) --  (6);
    \draw[red, thick] (12) circle(6pt);
   
    \end{tikzpicture}
    \caption{The labeling of an Andr\'e II tree  $T$ on $[18]$.} \label{labelAndre2}
    \end{figure}

The weight $\omega$ of $T$ is defined to be the product of all labels of $T$ arranged in $\AIO$ order. For the example above, we see that the weight of $T$  is 
\[\omega(T)=x_0x_2y_3x_5x_7y_8x_{10}y_{11}x_{13}x_{15}x_{17}.\]

To prove Theorem  \ref{thm:q-AndreII}, we aim to show that   the following assertion: For $n\geq 0$,
\begin{equation}\label{equ-Andre2}
D^n (x_0)=\sum_{T \in \mathcal{T}^{I\!I}_{n+1}}   q^{\inv(T)}\omega(T),
\end{equation}
where $D$ is $q$-derivative associated with $q$-grammar \eqref{def:AndreII}. 

\begin{proof}[Proof of Theorem   \ref{thm:q-AndreII}] We proceed by induction on $n$. For $n=0$, the statement is evident. Assume that this statement holds for $n$, that is, the relation \eqref{equ-Andre2} is valid for $n-1$. To demonstrate that it also holds for $n$, it suffices to show, by \eqref{equ-Andre2}, that
\begin{equation}\label{gram-anreII-inva}
D\left(\sum_{T \in \mathcal{T}^{I\!I}_{n}}   q^{\inv(T)}\omega(T)\right)=\sum_{\widetilde{T} \in \mathcal{T}^{I\!I}_{n+1}}   q^{\inv(\widetilde{T})}\omega(\widetilde{T}).
\end{equation}

Suppose $T$ contains $m$ vertices with at most one child and define the weight of $T$  by  
\[\omega(T)= \omega_1(T)\omega_2(T) \cdots \omega_m(T).\] 

Applying the $q$-Leibniz formula, we have  
\begin{equation}\label{qLeibnitz-AndreII}
D\bigl(\omega(T)\bigr)
=\sum_{i=1}^{m} \AIO\Bigl(
\omega_1(T)\cdots\omega_{i-1}(T)\,
R(\omega_i(T))\,
\uparrow\!\bigl(\omega_{i+1}(T)\cdots\omega_{m}(T)\bigr)
\Bigr). 
\end{equation}
 Let $v_i$ be the vertex of $T$ labeled by $\omega_i(T)$, and  set $\widetilde{T}^{(i)}=\phi^{I\!I}_{v_i}(T)$.  To establish \eqref{gram-anreII-inva},  it is sufficient to show that 
\begin{align}\label{qLeibnitz-AndreIIaa}
& q^{\inv  (T)}\AIO\Bigl(
\omega_1(T)\cdots\omega_{i-1}(T)\,
R(\omega_i(T))\,
\uparrow\!\bigl(\omega_{i+1}(T)\cdots\omega_{m}(T)\bigr)
\Bigr)\nonumber \\
&\quad =q^{\inv  (\widetilde{T}^{(i)})}\omega(\widetilde{T}^{(i)}).
\end{align}
Since the letters of $\omega(T)$ are arranged in $\AIO$ order, Proposition \ref{Andreperprop2} implies that the  vertices labeled  $\omega_1(T), \ldots, \omega_{i-1}(T)$ lie to the right of the path from 1 to $v_i$ in $T$, while those labeled $\omega_{i+1}(T), \ldots, \omega_{m}(T)$ lie either on this path or to its left. We consider  two cases:

\begin{itemize}
\item[Case 1:]  If $\omega_i(T)=x_k$, then $v_i$ is a leaf of $T$ with $\Delta^{T}_{I\!I}(v_i)=k$. By Proposition \ref{Andreperpropb2}, \eqref{qLeibnitz-AndreIIaa} holds since $D(x_k)=q^k x_k y_{k+1}.$

\item[Case 2:]  If $\omega_i(T)=y_k$, then $v_i$ is a vertex of $T$ with exactly one child and $\Delta^{T}_{I\!I}(v_i)=k$.   In light of  Proposition \ref{Andreperpropb2},   we find that  \eqref{qLeibnitz-AndreIIaa} is satisfied since $D(y_k)=q^{k+1} x_{k+1}.$

\end{itemize}
 
Summing the results from these two cases shows that the assertion \eqref{gram-anreII-inva}  holds, and thus \eqref{equ-Andre2} is valid for $n$. This completes the proof of Theorem   \ref{thm:q-AndreII}. 
\end{proof}

We conclude this section with a grammatical derivation of the following identity for $E^{I\!I}_{n+1}(q;x,y)$, which specializes to the relation for $E_{n+1}(x,y)$ established by Chen and Fu  \cite{Chen-Fu-2017}. 
 
\begin{Theorem} \label{thm:AndrIIrec} Set $E^{I\!I}_{0}(q;x,y)=x$.  For $n\geq 1$, 
\begin{align} \label{eq:AndrIIrec}
&E^{I\!I}_{n+1}(q;x,y) \nonumber \\[5pt]
&\quad =yE^{I\!I}_{n}(q;x,y) +\sum_{k=0}^{n-2} q^{2(n-k-1)} {n-1 \brack k}_q  E^{I\!I}_{k+1}(q;x,y)E^{I\!I}_{n-k-1}(q;x,y).
\end{align}
\end{Theorem}

\begin{proof}
Since the map $\phi$ from Theorem \ref{thm:q-AndreII}  is a master-linear evaluation, and the grammar $G_{\rm{AndII}}$
   from \eqref{def:AndreII} satisfies \eqref{eq:rule}, Theorem \ref{thm:product} yields the multiplicative identity:
\begin{align*}
\phi\big(\mathrm{Gen}_q^{(G_{\rm{AndII}})}(x_0y_1;u)\big)
=\phi\big(\mathrm{Gen}_q^{(G_{\rm{AndII}})}(x_0;u)\big)\phi\big(\mathrm{Gen}_q^{(G_{\rm{AndII}})}(y_1;u)\big).
\end{align*}
Consequently,
\begin{align}\label{pf:AndreIIrecc}
\phi\big(D^{n}(x_0y_1)\big)
=\sum_{k=0}^n {n\brack k}_q \phi\big(D^k(x_0)\big) \phi\big(D^{n-k}(y_1)\big).
\end{align}
By Theorem \ref{thm:q-AndreII}, we have for all $k\geq 0$: 
\begin{equation}\label{pf:AndreIIreca}
\phi\big(D^k(x_0)\big) =E^{I\!I}_{k+1}(q;x,y).
\end{equation}
Similarly,   from the proof of Theorem \ref{thm:q-AndreI},   the order $\AIO$ in the grammar $G_{{\rm AndII}}$ is equivalent to the order $\KSO$. It implies that the grammar $G_{{\rm AndII}}$ is $q$-linear. Hence, by Proposition \ref{q_linear_uparrow-prop}, we derive that  for all $k\geq 1$:
\begin{equation}\label{pf:AndreIIrecb}
\phi\big(D^k(y_1)\big)
=q^2\phi\big(D^{k-1}(x_2)\big)=q^2\phi\big(D^{k-1}(\uparrow^2\!x_0)\big)
=q^{2k}E^{I\!I}_{k}(q;x,y).
\end{equation}
Substituting \eqref{pf:AndreIIreca} and \eqref{pf:AndreIIrecb} into \eqref{pf:AndreIIrecc} gives
\begin{align}\label{pf:AndreIIrece}
\phi\big(D^{n}(x_0y_1)\big)
&=yE^{I\!I}_{n+1}(q;x,y)+\sum_{k=0}^{n-1} q^{2(n-k)} {n\brack k}_q E^{I\!I}_{k+1}(q;x,y) E^{I\!I}_{n-k}(q;x,y).
\end{align}
Meanwhile, Theorem \ref{thm:q-AndreII} also implies that for $n\geq 1$,
\begin{equation}\label{pf:AndreIIrecd}
\phi\big(D^{n}(x_0)\big)
=\phi\big(D^{n-1}(x_0y_1)\big)
=E^{I\!I}_{n+1}(q;x,y).
\end{equation}
The recurrence \eqref{eq:AndrIIrec} follows from \eqref{pf:AndreIIrece} and \eqref{pf:AndreIIrecd}.
\end{proof}

\appendix

\renewcommand{\thesection}{\Roman{section}}

\renewcommand{\theequation}{\thesection-\arabic{equation}}

\counterwithin{equation}{section}

\section{Overview of Grammatical Calculus}\label{sec:survey}
Given a finite or infinite alphabet $X = \{x_1, x_2, x_3, \ldots\}$ and the commutative algebra $\mathbb{K}[[X]]$ of formal power series in the variables $x_i$, a {\it context-free grammar} is a collection of substitution rules that replace each variable in $X$ by a formal function or a formal power series over $X$. This notion was introduced by Chen \cite{Chen-1993}. Equivalently, a context-free grammar is an application
\[
G \colon X \longrightarrow \mathbb{K}[[X]],
\]
which can be written in the form
\[
G = \{x_1 \rightarrow G(x_1),\; x_2 \rightarrow G(x_2),\; x_3 \rightarrow G(x_3),\; \ldots\},
\]
where each $G(x_i)$ is a formal function or  a formal power series over $X$. As pointed out by Chen \cite{Chen-1993}, these substitution rules resemble context-free grammars in formal language theory, which motivates the terminology. Dumont \cite{Dumont-1996} later referred to such grammars as {\it William Chen's grammars}.

To a grammar $G$, one associates a {\it formal derivation} $D_G$, or simply $D$, which coincides with $G$ on $X$ and treats each substitution rule as a differential rule. More precisely, for a variable $x_i \in X$, if there is a production $x_i \rightarrow G(x_i)$ in the grammar $G$, then we define $D(x_i) = G(x_i)$; otherwise, we set $D(x_i)=0$, and such a variable is called a {\it constant} or a {\it terminal}. For two formal functions $f$ and $g$ in $\mathbb{K}[[X]]$, the formal derivation $D$ satisfies the following relations:
\begin{align}
D(f+g) &= D(f)+D(g), \label{gramma-add-rela} \\[5pt]
D(fg) &= D(f)g+fD(g). \label{gramma-mult-rela}
\end{align}
It follows that Leibniz's rule remains valid for the formal derivative $D$:
\begin{equation}\label{leibniz}
D^{n}(fg) = \sum_{k=0}^{n} \binom{n}{k} D^{k}(f) D^{n-k}(g).
\end{equation}

In \cite{Chen-1993}, Chen developed a systematic {\it grammatical calculus}. To this end, Chen associated the derivation $D$ induced by a grammar $G$ with an exponential generating function. For a formal function $f$ in $\mathbb{K}[[X]]$, define
\begin{equation}\label{defi:gfgram}
\mathrm{Gen}^{(G)}(f;u) = \sum_{n \ge 0} D_G^n(f)\frac{u^n}{n!}.
\end{equation}
This construction provides grammatical interpretations of addition, multiplication, and functional composition of formal functions. More precisely, by \eqref{gramma-add-rela} and \eqref{gramma-mult-rela}, one readily obtains
\begin{align}
\mathrm{Gen}^{(G)}(f+g;u) &= \mathrm{Gen}^{(G)}(f;u)+\mathrm{Gen}^{(G)}(g;u), \label{gramma-add} \\[5pt]
\mathrm{Gen}^{(G)}(fg;u) &= \mathrm{Gen}^{(G)}(f;u)\mathrm{Gen}^{(G)}(g;u), \label{gramma-multiple} \\[5pt]
\mathrm{Gen}^{(G)}(D(f);u) &= \frac{d}{du}\mathrm{Gen}^{(G)}(f;u), \\
\int \mathrm{Gen}^{(G)}(f;u)\,du &= \mathrm{Gen}^{(G)}\!\left(\int f\,dG;\,u\right),
\end{align}
where if $g=\int f\,dG$, then $D(g)=f$. 

Using this framework, Chen \cite{Chen-1993} obtained elegant proofs of Fa\`a di Bruno's formula and various identities involving Bell polynomials, Stirling numbers, and symmetric functions. In particular, the Lagrange inversion formula admits a concise grammatical interpretation, from which Cayley's formula for labeled trees emerges naturally.

A context-free grammar exhibits two complementary aspects: a combinatorial aspect and a computational aspect. Given a grammar $G$ and its associated derivation $D$, one studies the sequence of formal functions: 
\[
f,\; D(f),\; D^2(f),\; \ldots,\; D^n(f),\; \ldots
\]
On the combinatorial side, the goal is to interpret $D^n(f)$ as enumerating certain combinatorial objects associated with some statistics. On the computational side, the same grammar governs algebraic manipulations of formal functions, which constitutes the essence of grammatical calculus.

As an illustrative example, consider the grammar
\begin{equation}\label{gramEuler}
G = \{x \to xy,\; y \to xy\}. 
\end{equation}
This grammar is closely related to Eulerian polynomials that record the number of descents of permutations, and it was formally introduced by Dumont \cite{Dumont-1996}. 

The associated derivation $D$ is given by
\[
D(x)=xy,\qquad D(y)=xy.
\]
Straightforward calculations yield
\[
D^2(x)=xy^2+x^2y,
\]
and further
\[
D^3(x)=xy^3+4x^2y^2+x^3y.
\]

Let $\mathfrak{S}_n$ denote the set of permutations on $[n]$.  For $n\ge1$, the bivariate Eulerian polynomials are defined by
\[
A_n(x,y)=\sum_{\sigma \in \mathfrak{S}_n} x^{\mathrm{asc}(\sigma)} y^{\mathrm{des}(\sigma)},
\]
where $\mathrm{des}(\sigma)$ and $\mathrm{asc}(\sigma)$ denote the numbers of descents and ascents of $\sigma$, respectively, with the convention that $\sigma$ is padded by zeros at both ends, see the beginning of Section  \ref{sec:qgrammars} for their definitions.  Dumont \cite{Dumont-1996} showed that for $n\ge1$,
\begin{equation}\label{eq:graEuler}
A_n(x,y)=D^n(x),  
\end{equation}
with $A_0(x,y)=x$. 

Chen and Fu \cite{Chen-Fu-2017}  introduced the notion of a {\it grammatical labeling} for \eqref{eq:graEuler}, which exhibits how the substitution rules in context-free grammar arise in the construction of the combinatorial structures. This idea was already implicit in Chen’s original work on partitions \cite{Chen-1993}. Grammatical labelings thus serve as a bridge between combinatorial structures and grammars.

The exponential generating function of Eulerian polynomials is well known:
\begin{equation}\label{gf-Euler}
\mathrm{Gen}(x;u)=\sum_{n=0}^{\infty} A_n(x,y)\frac{u^n}{n!}
= \frac{x-y}{1-x^{-1}y\,e^{(x-y)u}}.
\end{equation}
The grammatical calculus provides a transparent derivation of \eqref{gf-Euler}. Since $D(x)=D(y)=xy$, we obtain
\[
D(x^{-1})=-x^{-1}y,
\qquad
D(x^{-1}y)=(x-y)x^{-1}y.
\]
Notably, $x-y$ behaves as a constant associated with  $D$, since $D(x-y)=0$. This observation leads to
\[
D^n(x^{-1}y)=(x-y)^n x^{-1}y,
\]
and consequently to the generating function of $x^{-1}$. Using the identity
\[
\mathrm{Gen}(x;u)\mathrm{Gen}(x^{-1};u)=1,
\]
one immediately recovers \eqref{gf-Euler}.

Dumont was a strong advocate of grammatical methods, discovering grammars for numerous combinatorial families: Roselle polynomials (on permutations), second-order Eulerian polynomials (on Stirling permutations ), 0-1-2 increasing trees (Andr\'e trees), increasing binary/plane trees, Ramanujan-Shor polynomials (on rooted trees), and Schett polynomials (related to Jacobi elliptic functions), see for example,  \cite{Dumont-1980, dumont1981jacobi, Dumont-1996, Dumont-Ramamonjisoa-1996}.

  Dumont \cite{Dumont-1996} termed the following polynomials defined on   permutations as Roselle polynomials: 
\[
F_n(x,y,z) = \sum_{\sigma \in \mathfrak{S}_n} x^{\exc(\sigma)} y^{\drop(\sigma)} z^{\fix(\sigma)},
\]
where $\exc(\sigma)$, $\drop(\sigma)$ and $\fix(\sigma)$ denote the number of excedances, the number of drops and the number of fixed points of $\sigma$, respectively. For the definitions of excedances, drops and fixed points, please refer to  the beginning of Section \ref{sec:qgrammars}.  

Dumont \cite{Dumont-1996}  showed  that the polynomials $F_n(x,y,z)$ can be generated by  the following grammar  
\begin{equation}\label{gram-Rosselle}
G = \{a \to az,\ z \to xy,\ x \to xy,\ y \to xy\}.
 \end{equation}

Let $D$ be the formal derivative associated with  \eqref{gram-Rosselle},  Dumont  showed that  
\begin{equation}\label{eq:gramRoss}
D^n(a) = aF_n(x,y,z).
\end{equation}
A grammatical labeling for \eqref{eq:gramRoss} was given by Chen and Fu \cite{Chen-Fu-2023a}, who also derived the generating function of Roselle polynomials via grammatical methods.

Let $\mathcal{T}_n$ be the set of 0-1-2 increasing trees on $\{0,1,\dots,n-1\}$. The Andr\'e polynomial $E_n(x,y)$ is defined by
\[
E_n(x,y) = \sum_{T \in \mathcal{T}_n} x^{l(T)} y^{u(T)},
\]
where $l(T)$ and $u(T)$ denote the number of leaves of $T$ and the number of vertices of $T$ with degree $1$, respectively. For further details, we refer to Section \ref{sec:qAndre}.

Dumont  \cite{Dumont-1996} introduced the formal grammar
\begin{equation}\label{gram-Andretree}
G=\{x \to xy, \quad y \to x\}.
\end{equation}
Let $D$ be the formal derivative associated with  \eqref{gram-Andretree}, Dumont \cite{Dumont-1996} showed that
 \begin{equation}\label{eq:grAnr}
D^n(y) = E_n(x,y).
\end{equation}
Chen and Fu \cite{Chen-Fu-2017} provided a grammatical labeling for \eqref{eq:grAnr} and deduced the generating function formula for Andr\'e polynomials using grammatical calculus.  This formula was first obtained by Foata and Sch\"utzenberger \cite{FoataSchutzenberger1973Nombres} via a differential equation. Notably, Foata and Han \cite{FoataHan2001Arbres}  developed a method to compute the generating function of $E(x,1)$  without solving a differential equation.

In recent years, grammatical approaches to variations and generalizations of Eulerian polynomials have been extensively studied; see, for example, \cite{Chen-Fu-2022, Chen-Fu-2023a, Chen-Fu-2023, Chen-Fu-2024,  Chen-Fu-Yan-2023, Chen-Hao-Yang-2021, Fu-2018, Fu-Li-2020, Han-Ma-2024, Ji-2025SCM, Ji-Lin-2024, Ma-2012, Ma-Fang-Mansour-Yeh-2022, Ma-Ma-Yeh-2019a, Ma-Ma-Yeh-Zhu-2018}.  

Grammatical calculus for tree-enumerative polynomials was first studied by Dumont and Ramamonjisoa \cite{Dumont-Ramamonjisoa-1996}.
They established a grammar for Ramanujan–Shor polynomials, which arise from the enumeration of rooted trees counted by improper edges and provide a refinement of Cayley’s formula for rooted trees on $n$ vertices; see \cite{Shor-1995, Zeng-1999}. Subsequent grammatical labelings and recursive grammars recover classic functional equations satisfied by Ramanujan–Shor polynomials, as developed in \cite{Chen-Fu-Wang-2025, Chen-Yang-2021}. Parallel advances concern Narayana and Motzkin polynomials arising from plane tree statistics, along with their stable multivariate generalizations \cite{Dong-Du-Ji-Zhang-2025, Yang-Zhang-2025}.

Beyond generating functions, grammars have also proved effective in establishing $\gamma$-positivity of combinatorial polynomials, constructing bijections, and establishing the stability of multivariate combinatorial polynomials. 
   
Ma, Ma, and Yeh \cite{Ma-Ma-Yeh-2019a} made the observation that a transformation of Dumont's grammar \eqref{gramEuler} yields the $\gamma$-positivity of Eulerian polynomials. Chen and Fu \cite{Chen-Fu-2022} further showed that the transformed grammar not only implies $\gamma$-positivity but also provides a combinatorial interpretation of the $\gamma$-coefficients in terms of 0-1-2 increasing plane trees. Subsequently, grammatical transformations have been successfully applied to establish the $\gamma$-positivity ($e$-positivity) of various modifications and generalizations of Eulerian polynomials and Ramanujan polynomials; see Chen, Fu, and Yan \cite{Chen-Fu-Yan-2023}, Ji \cite{Ji-2025SCM}, Ji and Lin \cite{Ji-Lin-2024}, Dong et al. \cite{Dong-Du-Ji-Zhang-2025}.

Grammar also facilitates the construction of combinatorial bijections. If two combinatorial structures admit the same grammar, then grammar can be leveraged to assist in establishing explicit bijections between them, see Chen and Fu \cite{Chen-Fu-2023, Chen-Fu-2024} and Chen, Fu and Wang \cite{Chen-Fu-Wang-2025}  for example. 

Grammars for the descent polynomials of Legendre-Stirling permutations and marked Stirling permutations were constructed by Chen, Hao, and Yang \cite{Chen-Hao-Yang-2021}, yielding stable multivariate generalizations,   where stability is equivalent to real-rootedness in the univariate case.

Overall, grammatical methods allow one to derive generating functions and identities without explicit reliance on recurrence relations or differential equations. Beyond enumeration, grammatical approaches have also been applied to construct bijections, to establish $\gamma$-positivity of combinatorial polynomials, and to prove the stability of multivariate combinatorial polynomials.

\section{\texorpdfstring{Formulas for  $q$-derivatives and their $q$-exponential generating functions }{Formulas for  q-derivatives and their q-exponential generating functions }}\label{sec:qderivativeform}

We collect some important formulas for $q$-derivatives and their $q$-exponential generating functions in this section.

\begin{Proposition}
Let $D$ be the $q$-derivative associated with a $q$-grammar. For $f,g\in\mathbb{E}$ and $c \in \mathbb{K}[q]$, we have 
\begin{align}
D(c)&=0,  \\[5pt]
D(cf)&=cD(f),  \\[5pt]
D(f+g)&=D(f)+D(g), \\[5pt]
D\left(f^{-1}\right)&=-f^{-1}{D(f)}\uparrow (f^{-1}).
\end{align}
\end{Proposition}

\begin{Proposition}
Let $D$ be the $q$-derivative associated with a $q$-linear grammar. For $f,g\in\mathbb{E}$, we have 
\begin{align}
D(fg)&=D(f)\uparrow g+fD(g),\\[5pt]
D^n(\uparrow^m f)&=q^{nm} \uparrow^{m} \left(D^n(f)\right),\\[5pt]
D^n(fg)&=\sum_{k=0}^n {n \brack k}_q D^k(f) \uparrow^k \left(D^{(n-k)}(g)\right).
\end{align}
\end{Proposition}

Let $D$ be the $q$-derivative associated with a $q$-grammar $G$. For $f\in\mathbb{E}$, the $q$-exponential generating function of $D^n(f)$ is defined by 
\begin{equation*}
    {\rm Gen}_q^{(G)}(f;u)=\sum_{n\geq 0}D^n(f)\frac{u^n}{(q;q)_n}.
\end{equation*}

\begin{Proposition} We have 
\begin{align}
{\rm Gen}_q^{(G)}(f+g;u) &= {\rm Gen}_q^{(G)}(f;u)+{\rm Gen}_q^{(G)}(g;u),\\[5pt]
D_q{\rm Gen}_q^{(G)}(f;u) &=
{\rm Gen}_q^{(G)}(D(f);u),
\end{align}
where $D_q$ is a real $q$-derivative defined in \eqref{defi:qderiv}. 
\end{Proposition}

\begin{Proposition} Let $G$ be a $q$-linear grammar and let $D$ be the $q$-derivative associated with $G$, then for $f,g\in\mathbb{E}$, 
\begin{align}
{\rm Gen}_q^{(G)}(fg;u)&= \sum_{k\geq 0} D^k(f) \frac{u^k}{(q; q)_k}  \uparrow^k
  {\rm Gen}_q^{(G)}(g;u),\\[5pt]
  \Gen_q^{(G)}(\uparrow^m f;u) &= \uparrow^m \Gen_q^{(G)}(f;uq^m).
\end{align}
If $\phi$ is a master-linear evaluation, then
 \begin{align}
\phi\left({\rm Gen}_q^{(G)}(fg;u)\right) &= \phi\left({\rm Gen}_q^{(G)}(f;u)\right)\cdot \phi\left({\rm Gen}_q^{(G)}(g;u)\right).  
\end{align}
\end{Proposition}

\begin{Theorem} Let $G=(S,R,\rho)$ be a $q$-grammar, where for each   variable $s_i \in \mathbb{S}$ and any $i\geq 0$, we have $
      R(s_{i+1})= R(\uparrow s_{i})=q\uparrow R(s_{i}). $ 
Let $\phi$ be a master-linear evaluation, that is,  for any two variables $s_i, s_j \in \mathbb{S}$, $\phi (s_i)= \phi (s_j)$. We have 
\begin{align}
\phi\left({\rm Gen}_q^{(G)}(fg;u)\right) &= \phi\left({\rm Gen}_q^{(G)}(f;u)\right)\cdot \phi\left({\rm Gen}_q^{(G)}(g;u)\right). 
\end{align}
\end{Theorem}

\section{\texorpdfstring{Numbers of terms of the $q$-grammars}{Numbers of terms of the q-grammars}}\label{sec:nbterms}

In this section, a \emph{term} means a distinct word appearing with
nonzero coefficient in the corresponding expression. Coefficients are ignored when terms are counted. We write $\Omega(D^n(f))$ as the number of distinct terms in $D^n(f)$. The Fibonacci numbers are denoted by $F_n$, with $F_1=F_2=1$, and the Motzkin numbers are denoted by $M_n$, with $M_0=1$.

\subsection*{The grammar $G_{\tan}$}
For the $q$-grammar $G_{\tan}$, the initial values of the numbers of terms $\Omega(D^n(x_0))$ are  $(2, 3, 9, 20, 38, 65, 101, 150, 210, 287, 377, \ldots )$.
The exact formula is (see Proposition \ref{lem: number terms03})
$$
\Omega(D^n(x_0))= 
\begin{cases}
2k^3 + 5k^2 +2 , & \text{if } n = 2k+1,\\[4pt]
(k+2)(2k^2-2k+1),        & \text{if } n = 2k.
\end{cases}
$$

\subsection*{The grammar $G_{\tan'}$}

For the $q$-grammar $G_{\tan'}$, the initial values of the numbers of terms $\Omega(D^n(x_0))$ are  $(2, 3, 5, 8, 13, 21, 34, 55, 89, 144, 233, \ldots )$.
The exact formula is $ \Omega(D^n(x_0)) = F_{n+2}$, where $F_n$ is the $n$-th Fibonacci number.

More precisely, we derive the following two results:
\begin{Proposition}\label{prop:tan'1}
Consider the $q$-grammar $G_{\tan'}$.
For $n\ge 1$, a word (including the empty word) appears as a term in $D^n(x_0)$ (ignoring coefficients) if and only if it satisfies:
\begin{enumerate}
\item The word is of the form $x_{i_1}x_{i_2}\ldots x_{i_{k}}$ where the indices are strictly increasing (i.e., $0\leq i_1<i_2<\cdots<i_k$).
\item The first index $i_1$ is even.
\item Indices alternate in parity: even, odd, even, odd, \ldots
\item The word contains $k$ variables, where $k \equiv n+1 \pmod{2}$.
\end{enumerate}
\end{Proposition}

\begin{proof}
$\Rightarrow:$ We prove by induction on $n$. 

For $n=1$, we have
\[
    D(x_0)=1+x_0x_1.
\]
Thus the two terms are the empty word and $x_0x_1$, and both satisfy the
stated conditions.

Assume the assertion holds for $D^n(x_0)$. Take any term $w=x_{i_1}\cdots x_{i_k}$ in $D^n(x_0)$ (with $k\ge1$; constant terms give no contribution to $D^{n+1}$ because $D(1)=0$). Apply $D$ to $w$ and examine the two contributions for each position $j$ ($1\le j\le k$).

\paragraph{Constant contribution:}
From $R(x_{i_j})$ take the constant term $q^{i_j}$. The resulting word is
\[
w' = x_{i_1}\cdots x_{i_{j-1}}\; \uparrow\!\bigl(x_{i_{j+1}}\cdots x_{i_k}\bigr)
= x_{i_1}\cdots x_{i_{j-1}}\; x_{i_{j+1}+1}\cdots x_{i_k+1}.
\]
Its length is $k-1$. The indices of the prefix stay unchanged; the suffix indices increase by $1$, which flips their parity. Because $w$ alternates and starts with an even index, $w'$ also alternates and starts with an even index. The parity condition for $D^{n+1}$ follows from $(k-1)\equiv (n+1)-1 = n \equiv (n+1)+1 \pmod{2}$.

\paragraph{Product contribution:}
From $R(x_{i_j})$ take $q^{i_j}x_{i_j}x_{i_j+1}$. The resulting word is
\begin{align*}
w'' &= x_{i_1}\cdots x_{i_{j-1}}\; x_{i_j}x_{i_j+1}\; \uparrow\!\bigl(x_{i_{j+1}}\cdots x_{i_k}\bigr)
\\&= x_{i_1}\cdots x_{i_{j-1}}\; x_{i_j}x_{i_j+1}\; x_{i_{j+1}+1}\cdots x_{i_k+1}.
\end{align*}
Its length is $k+1$. The two inserted variables have indices $i_j$ and $i_j+1$; because $i_j$ has parity $p$, $i_j+1$ has parity $1-p$. The suffix is shifted by $1$ as before. Using the alternating property of $w$ one verifies that $w''$ also alternates and starts with an even index. The parity condition gives $k+1\equiv (n+1)+1 \pmod{2}$, which is required for $D^{n+1}$.

Thus every term arising from $D^{n+1}(x_0)=D(D^n(x_0))$ satisfies the four properties. This completes the induction. 

$\Leftarrow:$ On the other hand, similarly, each term in $D^n(x_0)$ that satisfies the four conditions in this proposition can be obtained by applying the operator $D$ on some term in $D^{n-1}(x_0)$ that satisfies the four conditions, which completes the proof.
\end{proof}

\begin{Proposition}\label{prop:tan'2}
Consider the $q$-grammar $G_{\tan'}$.
The number of distinct terms (including the empty term) in $D^n(x_0)$ is the Fibonacci number $F_{n+2}$, where $F_1=F_2=1$.
\end{Proposition}

\begin{proof}
Using Proposition~\ref{prop:tan'1}, each term $x_{i_1}\cdots x_{i_k}$ in $D^n(x_0)$ corresponds to a strictly increasing sequence of indices with alternating parity, starting with an even index. Define the gap sequence
\[
g_1 = i_1 - (-1),\quad g_2 = i_2-i_1,\quad \ldots,\quad g_k = i_k-i_{k-1},\quad g_{k+1} = (n+1)-i_k.
\]
Because $i_1$ is even, $i_1+1$ is odd; each difference $i_{t+1}-i_t$ is odd (alternating parity); and $i_k$ has parity opposite to $n+1$ (since $k\equiv n+1\pmod2$), hence $(n+1)-i_k$ is odd. Therefore every $g_t$ is a positive odd integer. Their sum telescopes:
\[
g_1+g_2+\cdots+g_{k+1} = (i_1+1)+(i_2-i_1)+\cdots+(n+1-i_k)=n+2.
\]
Thus each term yields a composition of $n+2$ into odd positive parts. Conversely, given such a composition $(g_1,\dots,g_{m})$ ($m\ge1$) of $n+2$, set $k=m-1$ and $i_t = -1+\sum_{s=1}^t g_s$. Then $i_1 = g_1-1$ is even because $g_1$ is odd; each $g_t$ odd ensures alternating parity; and $i_k = n+1-g_{m} \le n$ because $g_{m}\ge1$. Hence the term $x_{i_1}\cdots x_{i_k}$ belongs to $D^n(x_0)$. The empty term corresponds to the composition of $n+2$ into a single part $g_1=n+2$, which is odd precisely when $n+2$ is odd, i.e., when $n$ is odd.

The number of compositions of a positive integer $N$ into odd parts is the Fibonacci number $F_N$ (with $F_1=F_2=1$). Indeed, let $C_N$ be that number. Then $C_1=1$, $C_2=1$, and for $N\ge3$, a composition either starts with $1$ (leaving $N-1$) or starts with an odd number $\ge3$, then subtracting $2$ gives a bijection with compositions of $N-2$. Hence $C_N=C_{N-1}+C_{N-2}$ with the same initial values, thus $C_N=F_N$. Taking $N=n+2$ gives $\Omega(D^n(x_0))=F_{n+2}$. 
\end{proof}

\subsection*{The grammar $G_{\sec}$}

For the $q$-grammar $G_{\sec}$, the initial values of the numbers of terms $\Omega(D^n(y_0))$ are  $(1, 3, 8, 19, 36, 63, 98, 147, 206, 283, 372, \ldots )$.
The exact formula is  
$$
\Omega(D^n(y_0))= 
\begin{cases}
1 , & \text{if } n = 1,\\[4pt]
2k^3 + 5k^2 - k + 2 , & \text{if } n = 2k+1, \quad (k\geq 1)\\[4pt]
2k^3 + 2k^2 - 4k + 3,        & \text{if } n = 2k. \quad (k\geq 1)
\end{cases}
$$

More precisely, similar to Proposition~\ref{lem: number terms02} and  Lemma \ref{lem: number terms03}, we can derive the following two results (the proofs are similar and thus ignored here):

\begin{Proposition}\label{prop:sec1}
Consider the $q$-grammar $G_{\sec}$.
For $n\ge 1$, a word appears as a term in $D^n(y_0)$ (ignoring coefficients) if and only if it has one of the following forms:
\begin{enumerate}
    \item[(i)] $y_j\,x_{j-1}^{\,n}$ with $2\le j\le n$;
    \item[(ii)] $x_1^{\,n-1}y_1x_0$ or $x_1^{\,n-2}y_1$ (for $n\ge2$; for $n=1$ the only term is $y_1x_0$);
    \item[(iii)] $x_{j+1}^{\,a}\,y_{j+1}\,x_j^{\,b}$ where $1\le j\le n-2$, $0\le a,b\le n-1$, $a+b\le n$, and $a+b+n$ is even.
\end{enumerate}
\end{Proposition}

\begin{Proposition}\label{prop:sec2}
Consider the $q$-grammar $G_{\sec}$.
The number of terms in $D^n(y_0)$ is
$$
\Omega(D^n(y_0))= 
\begin{cases}
1 , & \text{if } n = 1,\\[4pt]
2k^3 + 5k^2 - k + 2 , & \text{if } n = 2k+1, \quad (k\geq 1)\\[4pt]
2k^3 + 2k^2 - 4k + 3,        & \text{if } n = 2k. \quad (k\geq 1)
\end{cases}
$$
\end{Proposition}

\subsection*{The grammar $G_{\sec'}$}

For the $q$-grammar $G_{\sec'}$, the initial values of the numbers of terms $\Omega(D^n(y_0))$ are  $(1, 2, 3, 5, 8, 13, 21, 34, 55, 89, 144, \ldots )$.
The exact formula is $ \Omega(D^n(y_0)) = F_{n+1}$.

More precisely, similar to Propositions~\ref{prop:tan'1} and \ref{prop:tan'2}, we can derive the following two results (the proofs are similar and thus ignored here):
\begin{Proposition}\label{prop:sec'1}
Consider the $q$-grammar $G_{\sec'}$.
For $n\ge 1$, a word is a term in $D^n(y_0)$ if and only if it satisfies:
\begin{enumerate}
\item The word is of the form $x_{i_1}x_{i_2}\ldots x_{i_{k-1}}y_n$ where the indices are strictly increasing (i.e., $0\leq i_1<i_2<\cdots<i_{k-1}<n$).
\item The first index $i_1$ is even.
\item Indices alternate in parity: even, odd, even, odd, \ldots
\item The word contains $k\geq 1$ variables (at least one $y_n$), where $k \equiv n+1 \pmod{2}$.
\end{enumerate}
\end{Proposition}

\begin{Proposition}\label{prop:sec'2}
Consider the $q$-grammar $G_{\sec'}$.
The number of distinct terms in $D^n(y_0)$ is the Fibonacci number $F_{n+1}$, where $F_1=F_2=1$.
\end{Proposition}

\subsection*{The grammar $G_{\operatorname{Sec}}$}

For the $q$-grammar $G_{\Sec}$, the initial values of the numbers of terms $\Omega(D^n(y_0))$ are  $(1, 3, 8, 23, 48, 86, 139, 210, 301, 415, 554, \ldots )$. Computations suggest the following closed formula.

\begin{Conjecture}\label{conj:GSec-count}
For the $q$-grammar $G_{\operatorname{Sec}}$, the number of distinct terms
in $D^n(y_0)$ is
\[
\Omega(D^n(y_0))=
\begin{cases}
1, & \text{if } n=1,\\[4pt]
3, & \text{if } n=2,\\[4pt]
\dfrac{1}{6}\bigl(20k^3+33k^2+k-6\bigr),
& \text{if } n=2k+1,\quad k\ge 1,\\[8pt]
\dfrac{1}{6}(20k-17)(k+1)k,
& \text{if } n=2k,\quad k\ge 2.
\end{cases}
\]
\end{Conjecture}

\subsection*{The grammar $G_{\operatorname{Sec}^{\prime}}$}

For the $q$-grammar $G_{\operatorname{Sec}^{\prime}}$, the initial values
of $\Omega(D^n(y_0))$, for $n\ge 1$, are
\[
    1,\,3,\,8,\,21,\,53,\,132,\,325,\,795,\,1936,\,4701,\,11393,\ldots .
\]
At present, we record these values as computational data. Finding a closed formula or a structural characterization for this sequence remains an open
problem.

\subsection*{The grammar $G_{\mathrm{maj}}$}

For the $q$-grammar $G_{\mathrm{maj}}$, the initial values of
$\Omega(D^n(x_0))$, for $n\ge 1$, are
\[
    1,\,2,\,6,\,20,\,73,\,283,\,1147,\,4814,\,20774,\ldots .
\]
We record this sequence as computational evidence. A closed formula for
$\Omega(D^n(x_0))$ is not currently known.

\subsection*{The grammar $G_{\mathrm{inv}}$}

For the $q$-grammar $G_{\inv}$, the initial values of the numbers of terms $\Omega(D^n(x_0))$ are  $(1, 2, 4, 8, 16, 32, 64, 128, 256, 512, \ldots )$.
The exact formula is $ \Omega(D^n(x_0)) = 2^{n-1}$.

More precisely, the following results are easy to derive:

\begin{Proposition}\label{prop:inv1}
Consider the $q$-grammar $G_{\inv}$.
For $n\ge 1$, a word appears as a term in $D^n(x_0)$ (ignoring coefficients) if and only if it is of the form $y_0  k_{1}k_{2}\ldots k_{n-1}x_{n}$ where $k_{i}=x_{i}$ or $y_{i}$.
\end{Proposition}

\begin{Proposition}\label{prop:inv2}
Consider the $q$-grammar $G_{\inv}$.
The number of distinct terms in $D^n(x_0)$ is $2^{n-1}$.
\end{Proposition}

\subsection*{The grammar $G_{\mathrm{cyc}}$}

For the $q$-grammar $G_{\cyc}$, the initial values of the numbers of terms $\Omega(D^n(e_0))$ are  $(1, 2, 5, 12, 28, 65, 151, 351, 816, 1897, \ldots )$.
The exact formula is
$$
\Omega(D^n(e_0))= \sum_{k=0}^{\lfloor (n+1)/2\rfloor} \binom{n+k}{3k}.
$$

More precisely, we derive the following results:

\begin{Proposition}\label{prop:char}
Consider the $q$-grammar $G_{\cyc}$.
For every $n\ge 1$, a word $w$ appears in $D^n (e_0)$ (ignoring coefficients) if and only if it is of the form $e_0  k_{1}k_{2}\ldots k_{n}$ that satisfies:
\begin{enumerate}
    \item[(c0)] The letters have indices $0,1,\dots,n$ each exactly once, in strictly increasing order, and $k_{i}=x_{i}$, $y_{i}$ or $z_i$ for $1\leq i\leq n$; 
    \item[(c1)] The pattern $z_i x_{i+1}$ does not occur as consecutive letters;
    \item[(c2)] The pattern $y_i z_{i+1}$ does not occur as consecutive letters;
    \item[(c3)] The second letter is not $x_1$ (i.e., $x_1$ does not appear in position $2$);
    \item[(c4)] The last letter is not $y_n$.
\end{enumerate}
\end{Proposition}

\begin{proof}
The proof is by induction on $n$.

For $n=1$, the assertion follows directly from the defining rule
$R(e_0)=\beta e_0z_1$.

Assume the assertion holds for $n-1$, with $n\ge 2$. Let
$w$ be a term in $D^n(e_0)$. Then $w$ is obtained by applying $D$ to a
term $v$ of $D^{n-1}(e_0)$ at one position. If the differentiated letter is
$e_0$, the local replacement is $e_0\mapsto e_0z_1$. If the differentiated
letter is $x_i$, $y_i$, or $z_i$, the local replacement is
\[
    x_i,y_i,z_i \mapsto y_ix_{i+1}.
\]
All letters to the right of the differentiated position are shifted by
$\uparrow$. Therefore the resulting word has indices $0,1,\ldots,n$ in
strictly increasing order. The newly created adjacent pair is either
$e_0z_1$ or $y_ix_{i+1}$, neither of which violates (c1) or (c2). Conditions
(c3) and (c4) follow from the induction hypothesis and from the fact that
the only possible final letters are $x_n$ and $z_n$.

Conversely, if a word satisfies (c0)--(c4), then one reverses the last local
operation: either delete a terminal $z_n$ produced from $e_0$, or contract a
terminal admissible pair $y_ix_{i+1}$ to one of $x_i,y_i,z_i$ and shift the
suffix indices down by one. The forbidden-pattern conditions ensure that
the resulting word again satisfies (c0)--(c4) with $n$ replaced by $n-1$.
The induction completes the proof.
\end{proof}

Let
\[
    a_n=\Omega(D^n(e_0)).
\]
For $n\ge 1$, define
\[
b_n = \#\{w\in D^n (e_0) : w \text{ ends with } x_n\},
\]
\[
c_n = \#\{w\in D^n (e_0) : w \text{ ends with } z_n\}.
\]
By Condition (c4), every admissible word ends with either $x_n$ or $z_n$,
and hence $a_n = b_n + c_n$. Furthermore, we derive the following results.

\begin{Lemma}\label{lem:rec1}
For $n\ge 1$,
\[
c_n = a_{n-1},\qquad 
b_n = b_{n-1} + \sum_{i=0}^{n-2} a_i,
\]
with $a_0=1$, $b_0=0$, $c_0=0$.
\end{Lemma}

\begin{proof}
Deleting the final $z_n$ gives a bijection between admissible words of
order $n$ ending in $z_n$ and admissible words of order $n-1$. Hence
$c_n=a_{n-1}$.

Now let $w$ be an admissible word of order $n$ ending in $x_n$. Write
\[
    w=u\,k_{n-1}x_n,
\]
where $k_{n-1}$ is the letter of index $n-1$. By condition (c1),
$k_{n-1}\ne z_{n-1}$, so $k_{n-1}\in\{x_{n-1},y_{n-1}\}$.

If $k_{n-1}=x_{n-1}$, deleting the final $x_n$ gives an admissible word of
order $n-1$ ending in $x_{n-1}$. This contributes $b_{n-1}$ words.

If $k_{n-1}=y_{n-1}$, then the terminal segment has the form
\[
    y_i y_{i+1}\cdots y_{n-1}x_n
\]
for a unique $1\le i\le n-1$, and the prefix preceding this segment is an
admissible word of order $i-1$. Hence this case contributes
$\sum_{i=1}^{n-1}a_{i-1}=\sum_{i=0}^{n-2}a_i$ words. Therefore
\[
    b_n=b_{n-1}+\sum_{i=0}^{n-2}a_i.
\]
\end{proof}

From Lemma~\ref{lem:rec1}, we derive a recurrence for $a_n$.

\begin{Lemma}\label{lem:rec2}
For $n\ge 4$,
\[
a_n = 3a_{n-1} - 2a_{n-2} + a_{n-3}.
\]
Furthermore,
\[
\Omega(D^n(e_0))=a_n=\sum_{k=0}^{\lfloor (n+1)/2\rfloor} \binom{n+k}{3k}.
\]
\end{Lemma}

\begin{proof}
We have $a_n = b_n + c_n$ and $c_n = a_{n-1}$, thus $b_n = a_n - a_{n-1}$. Substitute into the recurrence for $b_n$:
\[
a_n - a_{n-1} = b_{n-1} + \sum_{i=0}^{n-2} a_i = (a_{n-1} - a_{n-2}) + \sum_{i=0}^{n-2} a_i.
\]
Thus
\[
a_n = 2a_{n-1} - a_{n-2} + \sum_{i=0}^{n-2} a_i.
\]
Replace $n$ by $n-1$:
\[
a_{n-1} = 2a_{n-2} - a_{n-3} + \sum_{i=0}^{n-3} a_i.
\]
Subtract the second equation from the first:
\[
a_n - a_{n-1} = (2a_{n-1} - a_{n-2}) - (2a_{n-2} - a_{n-3}) + a_{n-2} = 2a_{n-1} - 2a_{n-2} + a_{n-3}.
\]
Hence
\[
a_n = 3a_{n-1} - 2a_{n-2} + a_{n-3}.
\]
It remains to identify the closed form. Let
\[
    A(t)=\sum_{n\ge 0}a_nt^n.
\]
The recurrence, together with the initial values $a_0=1$, $a_1=1$,
$a_2=2$, and $a_3=5$, gives
\[
    A(t)=\frac{(1-t)^2}{1-3t+2t^2-t^3}.
\]
On the other hand,
\[
\begin{aligned}
\sum_{n\ge 0}
\left(
\sum_{k=0}^{\lfloor (n+1)/2\rfloor}\binom{n+k}{3k}
\right)t^n
&=
\sum_{k\ge 0}\sum_{m\ge 0}
\binom{m+3k}{3k}t^{m+2k}  \\
&=
\sum_{k\ge 0}\frac{t^{2k}}{(1-t)^{3k+1}}  \\
&=
\frac{(1-t)^2}{1-3t+2t^2-t^3}.
\end{aligned}
\]
Thus the two sequences have the same generating function, proving the
formula.
\end{proof}

\subsection*{The grammar $G_{\rm AndI}$}

For the $q$-grammar $G_{{\rm AndI}}$, the initial values of the numbers of terms $\Omega(D^n(x_0))$ are  $(1, 2, 4, 9, 21, 51, 127, 323, 835, 2188, 5798, \ldots )$.
The exact formula is $ \Omega(D^n(x_0)) = M_{n}$, where $M_n$ is the $n$-th Motzkin number (Proposition \ref{prop:Motzkin}).

\subsection*{The grammar $G_{{\rm AndII}}$}

For the $q$-grammar $G_{\rm AndII}$, the initial values of the numbers of terms $\Omega(D^n(x_0))$ are  $(1, 2, 3, 5, 8, 13, 21, 34, 55, 89, 144, \ldots )$.
The exact formula is $ \Omega(D^n(x_0)) = F_{n+1}$.

More precisely, the following results are easy to derive:

\begin{Proposition}\label{prop:andii1}
Consider the $q$-grammar $G_{\rm AndII}$.
For $n\ge 1$, a word appears as a term in $D^n(x_0)$ (ignoring coefficients) if and only if it is of the form $x_0  k_{i_1}k_{i_2}\ldots k_{i_j}$ where 
\begin{itemize}
\item $0<i_1<i_2<\cdots<i_j$;
    \item 
$k_{i}=x_{i}$ or $y_{i}$; 
\item if $k_{i}=x_{i}$, then $x_{i-1}$ or $y_{i-1}$ doesn't occur in the word (i.e., the index $i-1$ doesn't occur).
\end{itemize}
\end{Proposition}


\begin{Proposition}\label{prop:andii2}
Consider the $q$-grammar $G_{\rm AndII}$.
The number of distinct terms in $D^n(x_0)$ is $\Omega(D^n(x_0))=F_{n+1}$.
\end{Proposition}
\begin{proof}
Let $a_n=\Omega(D^n(x_0))$. By Proposition~\ref{prop:andii1}, the
first letter after $x_0$ is either $y_1$ or $x_2$. If it is $y_1$, deleting
$y_1$ gives a term of $D^{n-1}(x_0)$. If it is $x_2$, deleting $x_2$ gives a
term of $D^{n-2}(x_0)$. These two cases are disjoint and exhaustive. Hence
\[
    a_n=a_{n-1}+a_{n-2}.
\]
Together with $a_1=1$ and $a_2=2$, this gives
$a_n=F_{n+1}$.
\end{proof}

\vskip 0.2cm
\noindent{\bf Acknowledgment.} This work was initiated during the 2025 Workshop on Enumerative Combinatorics at Shandong University, Qingdao Campus. The authors would like to thank Zhicong Lin, the workshop organizer,  for providing a valuable platform for academic exchange.  This work was supported by the National Natural Science Foundation of China.




\bibliographystyle{plain}

\end{document}